\documentclass[leqno]{amsart}

\let\oldtocsection=\tocsection
\let\oldtocsubsection=\tocsubsection
\let\oldtocsubsubsection=\tocsubsubsection

\renewcommand{\tocsection}[2]{\hspace{0em}\oldtocsection{#1}{#2}}
\renewcommand{\tocsubsection}[2]{\hspace{1em}\oldtocsubsection{#1}{#2}}
\renewcommand{\tocsubsubsection}[2]{\hspace{2em}\oldtocsubsubsection{#1}{#2}}

\usepackage{amsmath,amsthm,amssymb,latexsym,graphicx, mathtools, enumerate,bm,cite,todonotes,mathrsfs,esint,amsfonts,amscd,amsbsy, upref,epsf, euscript}

\usepackage[margin=1in]{geometry}

\usepackage{hyperref}
\hypersetup{
colorlinks=true,
allcolors=blue,
}
\usepackage{bbold}
\usepackage[numbers]{natbib}
\newcommand{\bal}{\begin{aligned}}
\newcommand{\eal}{\end{aligned}}
\newcommand{\bsplit}{\begin{split}}
\newcommand{\esplit}{\begin{split}}

\newcommand{\N}{{\mathbb N}}

\newcommand{\R}{{\mathbb R}}

\newcommand{\be}{\begin{equation}}
\newcommand{\ee}{\end{equation}}
\newcommand{\Pt}{\mathcal{P}_2}
\newcommand{\wtilde}{\widetilde}

\newcommand{\ba}{\begin{aligned}}
\newcommand{\ea}{\end{aligned}}

\def\ds{\displaystyle}
\newcommand{\dsum}{\displaystyle \sum}
\newcommand{\del}{\delta}

\newcommand{\pt}{\partial}
\newcommand{\ol}{\overline}

\newcommand{\eps}{\varepsilon}

\newcommand{\oo}{\infty}

\newcommand{\ep}{\epsilon}
\newcommand{\what}{\widehat}

\newcommand{\beq}{\begin{equation}}
\newcommand{\eeq}{\end{equation}}

\newcommand{\la}{\lambda}
\newcommand{\bet}{\beta}
\newcommand{\al}{\alpha}

\newcommand{\dint}{\displaystyle \int}

\newcommand{\mcl}{\mathcal}

\newtheorem{lemma}{Lemma}
\newtheorem{proposition}{Proposition}
\newtheorem{theorem}{Theorem}

\newtheorem{remark}{Remark}

\theoremstyle{definition}
\newtheorem{definition}{Definition}
\newcommand{\vs}{\vskip.05in}

\numberwithin{equation}{section}
\numberwithin{lemma}{section}
\numberwithin{proposition}{section}
\numberwithin{theorem}{section}
\numberwithin{corollary}{section}
\numberwithin{definition}{section}
\numberwithin{remark}{section}

\begin{document}
\title{Optimal control of the Dyson equation}
\author{Charles Bertucci$^1$}
\address{$^1$: CEREMADE, UMR7534, Universit\'e Paris Dauphine-PSL, Place du Mar\'echal de Lattre de Tassigny, 75016 Paris, France}
\author{Pierre-Louis Lions$^{1,2}$}
\address{$^2$: Coll\`ege de France, 3 rue d'Ulm, 75005 Paris, France}
\author{Panagiotis E. Souganidis$^3$}
\address{$^3$: The University of Chicago, Eckhart Hall, 5734 S University Ave,
Chicago, IL 60637, USA}
\date{\today}

\maketitle

%\tableofcontents
\begin{abstract}
We provide a complete study of the optimal control problem associated to the so-called Dyson equation, which describes the mean field evolution of the spectrum of large random matrices. We treat both the case of regular  and singular terminal data. We construct  admissible controls and prove the existence of optimal ones. We then proceed to give a thorough mathematical analysis of the associated Hamilton-Jacobi-Bellman equation by means of the theory of viscosity solutions in the space of measures. In particular, we characterize the value function of the  optimal control problem as the unique viscosity solution to this infinite dimensional HJB equation.
\end{abstract}

\section{introduction}
Motivated by the study of large deviations of the Dyson Brownian motion, we investigate the problem of the optimal control associated to  the Dyson equation and we establish the well-posedness of the resulting infinite dimensional Hamilton-Jacobi-Bellman (HJB) equation. In a subsequent paper \citep{BLS}, we provide a complete and direct proof for the large deviations for the spectrum of large random matrices, and we study the asymptotic behavior of the transition probabilities of the Dyson Brownian motion. 
%The connection  between large deviations on one hand and optimal control and HJB equation on the other hand  is not presented here. % and we rather refer to \citep{?} for such a link in another context.

\vs

We continue with a brief discussion of our results. The  assumptions and  the precise statements are  presented in the main body of the paper. 
\vs

Random matrices are ubiquitous in mathematics and statistics as well as applications. In many settings, the entries of the random matrices are not simply random variables but stochastic processes. In such cases, it is fundamental to understand the dynamics of their eigenvalues. The model case of such phenomenon is the so-called the Dyson Brownian motion, which we present next.

\vs

The Dyson $N\times N-$matrix evolution is the matrix valued stochastic differential equation (SDE)
\beq\label{takis0}
dA_N = d D_N \ \ \text{in} \ \ (0,T),
\eeq
where $D_N$ is a $N\times N-$Hermitian  random matrix. 
\vs
Two classical  examples of $D_N$ are 
\beq\label{takis0.0}
(D_N)_{ij}= \dfrac {W^i + W^j}{\sqrt {2N}}, 
\eeq
where $W^i, W^j$ are independent Brownian motions 
and
\beq\label{takis0.01}
(D_N)_{ij}=\dfrac{W^{ij}}{\sqrt N},
\eeq
where the $W^{ij}$'s are independent Brownian motions when $j\geq i$ and $W^{ij}=W^{ji}$.
\vs

The $N-$dimensional Dyson Brownian motion, which is a stochastic particle system with repulsive  logarithmic Coulomb interaction,  expresses \eqref{takis0} in terms of the eigenvalues $\la_1\leq \ldots \leq \la_N$ of $A_N$, which evolve according to the SDE
\beq\label{takis00}
d\la_{N,t}^i=\frac{1}{N}\sum^{N}_{j=1, j\neq i}\frac{1}{\la_{N,t}^i-\la_{N,t}^j} dt + \sqrt{\dfrac{\bet}{N}} dB^i_t \ \ \text{in } \ \ (0,T), 
\eeq
where $B=(B^i,\ldots, B^N)$ is a $N-$dimensional Brownian motion which can be written in terms of $D_N$ and $\beta\in (0,2]$ is the inverse temperature. 
\vs
The value of $\beta$ %inverse temperature $\beta\in (0,2]$ is a coefficient 
is determined by the type of random matrix $D_N$. For example, physicists are interested in the case $\beta=1$ where $D_N$ is given by \eqref{takis0.0}. On the other hand, probabilists are mostly interested in the case that  $D_N$ is given by  \eqref{takis0.01} for which $\beta=2$.

\vs

The mean field equation of the Dyson evolution, which is obtained in the limit  $N\to \oo$  of the empirical measure 
\beq\label{takis00.01}
m_{N,t}=\dfrac{1}{N} \dsum_{i=1}^N \delta_{\lambda_{N,t}^i}
\eeq
of  the solution $(\la^1_{N,t},\ldots,\la^N_{N,t})$ to  \eqref{takis00} is the Dyson equation
\beq\label{takis0.00}
\begin{cases}
\pt_t m +\pt_x \big(m H[m]\big)=0 \ \ \text{in} \ \ \R\times (0,T),\\[1.2mm]
m(\cdot, 0)= \mu,
\end{cases}
\eeq
where $\mu$ is the limit of the empirical measure of the $\lambda_{N,0}^i$ and $H[m]$ is the Hilbert transform of $m$ given by
\beq\label{takis0.000}
H[m](x)= \text{PV} \left (\frac{1}{\la} \star m\right)(x)=\text{PV}\dint \dfrac{1}{x - y} dm(y)=\underset{\delta \to 0} \lim \;  \dint_{|x-y |\geq \delta}\dfrac{1}{x - y} dm(y);
\eeq
\vs
the Hilbert transform is usually stated with a normalizing constant $1/\pi$ in front of the integral. In this paper, to keep the notation simple, we assume that the constant is $1$. 
% since it is the correct scaling for the Dyson Brownian motion.

\vs
The mean field equation for the particles has been derived in law by Chan \citep{chan} and Rogers and Shi \citep{RC}, and, more recently, 
by Bertucci, Debbah, Lasry and  Lions \citep{BDLL} at the level of the repartition function.
\vs

The problem of large deviations (LDP) of random matrices has been studied extensively; see, for example, Ben Arous and Guionnet \citep{BG}, Ben Arous and Zeitouni \citep{BZ}, Dembo,  Guionnet and Zeitouni \citep{DGZ} and Guionnet and Zeitouni \citep{GZ1, GZ2}. 
\vs

In view of our interest in the LDP theory,  we consider  the controlled version of \eqref{takis0.00},  that is, given   control $\alpha: [0,T]\times\R \mapsto \R$, we study the measure evolving  according to
\beq\label{takis3}
\begin{cases}
\ds {\partial_s}m + {\partial_x}\Big( m\big(H[m] +\alpha\big)\Big)=0 \ \ \text{in} \ \ \R\times (t,T),\\
\ds m(\cdot, t)= \mu.
\end{cases}
\eeq
One of the contributions of this work is a number of  new results about the existence of solutions to \eqref{takis3} for general $\alpha$.  We note that the well-posedness of \eqref{takis3} for general controls $\alpha$ is mostly an open question as only cases with smooth vector fields $\alpha$ have been investigated. 
%Among other things, here we establish here new results of existence for \eqref{takis3} for general $\alpha$.%where $Hm$ is that Hilbert transform of $m$.
\vs

The payoff associated with the dynamics  \eqref{takis3} with running cost $F$ and terminal cost $U_T$  is the functional 

\[
\ds P(\al,m,t) =\dint_t^T\Big[\dint_\R \frac{1}{2} \al(x,s)^2  dm_s(x) + F[m_s] \Big]ds + U_T( m_T).
\]

\vs

In view  of the quadratic nature of the payoff,  we restrict our attention to initial data $\mu \in \mcl P_2$, the set of probability measures on $\R$ with  finite second order moments endowed with its natural $2-$Wasserstein metric; we show here that,  for controls of finite cost, the dynamics stay in $\mcl P_2$. 
%The space $\mcl P_2$ is endowed with its natural $2-$Wasserstein metric. 
\vs
In what follows, we  assume that $F \in C(\mcl{P}_2)$ and that either $U_T\in C(\mcl{P}_2)$ or $U_T(\mu)=\mathbf{1}_\nu$, that is, $U_T(\mu)=\oo, \ \text{if} \ \mu\neq \nu, \ \text{and} \ 0, \ \text{if} \ \mu=\nu$ for $\nu \in \Pt$. The first case is  linked with the so-called soft large deviations, while  the second is related to the LDP asymptotics of the transition probabilities.
\vs

Our results about the control of the Dyson equation and the resulting Hamilton-Jacobi-Bellman (HJB) equation extend to more general than quadratic energies and separated running costs, that is,  to $L(\alpha,m)$ with $L$ uniformly convex in $\alpha$.
% which then yield more general Hamiltonians than the one we study here. 
We leave it up to the interested reader to complete the details. 

\vs

The optimization problem we are interested in is to minimize $P(\al, m,t)$ over all $(\al,m) \in \mcl A(\mu,t)$, the set 
of pairs $(\al,m)$ solving  \eqref{takis3} and such that $(m_s)_{s\in [t,T]} \subset C([t,T],\Pt)$ and  $\al:\R\times [0,T]\to \R$ is measurable and has finite energy.
%\[L(\alpha,m)=\frac{1}{2}\int_t^T\int_\R a^2(x,s)dm(x,s).\]
%  and the pair $(\alpha,m)$ solves \eqref{takis3}.

\vs

One of the main results of the paper is the characterization of the infimum as the unique  solution of a backward   infinite dimensional HJB equation set on $\mathcal{P}_2 \times (0,T)$.
\vs

Indeed, we show that  $U(\mu,t)=\inf_{\al\in \mcl A(\mu,t)} P(\al , m,t)$ is the unique (viscosity) solution to 
\beq\label{manos1000}
\begin{cases}
-\pt_t U + \frac{1}{2}\int_{\R}|D_\mu U(\mu,t)(x)|^2d\mu(x) -\int_\R D_\mu U(\mu,t)(x) \cdot H[\mu](x)d\mu(x)=0 \ \ \text{in} \ \ \Pt\times [0,T),\\[1.5mm]
U(\cdot,T)=U_T,
\end{cases}
\eeq

where $D_\mu U$ is the horizontal or Wasserstein derivative of $U$. Precise definition is given in Section \ref{sec:der} below and more details can be found in \citep{Bbook}.

\vs 
We note that, in the absence of the Hilbert transform and for $F\equiv 0$, the minimization problem discussed above is the problem introduced and studied by Benamou and Brenier \citep{BB} in the context of mass transport, by considering a system of two finite dimensional PDEs instead of the associated HJB equation on $\mcl P_2$.
% but rather by a system of two finite dimensional PDEs. 
 A detailed presentation of this system can be found in Santambrogio \citep{Santambrogio}. The study of associated HJB equation on $\mcl P_2$ has been the subject of several works by: the second author in \citep{Lionscours}, Gangbo and Tudorascu in \citep{GT} and the first author in \citep{B}.
\vs
When, however, $H$ is present, the problem is more complicated and singular. Showing the well-posedness of \eqref{manos1000} and the fact that $U$ is a solution 
requires developing several new ideas on HJ equations on Wasserstein spaces, as well as new results on Dyson equation with a drift. Some of the challenges we encounter here were addressed in the study of the simpler problem of optimal control of the heat equation by Cosso et al \citep{cosso}, Daudin and Seeger \citep{DS}, Bayraktar, Ekren and Zhang \citep{BEZ}, Soner and Yan \citep{SY}, the first and second authors \citep{BL}, Cecchin and Delarue \citep{CD}, Cardaliaguet and the third author \citep{CS}, Daudin, Jackson and Seeger \citep{DJS}, to mention only a few.
\vs

Proposing an appropriate notion of viscosity solution to \eqref{manos1000}, and establishing its well-posedness, especially in the case $U_T =\mathbb{1}_\nu$ %, where $\mathbb{1}_\nu(m)=\oo$ if $m\neq \nu$ and $\mathbb{1}_\nu(\nu)=0$, 
are two of the major contributions of the paper. 
\vs

\subsection*{Organization of the paper} The paper is organized as follows: In Section \ref{sec:pre} we state and prove some preliminary and fundamental facts that are used in the rest of the paper. In Section \ref{sec:contr}, we introduce the controlled Dyson equation and prove a number of  important results about the trajectories that will be needed later.  In Section \ref{sec:bound}, we derive the HJB equation, and obtain various bounds on the value function. Section \ref{sec:visc} is about the notion of (viscosity) solution to the HJB equation of the controlled Dyson problem and the well-posedness of the resulting HJB equation with continuous terminal data. In Section \ref{sec:valuevisc} we establish the continuity of the value function of the controlled Dyson equation and  show that the value function is indeed a viscosity solution to the HJB equation.
% and we conclude showing  which then concludes the argument that it can be characterized as the unique viscosity solution to the associated HJ equation.

\subsection*{Notation and terminology}
The set of probability measures on $\R$ is denoted by $\mathcal{P}$ while the  set of probability measures with finite second moment $M_2(m) = \int_\R x^2dm(x)$ is denoted by $\mcl P_2$, and $\mcl L(X)$ is the law of the random variable $X$. We write $BC(\Pt)$ to denote the set of real valued bounded and continuous maps on $\Pt$. When a measure $m$ has a density, we denote it also by $m$. For $p\in [0,\oo)$, the $L^p-$space  with respect to a positive measure $m$ is denoted by $L^p(m)$, and, when $m$ is the Lebesgue measure, we write simply  $L^p$. The $2$-Wasserstein distance is $W_2$, $H^k$ is   the usual Sobolev space on $\R$, $\del_y$ is the Dirac mass at $y\in \R$, and $\mathbb 1_A$ is  the function which is $0$ in $A$ and $+\oo$ off $A$. Given $a,b\in\R$, we write $a\lesssim_l b$ if there exists $C(l)$ such that $a\leq C(l) b$ and $a\lesssim b$ if there exists $C$ such that $a\leq C b$.

\subsection*{Acknowledgments} The first and second author acknowledge a partial support from the chair FDD (Institut Louis Bachelier). The first author is partially funded by the European Union (ERC, PaDiESeM, 101222038). Views and opinions expressed are however those of the authors only and do not necessarily reflect those of the European Union or the European Research Council. Neither the European Union nor the granting authority can be held responsible for them. The third author was partially supported by  the National Science Foundation grants DMS-2153822 and DMS-2452972,  the Office for Naval Research grant N000141712095 and the Air Force Office for Scientific Research grant FA9550-18-1-0494.

\section{Preliminary facts}\label{sec:pre}

We collect here a number of facts that we will use in our analysis and sketch some of their proofs. 
For details  we refer to \citep{BDLL} and \citep{BLL}. 

\subsection{The Hilbert transform and Cotlar's identity}

The Hilbert transform  of  a measure is defined in  \eqref{takis0.000}. Using the same convention about the constant in front of the integral, if the measure has a density  $m$, then it follows that %defined by  
\beq\label{hilbert}
H[m](x)=\dint_\R\dfrac{m(y)}{x-y}dy = \dint_\R\dfrac{m(y)- m(x)}{x-y}dy
\eeq
with the integrals understood  as principal values.
\vs

We recall that, for any $1 < p < \infty$, $H$ is a bounded linear operator from $L^p$ onto itself. Moreover, $H^\ast = - H$ and $H[H[u]] = -u$.
\vs

An important fact we used several times in the paper is Cotlar's identity, which says that there exists a  universal constant $c_0>0$ such that, 
for smooth functions $u : \R \to \R$,
\beq\label{cotlar}
H[u]^2 = c_0u^2 + 2H[uH[u]],
\eeq
Multiplying \eqref{cotlar} by $u$, integrating and using the antisymmetry of H, we find
\[
\dint_\R H[u]^2 u dx = c_0\dint_\R u^3 dx- 2 \dint_\R H[u]^2 u dx = \dfrac{c_0}{3} \dint_\R u^3 dx,
\]
and, hence,
\beq\label{cotlar1}
\dint_\R H[u]^2 u dx=\dfrac{c_0}{3}\dint_\R u^3 dx.
\eeq
It follows that  for a probability measure $m \in \mathcal{P}$, $H[m] \in L^2(m)$ if and only if $m \in L^3$.
\vs

\subsection{The Wasserstein space $\mcl P_2$ and its calculus}\label{sec:der}
The space $\mcl P_2$ is a complete separable metric space when endowed with the Wasserstein distance
\beq\label{takis112}
W_2(\mu,\nu) = \inf_{\pi \in \Pi(\mu,\nu)}\left(\int_{\R^2}|x-y|^2\pi(dx,dy) \right)^\frac12,
\eeq
where $\Pi(\mu,\nu)$ is the set of couplings between $\mu$ and $\nu$.
\vs

A map $U: \mcl P_2 \mapsto \R$  is said to be horizontally differentiable at $\mu_0 \in \mcl P_2$, if there exists $\varphi \in L^2((\R,\mu); \R)$ such that, for any $\nu \in \mcl P_2$ and $ \pi \in \Pi(\mu_0,\nu)$,
$$
U(\nu) = U(\mu_0) + \int_{\R^2}\varphi(x)\cdot(y-x)\pi(dx,dy) + o\left(\left( \int_{\R^2}|x-y|^2\pi(dx,dy) \right)^\frac12 \right).
$$
In such a case, we write $D_\mu U(\mu_0)=\varphi$.
\vs

Next we  define an associated $C^{1,1}-$type regularity space (see also \citep{BL,tangent}). We remark that this notion of $C^{1,1}-$regularity is naturally linked to the one of the Hilbertian lift of a map $\phi:\Pt\to \R$. 
\vs
\begin{definition}\label{C1}
A function $U :\Pt\to \R$ is in  $C^{1,1}(\Pt)$ if, for any $\mu \in \Pt$, $D_\mu U \in L^2(\mu)$ and there exists $K > 0$
such that,  for all $\mu,\nu \in \Pt$ and  $\gamma\in \Pi(\mu,\nu)$,
\[
\int_{\R^2}|D_\mu U(\mu,x)-D_\mu U(\nu,y)|^2d\gamma(x,y) \leq K \int_{\R^2}|x-y|^2d\gamma(x,y).
\]
\end{definition}
\vs
%In what follows we write $U \in C^{1,1}(\Pt)$ if $U$ satisfies the conditions of Definition~\ref{C1}. 
It is immediate that $U \in C^{1,1}$ if and only if the map $\tilde U: L^2(\Omega,\R) \mapsto \R$ defined by $\tilde U(X) = U(\mathcal L(X))$ is in $C^{1,1}(L^2(\Omega,\R))$, where  $(\Omega,\mathcal A,\mathbb P)$ a standard probability space.
\vs
An important aspect of this notion of regularity is that it allows for the Taylor expansion stated next.
\begin{proposition}
Let $(\alpha,m)$ be such that $\al:\R\times [0,T]\to \R$ is measurable, $(m_s)_{s\in [t,T]} \subset C([t,T],\Pt)$ and the pair is a weak solution to  
\[
\partial_t m + \partial_x(\alpha m) = 0 \ \  \text{ in } \ \  \R \times (0,T),
\]
with 
\[
\int_0^T\int_\R \alpha^2_s(x)dm_s(x) < \infty.
\]
Then, for any $\phi\in C^{1,1}(\Pt)$  % there exists a constant $C=C(\phi)>0$ %$C=C(\|\phi\|_{C^{1.1}})>0$ 
and all $t \in (0,T]$, 
\[
\phi(m_t) - \phi(m_0) - \int_0^t\int_\R D_\mu\phi(m_s,x)\alpha_s(x)dm_s(x)ds \lesssim_{\phi} \left(\int_0^t\int_\R\alpha_s^2(x)dm_s(x)ds\right).
\]
\end{proposition}
\begin{proof}
We prove the claim by lifting the relation on the Hilbert space. Indeed, it follows from standard results (see, for example, \citep{ags}) that there exists a real valued process $(X_s)_{s \geq 0}$ on a standard atom-less probability space $(\Omega, \mathcal{A},\mathbb{P})$ such that %$\mathcal{L}(X_s) =m_s$ almost surely   in $[0,T]$ and 
\[
X_t = \int_0^t\alpha_s(X_s) ds \ \ \text{and} \ \ \mathcal{L}(X_s) =m_s \ \ \text{a.s.in} \ \ [0,T].
\]
Since the  lift of $\phi$ on $L^2(\Omega,\R)$  is in $C^{1,1}$, there exists some $C=C(\phi)>0$ such that 
\[
\left|\phi(m_t) - \phi(m_0) - \mathbb{E}\left[\int_0^t\alpha(s,X_s)D_\mu\phi(m_s,X_s)ds\right]\right| \lesssim_{\phi} \mathbb{E}[|X_t - X_0|^2]=\left(\int_0^t\int_\R\alpha^2dm_sds\right).
\]
\end{proof}

We refer to  \citep{Bbook} for more details on the calculus on spaces of measures.
\vs
\subsection{Regularity of the $W_2-$distance} \label{sec:regW2}

If the measure $\mu$ is absolutely continuous with respect to the Lebesgue measure, the minimum in \eqref{takis112} is achieved at a measure $\pi_{\text{opt}}$ such that 
\beq\label{takis113}
\pi_{\text{opt}}(x,y)=\mu(x)\del_{T(x)}(y) \ \ \text{with $T: \R \mapsto \R$ increasing.}
\eeq 
\vs
If the same holds for $\nu$, then the similarly given map $S$ is the inverse of $T$.
\vs

Fix $\nu \in \mcl P_2$. If $\mu_0$ is absolutely continuous with respect to the Lebesgue measure, then it follows from   \citep{AJ} that $\mu \mapsto W^2_2(\mu,\nu)$ is horizontally differentiable at $\mu_0$ and $D_\mu W_2^2(\mu_0,\nu)(x) = 2(x-T(x))$.
%see \citep{AJ}.\\
\vs

We note that, even when it is differentiable at one point, it is not true  that the function $\mu \mapsto W_2^2(\mu,\nu)$ is differentiable everywhere around it. 
\vs

To deal with this difficulty,  we  recall the following regularization of the $W_2-$distance which was introduced in \citep{BL}. For a fixed  $\nu \in \mcl P_2$, let  $\Phi_*(\mu) = W_2^2(\mu,\nu)$ and, for all $\gamma \in(0,1)$, consider  the map $\Phi_\gamma: \mcl P_2 \mapsto \R$ be defined by
$$
\Phi_\gamma(\mu) = \sup_{\mu' \in \mcl P_2}\{ W_2(\mu',\nu) - \gamma^{-1}W_2^2(\mu',\mu)\}. 
$$
%and denote by $\Phi_*(\mu) = W_2^2(\mu,\nu)$. 
The following approximation result is established in \citep{BL}.
\begin{theorem}\label{thm:approxW2}
The family $(\Phi_\gamma)_{\gamma \in (0,1)}$ converges locally uniformly toward $\Phi_*$. Moreover, if $\Phi_*$ is differentiable at $\mu_* \in \mcl P_2$,  then, for any sequences $(\mu_n)_{n \geq 1}$ in $\Pt$ and $ (\gamma_n)_{n \geq 1}$ in $(0,1)$  converging  respectively toward $\mu_*$ and $0$, and,  for any sequence $(\pi_n)_{n\geq 1}$ of couplings between $\mu_n$ and $\mu_*$, 
$$
\lim_{n \to \infty} \int_{\R^2}|D_{\mu}\Phi_{\gamma_n}(\mu_n)(x) - D_{\mu}\Phi_*(\mu_*)(y)|^2\pi_n(dx,dy) = 0.
$$
\end{theorem}
\vs

A functional closely related  to $W_2^2$ is the optimal transport cost
\be\label{defI2}
\mathcal I_2(\mu,\nu) = \inf_{\pi \in \Pi(\mu,\nu)} \Big\{-\int_{\R^2}x\cdot y\,\pi(dx,dy)\Big\}.
\ee
Since, for all $\mu,\nu\in \Pt$, 
\[2\,\mathcal I_2(\mu,\nu) + M_2(\mu) + M_2(\nu) = W_2^2(\mu,\nu),\]
it follows easily that 
 %From this relation, one easily deduce that 
 Theorem \ref{thm:approxW2} extends to approximations of $\mathcal I_2$ by the map $\Psi_\gamma: \mcl P_2 \mapsto \R$ defined by
$$
\Psi_\gamma(\mu) = \sup_{\mu' \in \mcl P_2}\{ \mcl I_2(\mu',\nu) - \gamma^{-1}W_2^2(\mu',\mu)\}.
$$
Hence, fixing $\nu \in \Pt$, writing $\Psi_*(\mu) = \mathcal I_2(\mu,\nu) $ and noting  that the horizontal differentiability of $\Psi_*$ is equivalent to the one of $\Phi_*$, we get the following result. 
\begin{theorem}\label{thm:approxI2}
The family $(\Phi_\gamma)_{\gamma \in (0,1)}$  converges locally uniformly toward $\Psi_*$. Moreover, if 
 %if $\mu_* \in \mcl P_2$ is such that 
 $\Psi_*$ is differentiable at $\mu_* \in \mcl P_2$, then, for any sequence $(\mu_n)_{n \geq 0}, (\gamma_n)_{n \geq 0}$ in $\Pt \times (0,1)$ converging  toward $(\mu_*,0)$, and,  for any sequence $(\pi_n)_{n\geq 1}$ of couplings between $\mu_n$ and $\mu_*$, 
$$
\lim_{n \to \infty} \int_{\R^2}|D_{\mu}\Psi_{\gamma_n}(\mu_n)(x) - D_{\mu}\Psi_*(\mu_*)(y)|^2\pi_n(dx,dy) = 0.
$$
\end{theorem}

\vs

\subsection{A  review  of the Dyson equation and the free entropy}

The well-posedness of the initial value problem of the Dyson equation \eqref{takis3} has been studied in many references. In this note, we use results from the  recent papers \citep{BDLL, BLL}, 
which establish  that, for each $\mu \in \mcl P$ and $T>0$, the initial value problem \eqref{takis3} 

admits a unique solution $m\in C([0,T]; \mcl P)$.
\vs
To facilitate the presentation below, we denote by $(S(t))_{t \geq 0}$ the semi-group of \eqref{takis3}, that is, if  $\mu \in \mcl P$,  then $(S(t) \mu)_{t \geq 0}$ is the unique solution to \eqref{takis3}.
\vs

It is shown in  \citep{BDLL, BLL} that,  for any $\mu\in \mcl P$ and $t>0$, $S(t) \mu\in  L^\oo(\R)$ and, in particular, 

\beq\label{takis501}
\|S(t)\mu\|_\oo \leq t^{-1/2}.
\eeq
\vs
We introduce next the free entropy function $E:
\mcl P\to \R\cup \{\pm \infty\}$ of the physical system associated to the Dyson flow which is given by

\beq\label{entropy}
E(m) = -\dfrac 12 \dint_\R \dint_\R \log(|x-y|)dm(x)dm(y), 
\eeq

and note that, even when finite, $E$ is not necessarily nonnegative. 
\vs
From Jensen's inequality, it easily follows that $E>-\oo$ for measures with finite first moment.  Moreover, when $m$ belong to $\Pt$, 
\beq\label{entropy10}
E(m)\geq -\frac12\log(M_2(m)) -\frac12.
\eeq
\vs

In the other direction, since $\log |z| \geq 0$ if $|z|\geq 1$, we find 
\[
\begin{aligned}
E(m) &\leq -\dfrac{1}{2}\dint \dint_{|x-y|\leq 1}\log(|x-y|) dm(x)dm(y)\\[1.2mm]
& \leq  \sup_{x \in \R} \dint_{|x-y|\leq 1}|\log(|x-y|)|dm(y) \leq \left( \dint_\R m^3 \right)^{\frac 13} \left( \dint_{|z|\leq 1}|\log |z||^{\frac32}\right)^{\frac 23}. %\\[1.2mm]
\ea
\]
It follows that there exists some $C>0$ such that, for any  $m\in L^3$, 
\beq\label{eq:estE}
E(m) \leq C \|m\|_{L^3}.
\eeq
\vs
For future reference,  we introduce the subset $\wtilde {\mcl P}_2$ of $\Pt$ given by 
\beq\label{takisnew}
\widetilde {\mcl P}_2  = \big \{ m \in \mcl P_2 \; :  E(m) \in \R \big\},
\eeq
and note that \eqref{entropy10} and  \eqref{eq:estE} imply that 
\[
L^3 \cap \Pt \subset \widetilde {\mcl P}_2 .
\]

We also recall  that  in \citep{BLL} it is proved that 
\beq\label{takis34}
E(m)<\infty \ \text{ if and only if} \ m \in H^{-1/2}.
\eeq
\vs
For future use, we formulate here as a lemma an important property of the free entropy.

\begin{lemma}\label{T1}
The free entropy is lower semicontinuous with respect to the weak convergence of measures. 
\end{lemma}
\begin{proof}

Assume that $m_n\rightharpoonup m$. Then, for any smooth function $\phi$, we have 
\beq\label{manos}
\underset{n\to \oo}\lim  \dint_\R \dint_\R \phi(x,y) dm_n(x)dm_n(y)) =   \dint_\R \dint_\R  
\phi(x,y) dm(x)dm(y)).
\eeq
Next we observe that, for each $K>0$,
\beq\label{takis34.1}
E(m_n)= -\dfrac 12 \dint_\R \dint_\R \log(|x-y|)dm_n(x)dm_n(y) \geq   \dfrac 12 \dint_\R \dint_\R  \big(-\log(|x-y| \wedge K )dm_n(x)dm_n(y)).
\eeq

Approximating  $\big(-\log(|x-y| \wedge K \big)$ from below by an increasing  sequence of smooth  functions yields,  after letting $K\to \oo$ and using \eqref{manos} and \eqref{takis34.1},  that 
\beq\label{takis34.12}
\underset{n\to \infty}\liminf\; E(m_n) \geq E(m). 
\eeq
\end{proof}

Formally, we also have 
\beq\label{gradE}
 D_\mu E(m,x) =-H[m](x).
\eeq

We state now an important  identity for solutions to \eqref{takis3}. This is Proposition~5.6 in \citep{BLL},
where we refer for the complete proof. 
%We recall the convention that we use the same symbol for a measure and its density. 

\begin{proposition}\label{takis502}
Let $(m_t)_{t \geq 0}$ be a $\mcl P_2-$valued  solution to \eqref{takis3}. Then, for all $0<t<t'<T$,  
\beq\label{takis502.1}
E(m_{t'}) = E(m_{t}) +\dint_t^{t'} \dint_\R (H[m_s])^2 dm_s (x) ds
= E(m_{t}) + \dfrac{c_0}{3} \dint_t^{t'} \dint_\R m_s^3 dx ds,
\eeq
and, in particular,  the map  $t \to E(m_t) $ is decreasing. 
\end{proposition}
\begin{proof}
The decreasing property of the entropy flow follows from \eqref{takis502.1}, while the second equality in \eqref{takis502.1} is an immediate consequence of Cotlar's identity. Note that, since $m_t\in L^\oo \cap \Pt$ for $t>0$, we know that $m_t\in L^3$.
\vs
We present the proof of the first equality in \eqref{takis502.1} assuming that $(m_t)_{t \geq 0}$ is a smooth solution to \eqref{takis3} and refer to \citep{BLL} for the complete argument.
\vs 
Using \eqref{gradE} leads to
\[\dfrac{d}{dt} E(m_t) = - \dint_\R H[m_t]^2 m_t dx.\]

\end{proof} 

Given $m \in \widetilde {\mcl P}_2$, the map $t \to E(S(t) m)$ is a continuous function. It then follows from \eqref{takis501} and  \eqref{takis502.1} that, for any $m \in \mathcal{P}$ and $t > 0$,
\beq\label{manos1}
E(S(t) m) \leq -\dfrac{c_0}{3}\log t.
\eeq

We also recall the following result from \citep{BLL} and sketch its proof.
\begin{proposition}\label{prop:decreas}
Let $(m_t)_{t \geq 0}$ and $(\widetilde {m}_t)_{t \geq 0}$ be two $\mcl P_2-$valued solutions to  \eqref{takis3}. Then, for any $t \leq s$, \[W_2(m_s,\widetilde {m}_s) \leq W_2(m_t,\widetilde {m}_t).\]
\end{proposition}
\begin{proof}
We only give the proof in the case that  $(m_t)_{t \geq 0}$ and $(\widetilde {m}_t)_{t \geq 0}$ are $L^3-$valued and refer 
to  \citep{BLL} for the complete argument.
\vs
Let $\gamma$ be an optimal measure for $W_2(m_0, \widetilde m )$ and  $(X_0,Y_0)$ 
a couple of real random variables on a standard probability space such that $\mathcal{L}((X_0,Y_0)) = \gamma$. Since it is assumed that  the processes are valued in $L^3$, we deduce from Cotlar's identity that $H[\mathcal{L}(X_0)](X_0)$ is a squared integrable random variable.
\vs
Consider now the odes
\vskip.0025in
\[
\dfrac{d}{dt} X_t = H[\mathcal{L}(X_t)](X_t) \quad \text{and}  \quad \dfrac{d}{dt} Y_t = H[\mathcal{L}(Y_t)](Y_t).
\]
\vskip.05in
By construction, for any $t \geq 0$, $\mathcal{L}(X_t) = m_t$ and $\mathcal{L}(Y_t) = \widetilde {m}_t$. It then follows that, for all $t \geq 0$,
\[
\bal
\dfrac{d}{dt}W_2^2(m_t,\widetilde {m}_t)\bigg |_{t = 0} \leq \dfrac{d}{dt}\mathbb{E}[|X_t - Y_t|^2]\bigg |_{t = 0}.
\eal
\]
Then
\[
\bal
\dfrac12\frac{d}{dt}\mathbb{E}[|X_t - Y_t|^2]\bigg |_{t = 0} &= \mathbb{E}[(X_0-Y_0)(H[\mathcal{L}(X_0)](X_0) - H[\mathcal{L}(Y_0)](Y_0)]\\[1.2mm]
& = \dint_\R\dint_\R\dint_\R\dint_\R (x-y)\left(\dfrac{1}{x-x'} - \dfrac{1}{y -y'}\right)d\gamma(x,y)d\gamma(x',y')\\[1.2mm]
& = \dint_\R\dint_\R\dint_\R\dint_\R \frac{-(x-y)^2 + (x-y)(x'-y')}{(x-x')(y-y')} d\gamma(x,y )d\gamma(x',y').
\eal
\]
\vs
Recall next that, since $\gamma$ is an optimal coupling, it follows that, $\gamma\times \gamma$-almost everywhere,  $(x-x')(y-y') \geq 0.$ 
\vs
Then, using that $2(x-x')(y-y') \leq (x-y)^2 + (x'-y')^2$, we find that
\vs
\[
\bal
\dfrac 12\dfrac{d}{dt}\mathbb{E}[|X_t - Y_t|^2]\bigg |_{t = 0} &\leq \dfrac 12\dint_\R\dint_\R\dint_\R\dint_\R \frac{(x'-y')^2-(x-y)^2 }{(x-x')(y-y')} d\gamma(x,y)d\gamma(x',y').
\eal
\]
\vs
Since the term on the right side vanishes by symmetry, we get  that 
\[\dfrac{d}{dt}\mathbb{E}[|X_t - Y_t|^2]\bigg |_{t = 0} \leq 0,\].

The proof is now complete given that  the same argument can be made for any time.

\end{proof}
The following result is the analogous of its counterpart for the heat equation.
It has been proven using the language of free probability and their associated Wasserstein distances, see \citep{BV}. We provide here a more analytical approach.
\begin{proposition}
For any  $\mu \in \mcl P_2$ and $t \geq 0$, we have  $$W_2(S_t\mu,\mu) \leq \sqrt{t}.$$
\end{proposition}
\begin{proof}
In view of  the continuity of $S_t$ for $t \geq 0$, we prove the estimate for $\mu \in \mcl P_2$ with a smooth density, bounded from below by a positive constant on any compact set. %as the result then extends by continuity. 
\vs
For such a $\mu$, denote by $X_t: [0,1] \mapsto \R$ the inverse cumulative distribution function of $S_t\mu$. We remark that $X_t$ is a smooth function, since $S_t\mu$ is smooth and, in particular,  satisfies the equation
$$
\partial_t X_t(u) = \int_{0}^1\frac{1}{X_t(v)-X_t(u)}dv \text{ on } (0,\infty)\times (0,1).
$$
Furthermore, 
$$
W_2^2(S_t \mu,\mu) = \int_0^1(X_t(u) - X_0(u))^2du.
$$
In particular,  we can compute
$$
\ba
\frac{d}{dt}W_2^2(S_t \mu,\mu) &= 2 \int_0^1(X_t(u) - X_0(u))\int_{0}^1\frac{1}{X_t(v)-X_t(u)}dvdu\\
& = \int_0^1\int_0^1 \Big[1 - \frac{X_0(u) - X_0(v)}{X_t(u) - X_t(v)}\Big]dudv \leq 1.
\ea
$$
The result follows. 
\end{proof}

Next we recall without proofs two results about the Wasserstein distance which are well known in the transport theory. The first is about the displacement convexity of the $L^p-$norms.

\begin{proposition}\label{prop:displace}
Fix  $m^1, m^2 \in \mcl P_2(\R)$ and let  $(m_s)_{s \in [0,1]}$ be a $W_2-$geodesic between $m^1$ and $m^2$. Then, for all $1\leq p \leq \infty$ and $s \in [0,1]$, 
\[
\|m_s\|_p \leq \max(\|m^1\|_p,\|m^2\|_p).
\]
\end{proposition}
\vs
The second result is a useful upper bound for the $2-$Wasserstein metric between measures in $\Pt$.

\begin{proposition}\label{takis505}
Fix $\mu, \nu \in \mcl P_2$,  $t > 0$ and a pair $(\alpha, m)$ such that $m \in C([t,T];\mcl P_2)$, $m(t) =\mu, m_T = \nu$ and 
\[\partial_s m+ \partial_x (\alpha m) = 0 \   \text{ in } \  \R \times (t,T).\]
\vs
Then
\be\label{rep:w2}
W_2^2(\mu,\nu) \leq (T-t) \dint_t^T\dint_\R\alpha_s^2(x)dm_s(x).
\ee
\end{proposition}
\vs

\section{The controlled Dyson equation}\label{sec:contr}
We introduce and study some properties of the controlled Dyson equation 
\beq\label{takis30.1}
\pt_s m + \pt_x \big(m_s\big(H[m_s] +\al_s\big)\big)=0  \ \ \text{in} \ \ \R\times (t,T) \ \ \text{and} \ \ m_t=\mu.
\eeq
with a measurable control or velocity $\alpha:[t,T]\times \R \mapsto \R$. 

\vs
We say that the pair $(\alpha,m)$ is a weak solution to \eqref{takis30.1} if  $m \in C([t,T],\mcl P_2)$ and, 
%  $(\alpha,m)$ is a weak solution to \eqref{takis30.1}, meaning that, 
for any smooth test functions $\varphi$ with compact support in $\R\times [0,T)$, the weak formulation of \eqref{takis30.1} against $\varphi$ holds.

%
% is satisfied in the sense of $m \in C([t,T],\mcl P_2)$ and 
%%we say that $(\alpha,m)$ is a weak solution to \eqref{takis30.1} if 
%for any smooth test functions $\varphi$ with compact support, the weak formulation of \eqref{takis30.1} against $\varphi$ holds. 
%\vs
%We denote the set of such pairs $\mcl A(\mu,t)$.
\vs
Given a pair $(\alpha,m)$ we define its energy $L(\alpha,m)$ by
\beq\label{manos15}
L(\alpha,m)=\dfrac{1}{2}\dint_t^T\dint_\R (\alpha_s(x))^2 dm_s(x).
\eeq

We consider three problems. The first one is about compactness of sequences of solutions of bounded energy. The second one is related to  the question of existence of solutions for controls arising from the use of $C^{1,1}$ functions. The third one tackles the transposition of a control $\alpha$ used for one solution to another solution.

\vs
For the first,  given $\mu\in \Pt$ and $t\in [0,T)$,  we denote by $\mcl A(\mu, t)$ the set of  pairs 
$(\al,m)$ such that $m\in C([t,T]; \Pt)$, $\al: \R\times [t,T]\to \R$ is measurable and has finite energy, and  $(\al,m)$  is a weak solution to \eqref{takis30.1}
\vs

For the second,  given $\mu  \in \Pt$, $\nu \in \widetilde {\mcl P}_2$ and $t\in [0,T)$, we consider  the set  $\mcl A(\mu,\nu,t)$ of  pairs $(\al,m)$ such that $m\in C([t,T]; \Pt)$, $\al: \R\times [t,T]\to \R$ is measurable and has finite energy and $(\al,m)$ is  a weak  solution to 
\beq\label{takis30.5}
\begin{cases}
\pt_s m_s + \pt_x \big(m_s\big(H[m_s] +\al\big)\big)=0  \ \ \text{in} \ \ \R\times (t,T),\\
m_t=\mu \ \ \text{and} \ \  m_T=\nu.
\end{cases}
\eeq

We note that for \eqref{takis30.5}, we assume that the terminal measure has finite entropy. It turns out, as we will see below, that this is necessary in order to have a solution. 

%\vs
%
%Given any pair $(\alpha,m)$ in $\mcl A(\mu,t)$ or $\mcl A(\mu,\nu,t)$ we define its energy $L(\alpha,m)$ by
%\beq\label{manos15}
%L(\alpha,m)=\dfrac{1}{2}\dint_t^T\dint_\R (\alpha_s(x))^2 dm_s(x).
%\eeq
%

\subsection{The compactness of sub-levels of the energy}

The following result implies that sequences of pairs, solutions to \eqref{takis30.1}, which have bounded energy are pre-compact. 
The claim also asserts that the controlled Dyson equation has a regularizing effect manifested by the fact that, for $s>t$, $m_s$ has finite entropy and, for almost all $s \in (t,T)$, $m_s\in L^3$.

\begin{proposition}\label{prop:comp}
Fix $t\in (0,T)$ and $(\alpha,m)\in \mcl A(\mu,t)$. Then, there exists $C>0$ depending only on $L$, $T-t$, and $M_2(\mu)$ such that, for all $s\in (t,T]$,
\[
M_2(m_s) \leq C  \ \  \text{and} \ \ E(m_s) \leq C(1+ (s-t)^{-\frac12}).
\]

Moreover, for all $s,s'\in (t, T]$ such that $ s\leq s' $, 
\be\label{eq:est1}
\dint_s^{s'}\dint_\R m_\tau^3 dx d\tau \leq 2\dint_s^{s'}\dint_\R\alpha^2m_\tau(dx) d\tau + C(s'-s) + E(m_s) - E(m_{s'}).
\ee
\end{proposition}
\vs
\begin{proof}
To prove the estimates, we assume that the path $(\alpha_s,m_s)_{s\in [t,T]}$ is smooth. The general case then follows from standard approximation results.
\vs

Since 

\[2\dint_\R x m H[m] dx= \dint_\R\dint_\R \dfrac{x m_s(x) m_s(y)}{x-y} dxdy=\dint_\R\dint_\R \dfrac{x-y}{x-y} m_s(x) m_s(y) dx dy=1, 
\]
a direct calculation yields 
\[
\frac{d}{ds}M_2(m_s)  = 2\dint_\R x\big ( \alpha_s(x) + H[m_s]\big ) m_sdx 
\leq  M_2(m_s) + \dint_\R \alpha^2_s(x)m_s dx + 1. 
\]

%Since 
%
%\[2\dint_\R x m H[m] dx= \dint_\R\dint_\R \dfrac{x m_s(x) m_s(y)}{x-y} dxdy=\dint_\R\dint_\R \dfrac{x-y}{x-y} m_s(x) m_s(y) dx dy=1, 
%\]
%

\vs
Next we  compute the evolution of the free entropy along the trajectory. 
Using Cotlar's identity we find that,   for all $s > t$, 
\be\label{est:fond}
\bal
\dfrac{d}{ds}E(m_s) &= -\dint_\R H[m_s](x)(\alpha_s (x) + H[m_s](x)])dm_s(x)\\
&\leq \dfrac{1}{2} \dint_\R \alpha_s(x)^2 dm_s(x)  - \dfrac 12 \dint_\R (H[m_s](x))^2 dm_s(x)\\
&= \dfrac 12 \dint_\R \alpha_s(x)^2 dm_s(x) -\dfrac{c_0}{6} \dint_\R m^3(x) dx.
\eal
\ee
Recall that, in view of  \eqref{eq:estE},  there exists some  $C>0$ such that 
\[
E(m)^3\leq C  \dint_\R m^3(x) dx.
\]

It then follows that, for some other $C>0$, 
\[
\dfrac{d}{ds}E(m_s) + C (E(m_s))^3 \leq \dfrac 12 \dint_\R \alpha_s(x)^2 dm_s(x).
\]
Since the map $s \to \dint_\R \alpha_s^2(x) dm_s(x)$ is in $L^1$ with a norm bounded by $L$, we find  that, for all $s \in (t,T]$, 
\[
E(m_s) \leq C(1+ (s-t)^{-\frac12}).
\]

The last claim is obtained by integrating \eqref{est:fond} between $s$ and $s'$.

\end{proof}

Proposition~\ref{prop:comp} implies that pairs in $\mcl A(\mu,t)$ with finite energy have finite free entropy for all $s>t$. Note that this regularizing property is weaker than the one of the Dyson equation, since we are only able to obtain a decay in $t^{-\frac 12}$ compared to the logarithmic one. This is of course the effect of the term in $\alpha$. 
\vs

It also follows that, as soon as the energy of the solution of the controlled Dyson equation is finite, then $E(m_s)$ stays bounded. As a result,  with finite energy, we can only reach states $m_T$ such that $E(m_T)\in \R$.
\vs

We show next that the subsets of $\mcl A(\mu,t)$  with bounded energy are pre-compact.

\begin{proposition}\label{takis1250}
Let $(\mu_n,t_n)_{n\geq 1}\in \Pt\times [0,T)$ and $(\al_n, m_n)_{n\geq 1}$ such that $(\al_n,m_n) \in \mathcal A(\mu_n,t_n)$ %be  sequences in $\Pt\times \Pt\times [0,T)$ and $\mathcal A(\mu_n,\nu_n,t_n)$ respectively$
be such that, for some $K>0$, % and $\ol t \in [0,T)$,
\[\underset{n\geq 1}\sup \dint_{t_n}^T \dint_\R \al_n^2 m_{n,s} (dx) ds\leq K.\]
%and assume that $\underset{n\to \infty} \lim t_n = \ol t \in [0,T).$ 
Then there exists a modulus of continuity $\omega$ such that, for all $s,s'\in [ t_n,T)$, %and $n$ such that $s,s'\in[t_n,T)$,
\[W_2(m_{n,s}, m_{n,s'})\leq \omega(|s-s'|).\]
%(ii)~Let $(\mu_n,\nu_n,t_n)_{n\geq 1}\in \Pt\times \Pt\times [0,T)$ and $(\al_n, m_n)_{n\geq 1 \in \mathcal A(\mu_n,\nu_n,t_n)}$ %be  sequences in $\Pt\times \Pt\times [0,T)$ and $\mathcal A(\mu_n,\nu_n,t_n)$ respectively$
%be such that, for some $K>0$, % and $\ol t \in [0,T)$,
%\[\underset{n\geq 1}\sup \dint_{t_n}^T \dint_\R \al_n^2 m_{n,s} (dx) ds\leq K.\]
%%and assume that $\underset{n\to \infty} \lim t_n = \ol t \in [0,T).$ 
%Then there exists a modulus of continuity $\omega$ such that, for all $s,s'\in [ t_n,T)$ %and $n$ such that $s,s'\in[t_n,T)$,
%\[W_2(m_{n,s}, m_{n,s'})\leq \omega(|s-s'|).\]
\end{proposition}
\begin{proof}
Fix $\phi \in H^k$ with $k\geq 3$ so that $H^k$ is continuously  imbedded in $C^2(\R)$, and assume that $\|\phi\|_{H^k}\leq 1.$
\vs

Using  the equation satisfied by $m_n$ and for any $n\geq 1$ and $s\in [t_n,T)$ we have

\[\dfrac{d}{ds}\bigg \rvert \dint_\R \phi(x)m_{n,s}(dx)\bigg \lvert \leq \dint_\R |\phi'(x)||\al_n(x,s)|m_{n,s} (dx) + \dint_\R\dint_\R \bigg \rvert\dfrac{\phi'(x)-\phi'(y)}{x-y} \bigg \lvert m_{n,s}(dx)m_{n,s}(dy).
\]
\vs
After using H\"{o}lder's inequality, we find that,  for some $C=C(\|\phi\|_{H^k})$,

\[\dfrac{d}{ds}\bigg \rvert \dint_\R \phi(x)m_{n,s}(dx)\bigg \lvert \leq C\Big[\Big(\dint \al_n^2 m_{n,s}(dx)\Big)^{1/2} +1\Big].
\] 
%for some $C=C(\|\phi\|_{H^k})$. % such that $\|phi'\|_\oo+\|\phi''\|_\oo$. 
\vs
It follows that, for some other constant $C$ which depends on K and is independent of $n$, we have 
\beq\label{takis2001.1}
\underset{s\in [t_n,T]}\sup \|\pt_s m_n\|_{L^2_s(H^{-k}_x)}\leq C.
\eeq
We show next that \eqref{takis2001.1} implies some compactness for the sequence $\big((m_{n,s})_{s\in (t,T]} \big)_{n\geq 1}$. 
\vs

Indeed, let $m_{n,\ep}=m_n\ast_x \rho_\ep$, where $\rho_\ep$ is a Gaussian with mean $0$ and variance $\ep>0$. The previous argument yields an independent of $n$ constant $C_\ep>0$ such that 
\beq\label{takis2002}
\underset{n\geq 1} \sup \underset{s\in [t_n,T]}\sup \|\pt_s m_{n,\ep}\|_{L^2_s(H^{-k}_x)}\leq C_\ep.
\eeq
\vs
It follows that,  for all $n\geq 1$ and $s,s'\in [t_n,T]$,
\beq\label{takis2002}
\|m_{n,\ep,s}-m_{n,\ep,s'}\|_{H^1_x}\leq C_\ep|s-s'|^{1/2}.
\eeq
Moreover, we also have, for all $n\geq 1$, $\ep>0$  and $s \in [t_n,T]$, that 
\beq\label{takis2003}
W_2(m_{n,s}, m_{n,\ep,s})\leq \ep^{1/2}.
\eeq
\vs
Hence, for all $n\geq 1$ and $s,s'\in [t_n,T]$,
\[
\ba
W_2(m_{n,s}, m_{n,s'})&\leq W_2(m_{n,s}, m_{n,\ep, s}) + W_2(m_{n,\ep, s}, m_{n,\ep, s'}) +W_2(m_{n,\ep, s'}, m_{n, s'})\\
&\leq \ep^{1/2} + C_\ep|s-s'|^{1/2} + \ep^{1/2},
\ea
\]
and the claim follows.

\end{proof} 

As a consequence of the two previous propositions, we have the following result.

\begin{proposition}\label{takis2004}
Let $(\mu_n,t_n)_{n\geq 1} \in \Pt\times [0,T)$ and $(\al_n, m_n)_{n\geq 1} \in \mathcal A(\mu_n,t_n) $ be  sequences 
%in $\Pt\times \Pt\times [0,T)$ and $\mathcal A(\mu_n,\nu_n,t_n)$ respectively 
such that, for some $K>0$,
\beq\label{takis2001.11}
\underset{n\geq 1}\sup \dint_{t_n}^T \dint_\R \al_n^2 dm_n(x)\leq K,
\eeq
and assume that, as $n\to \oo$, $ t_n \to \ol t \in [0,T)$ and  $\mu_n\rightharpoonup \ol \mu$.
% and $\nu_n\rightharpoonup \ol \nu$.
%$\underset{n\to \infty} \lim ( \mu_n, \nu_n) = (\ol \mu,\ol \nu).$
Then,  there exists $(\ol \al, \ol m)\in \mathcal A(\ol \mu,t)$ such that 
\[\dint_{\ol t}^T\dint_\R \ol {\al} ^2 \ol m (dx) \leq K,\]
and, moreover, after extracting, if necessary,  a subsequence,     
\[\underset{n\to \oo} \lim \underset{s\in [\max(t_n, \ol t), T]} \sup W_2(m_{n,s}, \ol m_s)=0.\]  
%(ii)~Let $(\mu_n,\nu_n,t_n)_{n\geq 1} \in \Pt\times \Pt\times [0,T)$ and $(\al_n, m_n)_{n\geq 1} \in \mathcal A(\mu_n,\nu_n,t_n) $ be  sequences 
%%in $\Pt\times \Pt\times [0,T)$ and $\mathcal A(\mu_n,\nu_n,t_n)$ respectively 
%such that, for some $K>0$,
%\beq\label{takis2001.11}
%\underset{n\geq 1}\sup \dint_{t_n}^T \dint_\R \al_n^2 dm_n(x)\leq K,
%\eeq
%and assume that, as $n\to \oo$, $ t_n =\ol t \in [0,T)$, $\mu_n\rightharpoonup \ol \mu$ and $\nu_n\rightharpoonup \ol \nu$.
%%$\underset{n\to \infty} \lim ( \mu_n, \nu_n) = (\ol \mu,\ol \nu).$
%Then there exists some  $(\ol \al, \ol m)\in \mathcal A(\ol \mu,\ol \nu,t)$ such that 
%\[\dint_{\ol t}^T\dint_\R \ol {\al} ^2 \ol m (dx) \leq K,\]
%and, moreover, after extracting if necessary a subsequence,     
%\[\underset{n\to \oo} \lim \underset{s\in [\max(t_n, \ol t), T]} \sup W_2(m_{n,s}, \ol m_s)=0.\]                                           
\end{proposition}
\begin{proof}
It follows from Proposition~\ref{takis1250} that the  curves $\big(m_{n,s}\big)_{s\in [t_n, T]}$
are equicontinuous. Hence, in view of  the Ascoli-Arzela theorem and up, if necessary, to a subsequence, the $m_n$'s converge uniformly to a curve $(\ol m)_{s\in [\ol t,T]}$. In addition, we know that the sequence $\big(H[m_n]m_n)_{n\in \N}$  converges weakly to $H[\ol m]\ol m$. To conclude we need to show the existence of $\ol \al$, and for this we argue as follows. 
\vs

Fix  $t\in (\ol t,T)$.  Then \eqref{takis2001.11} 
%The assumption on the uniform bound on the cost for $(\al_n,m_n)$ 
yields that the measurable maps $\beta_n=\al_n m_n:\R\times[t_n,T]\to \R$ satisfy the uniform bound
\[\dint_{t_n}^T \dint_\R \dfrac{\beta_n^2}{m_n} dx ds\leq K,\]

while, in view of Proposition~\ref{prop:comp} and the fact that  

\[\underset{n\geq 1}\sup \; M_2(\mu_n)<\oo,\]
which is a consequence of the assumption of the weak-$\ast$ convergence of  $\mu_n$'s, we know that 
\[\underset{n\geq 1}\sup \dint_{t_n}^T\dint_\R m_n^2 dx ds <\oo. \]
Then, using H\"{o}lder's inequality,  we find  that 
\[
\ba
\underset{n\geq 1}\sup \dint_{t_n}^T\dint_\R |\beta_n|^{3/2} dx ds &= \underset{n\geq 1}\sup \dint_{t_n}^T\dint_\R  \bigg\rvert \dfrac{\beta_n}{m_n^{1/2}}\bigg\lvert^{3/2} m_n^{3/4} dx ds\\
& \leq \underset{n\geq 1}\sup \Big(\dint_t^T \dfrac{\beta^2_n}{m_n} dx ds \Big)^{3/4} \underset{n\geq 1}\sup \Big(\dint_t^T \dint_\R m_n^3  dx ds\Big)^{1/4}<\oo.
\ea
\]
It then follows that the sequence $(\beta_n)_{n\geq1}$ converges weakly in $L^{3/2}$ to some $\ol \beta$ and the limit $\ol m$ of the $m_n$'s is a weak solution to 
\[\pt_s \ol m + \pt_x \ol \beta + \pt_x \big(H[\ol m] \ol m\big)=0 \ \ \text{in} \ \ \R\times (t,T).\]
Setting $\ol \al=\dfrac{\ol \beta}{\ol m}$, with the understanding that $\ol \al=0$ when $\ol m=0$,  yields the result.

%& 1 %\leq \big(\dint_t^T \dfrac{\beta^2_n}{m_n} dx ds \big)^{3/4} \big(\dint_t^T m_n^3 \big)^{1/4}.
%\ea
\end{proof}

\subsection{Existence of solutions with smooth controls}
We now prove a result which highlights types of controls which yields a well posed Dyson equation.
\begin{proposition}\label{T2}
Given   $\phi \in C^{1,1}(\Pt)$ and  $\mu \in \Pt\cap L^3$,  there exists a unique solution $m\in C_t(\Pt \cap L^3)$ to 
\be\label{eq:mixed}
\partial_t m + \partial_x(m(D_\mu\phi(m) + H[m])) = 0 \  \text{ in } \ \R \times (0,\oo) \ \ \text{and} \ \ m(\cdot, 0)=\mu.
\ee
\end{proposition}
\begin{proof}
We start with the existence which follows from a compactness argument. Let $(\Omega, \mathcal{A},\mathbb{P})$ be  a standard atomless probability space and consider a random variable $X_0 \in L^2(\Omega, \R)$ such that $\mathcal{L}(X_0) = \mu$. The claim is equivalent to proving that there exists a process $(X_t)_{t \geq 0}$ satisfying, almost surely,
\[
X_t = X_0 + \int_0^t \Big [D_\mu\phi(\mathcal{L}(X_s),X_s) + H[\mathcal{L}(X_s)](X_s)\Big ]ds.
\]
%Note that thanks to Proposition 3.1 (in the full version...), it is sufficient to prove a local result of existence of such a process. Indeed, once a local existence result is established, Proposition 3.1 yields an a priori estimate which allows to extend it to arbitrary long time interval.
We construct such a process using a Trotter-Kato scheme as follows. We define, for $N \geq 1$ and $i \leq N$, the times $t_{i,N} = iT/(N+1)$, and consider  the trajectory $m^N$  given by 
\beq\label{eveni}
\partial_s m^N + 2\partial_x(m^NH[m^N])  = 0 \  \text{ in } \ \R \times (t_{i,N},t_{i+1,N}) \ \  \text{if \ \ $i$ is even},
\eeq
and 
\beq\label{oddi}
\partial_s m^N + 2\partial_x(m^ND_\mu\phi(m^N))  = 0 \text{ in }  \ \R \times (t_{i,N},t_{i+1,N}) \ \ \text{if $i$ is odd}.
\eeq
\vs
Note that \eqref{oddi} is  uniquely solvable on any time interval given an initial condition, since it simply amounts to solving in $L^2(\Omega, \R)$ the ode
\beq\label{takis1.00}
dX_t = D_\mu\phi(\mathcal{L}(X_t),X_t)dt.
\eeq
The assumed $C^{1,1}-$regularity of $\phi$ means that the drift $D_\mu\phi$ 
is  uniformly Lipschitz continuous in $L^2(\Omega,\R)$. Hence,  \eqref{oddi} is uniquely solvable for each initial condition in $\Pt$.
\vs
Furthermore, the Dyson equation \eqref{eveni} is also uniquely solvable for any initial condition in $\Pt$. Moreover the second moments $M_2(m^N_t)$ are  bounded locally uniformly in $t$ and uniformly  in $N$.
% and locally in $t$.
%uniform bounds on $M_2(m^N_t)$ is standard.
\vs
Next we note that that the sequence  $((m^N)_{t \in [0,T]})_{N \geq 0}$ is precompact for any $T> 0$. Indeed, 
%the bound above and the  uniform continuity in time of the sequence $((m^N)_{t \in [0,T]})_{N \geq 0}$ 
%yield that the sequence  $((m^N)_{t \in [0,T]})_{N \geq 0}$ is precompact for any $T> 0$. 
%\vs
%It follows  that any limit point of this sequence is a weak solution of \eqref{eq:mixed}. Indeed, 
%\vs
%The former follows from the  uniform continuity in time of the sequence $((m^N)_{t \in [0,T]})_{N \geq 1}$. For the latter, 
in the time intervals $[t_{i,N},t_{i+1,N}]$ for $i$ odd,  the regularity of the ode on $L^2(\Omega,\R)$ yields that there exists $C>0$ such that, for any $N >1$, odd $i \leq N$ and $s,s' \in [t_{i,N},t_{i+1,N}]$,
\[
W_2(m^N_s,m_{s'}^N) \leq C|s-s'|.
\]

Moreover, Proposition~\ref{takis1250}  yields a uniform modulus of continuity which holds for $s,s'$ in the time intervals $[t_{i,N},t_{i+1,N}]$ for $i$ even. 
\vs
Thus, there exists a continuous $\omega: \R_+ \to \R_+$ such that, for any $N \geq 1$, $s,s' \in [0,T]$,
\[
W_2(m^N_s,m_{s'}^N) \leq \omega(|s-s'|).
\]
Hence, extracting, if necessary,  a subsequence, the sequence $((m^N)_{t \in [0,T]})_{N \geq 1}$ converges uniformly toward some $\ol m \in \mathcal{C}([0,T],\Pt)$.
\vs
Next we observe that, for all  $\psi \in C^\infty (\R\times [0,T])$  and  $N \geq 1$,
\[
\ba
\int_\R\psi(T,x)dm^N_t(x) &= \int_\R\psi(0,x)dm_0(x) 
+2\sum_{j = 1}^{N/2} \int_{t_{2j,N}}^{t_{2j +1,N}}\int_\R\int_\R \frac{\psi(x,s)-\psi(y,s)}{x-y}dm^N_s(x)dm^N_s(y)ds\\[1.2mm]
& +2\sum_{j = 1}^{N/2} \int_{t_{2j+1,N}}^{t_{2j +2,N}}\int_\R\partial_x\psi(x,s)D_\mu\phi(m^N_s,x)dm^N_s(x)ds.
%\\[1.2mm]
%&+\sum_{j = 1}^{N/2} \int_{t_{2j+1,N}}^{t_{2j +2,N}}\int_\R\partial_x\psi(x,s)D_\mu\phi(m^N_s,x)dm^N_s(x)ds.
\ea
\]
Using the uniform continuity of the sequence $((m^N)_{t \in [0,T]})_{N \geq 1}$, we find  a modulus of continuity $\omega$ such that, for any $N \geq 1$ and for all $j \leq N/2 -1$,
%\be\label{eq:estimateH}

\[
\ba
\bigg| \int_{t_{2j,N}}^{t_{2j +1,N}}\int_\R\int_\R \frac{\psi(x,s)-\psi(y,s)}{x-y}dm^N_s(x)dm^N_s(y)ds % \\[1.2mm]
& - \int_{t_{2j+1,N}}^{t_{2j +2,N}}\int_\R\int_\R \frac{\psi(x,s)-\psi(y,s)}{x-y}dm^N_s(x) dm^N_s(y)ds\bigg |\\[1.2mm]
& \leq C\|\psi\|_{C^2}\int_{t_{2j,N}}^{t_{2j+1,N}}\omega(N^{-1})ds,
\ea
\]
%\ee
and, similarly, given any process $(X^N_s)_{s \geq 0}$ with law  $m^N_s$ at time $s$,
\[
\ba
\bigg| &\int_{t_{2j,N}}^{t_{2j +1,N}}\int_\R\partial_x\psi(s,x)D_\mu\phi(m^N_s,x)dm^N_s(x)ds \\
&- \int_{t_{2j+1,N}}^{t_{2j +2,N}} \int_\R\partial_x\psi(x,s)D_\mu\phi(m^N_s,x)dm^N_s(x)ds\bigg|\\
&= \bigg|\int_{t_{2j,N}}^{t_{2j +1,N}} \mathbb{E}\Big[\partial_x\psi(X_s,s)D_\mu\phi(\mathcal{L}(X^N_s),X^N_s) -\partial_x\psi(s+ N^{-1},X_{s+ N^{-1}})D_\mu\phi(\mathcal{L}(X^N_{s+ N^{-1}}),X^N_{s+ N^{-1}})\Big]ds\bigg| \\
& \hskip.5in  \leq C\|\psi\|_{C^2}\int_{t_{2j,N}}^{t_{2j+1,N}}(N^{-1} + \mathbb{E}[|X^N_s - X^N_{s + N^{-1}}|^2]^\frac12)ds.
\ea
\]

Since it is always possible to choose a process such that, for almost every $s \in [t_{2j,N},t_{2j+1,N}]$, \[\mathbb{E}[|X_s - X_{s + N^{-1}}|^2] = W_2^2(m^N_s,m^N_{s + N^{-1}}),\] we finally 
%arrive at an inequality similar to \eqref{eq:estimateH}, that is, 
we  deduce that there exists $r(N,\psi)$ such that 
\[|r(N,\psi)| \leq C\sum_{i = 0}^N\int_{t_{i,N}}^{t_{i+1,N}}\|\psi\|_{C^2}\omega(N^{-1})ds\leq C \|\psi\|_{C^2}\omega(N^{-1})\] 
and 
\[
\ba
\int_\R\psi(x,T)dm^N_t(x) =  & \int_{0}^{t}\int_\R\int_\R \frac{\psi(x,s)-\psi(y,s)}{x-y}dm^N_s(x)dm^N_s(y) \\
&+ \int_\R\partial_x\psi(s,x)D_\mu\phi(m^N_s,x)dm^N_s(x)ds +\int_\R\psi(0,x)dm_0(x) + r(N,\psi).
\ea
\]
%with  
%\[|r(N,\psi)| \leq C\sum_{i = 0}^N\int_{t_{i,N}}^{t_{i+1,N}}\|\psi\|_{C^2}\omega(N^{-1})ds\leq C \|\psi\|_{C^2}\omega(N^{-1}).%\] 
The existence follows by passing to the limit $N \to \infty$, along a subsequence if necessary.\\

%To show the uniqueness of the  solutions,  %follows from a Gronwall like estimate. 
Next, consider two solutions $m$ and $\tilde m$ as in the statement, with respective initial conditions $\mu$ and $\tilde \mu$, as well as an optimal coupling $\gamma \in W_2(\mu,\tilde \mu)$. Then, from the contraction of the Wasserstein distances by the Dyson flow, we deduce
$$
\begin{aligned}
\frac{d}{dt}W_2^2(m_t,\tilde m_t)|_{t =0} \leq \int_{\R\times \R}|D_\mu \phi(\mu)(x) - D_\mu \phi(\tilde \mu)(y)|^2\gamma(dx,dy)\\
\leq  C\int_{\R\times \R}|x - y|^2\gamma(dx,dy)= C W_2^2(\mu,\tilde \mu),
\end{aligned}
$$
where the second inequality comes from the fact that $\phi \in C^{1,1}$. The uniqueness follows from Gronwall's Lemma.
\end{proof}

% In what follows we consider  
%\beq\label{manos10}
%\begin{cases}
%\pt_s m + \pt_x(m(H[m]+\alpha)=0 \in \ \text{in} \ R\times (0,T),\\[1.2mm]
%m_0=\mu,
%\end{cases}
%\eeq
%and
%\beq\label{manos11} 
%\begin{cases}
%\pt_s m + \pt_x(m(H[m]+\alpha)=0 \in \ \text{in} \ R\times (0,T),\\[1.2mm]
%m_0=\mu \ \text{and} \ m_T=\nu.
%\end{cases}
%\eeq
%
%

\subsection{Translating the controls}
To make rigorous the idea of translating a control, we shall adopt a Lagrangian point of view, and we refer to \citep{tangent} for more on the topic. Hence, instead of working on the density of the eigenvalues $\mu \in \mcl P_2$, we shall work with a random variable $X$ of law $\mu$, and interpret $X(\omega)$ as the value of the eigenvalue labeled $\omega$. Furthermore, to fix ideas and to ease some arguments later on, we concentrate on the case of the probability space $([0,1],\text{Leb})$. We note that $\mathcal H=L^2([0,1]; \R)$ is the usual Hilbert space of square integrable real valued functions defined on $[0,1]$, and denote by $\mcl K$ the cone of $\mcl H$ of non-decreasing functions.
\vs
Given $(\al,m) \in \mcl A(t,\mu)$, we  consider a process $(Y_s)_{s \in [t,T]}$ solving 
$$
\frac{d}{ds} Y_s = \alpha(s,Y_s) + H[\mcl L(Y_s)](Y_s)  \text{ in } (t,T) \ \ \text{and} \ \ \mathcal L(Y_s) = m_s \ \ \text{for all} \ \ s \in [t,T].
$$
%such that, for all $s \in [t,T]$, $\mathcal L(Y_s) = m_s$. 
The existence of such a process is for instance given by Ambrosio's superposition principle, Theorem 8.2.1 in \citep{ags}, applied on $[t+ \eps,T]$, and then extended by taking the limit $\eps \to 0$. 
\vs
Let $\mcl E$ be defined on $\mcl H$ by $\mcl E(X) = E(\mcl L(X))$. It follows from Proposition \ref{prop:comp} that, for $\ep>0$,
 $$\sup_{ s \geq t +\eps}\mcl E(Y_s) \leq \oo.$$
 In particular, because $(Y_s)_{s \in [t,T]}$ is continuous in $\mcl H$, we know that,
 \[\text{ if $Y_t \in \mcl K$, then $Y_s \in \mcl K$ for all $s \in [t,T]$}.\]
\vs
Next we find a solution $(X_s)_{s \in [t,T]}$ to
$$
\frac{d}{ds} X_s = \alpha(s,Y_s) + H[\mcl L(X_s)](X_s)  \text{ in } (t,T),
$$
with the property that,
\[\text{if $X_t \in \mcl K$, then $X_s \in \mcl K$ for all $s \in [t,T]$}.\]

%and, to ease the arguments which are to come, we would also like to guarantee that if $X_t \in \mcl K$, then $X_s \in \mcl K$ for all $s \in [t,T]$.
%\vs
The existence of such process follows from the next result, which is stated in a more general way as it will also be useful in Section \ref{sec:valuevisc}.
\vs

\vs
\begin{lemma}\label{lemma:brezis}
Let $f,g \in L^2([0,t];\mathcal H)$ and $X_0 \in \mathcal K$.
% and $X_0 \in \mcl K$. 
 Then, for any $\eps>0$, there exists $X \in \mathcal C([0,t),\mcl H)$  with initial condition $X_0$ and such that 
\beq\label{odeH}
\frac{d}{ds} X_s = f(s) + H[\mcl L(X_s)](X_s) + \frac{g(s)-X_s}{t-s} \text{ in } (0,t) \ \ \text{and} \ \  X_s \in \mcl K \ \ \text{ for all} \ \  s \in [0,t).
\eeq
%with initial condition $X_0$, and which satisfies $X_s \in \mcl K$ for any $s \in [0,t)$.
\end{lemma}
\vs
Note that we do not assume that $f_s, g_s  \in \mcl K$. Furthermore, since we only look for solutions on $[0,t)$, we work on $[0,t-\delta]$ with $\delta>0$ and, hence, the singularity at the terminal condition $s = t$ does not appear, and we treat the term with $g$ as part of $f$.  The difficulty due  the singularity  at $s = t$ is  addressed in Section \ref{sec:valuevisc}. 
\vs
We present two proofs of this result since,  each of  them,  introduces some new ideas.
\begin{proof}[First proof] 
%\textcolor{blue}{The last term in \eqref{odeH} can be treated using the techniques we used below so we do not detail it here.} 
%\vs
We provide a proof by approximation by piecewise constant solutions defined, for each $N \geq 2$,  on fixed a uniform discretization of $[0,1]$ by $n$ subintervals $(I^N_j)_{1 \leq j \leq N}$. 
\vs
%For each $n \geq 2$,  we fix a uniform discretization of $[0,1]$ by $n$ subintervals $(I^n_j)_{1 \leq j \leq n}$, and 
Then we approximate $X_0, f$ and $g$ by  sequences $(X^N_0)_{ N \geq 2}$ in $\mcl K$, and $(f^N)_{ N \geq 2}$ and $(g^N)_{ N\geq 2}$ in $L^2([0,t);\mathcal H)$ respectively,  such that, for all $N$ and $s\in [0,t)$, $X^N_0$, $f^N_s$ and $g^N_s$  are constant on any of the partition intervals $I^N_j$, and converge, as $N \to \infty$,  toward $X_0$ in $\mcl H$, and $f$ and $g$  in $L^2([0,t];\mathcal H)$.

\vs
%Next we find $(f_n)_{ n \geq 2}$
%
%$g$ by $(g_n)_{ n \geq 2}$ which are constant, for almost every $s$, on each of the $I^n_j$ and approximate $g$ in $L^2([0,t];\mathcal H)$,  and, finally,  $f$  by $(f_n)_{ n \geq 2}$  which are constant, for almost every $s$, on each of the $I^n_j$ and approximate $f$ in $L^2([0,t];\mathcal H)$.
%

%%, $f$ and $g$ by piecewise constant functions on each $I^n_j$. 
%\vs
%
%Indeed, let $(f_n)_{ n \geq 2}$ and $(g_n)_{ n \geq 2}$ in $\in L^2([0,t];\mathcal H)$ be approximations of $f$ and $g$ which are constant, for almost every $s$, on each of the $I^n_j$, 
%
%We consider a sequence $(X^n_0)_{ n \geq 2}$ in $\mcl H$ such that, for all $n$, $X^n_0$ is constant on any of the $I^n_j$ and it converges, as $n \to \infty$,  toward $X_0$ in $\mcl H$and $(f_n)_{ n \geq 2}$ and $(g_n)_{ n \geq 2}$ in $\in L^2([0,t];\mathcal H)$ which are constant, for almost every $s$, on each of the $I^n_j$. 
%
%
% be the uniform discretization of $[0,1]$ into $n$ subintervals, and 
% consider a sequence $(X^n_0)_{ n \geq 2}$ in $\mcl H$ such that, for all $n$, $X^n_0$ is constant on any of the $I^n_j$ and it converges, as $n \to \infty$,  toward $X_0$ in $\mcl H$. Similarly, 
% we consider an approximation $f^n$ of $f$, which is also, for almost every $s$, is constant on each of the $I^n_j$. 
% 
 \vs
 
 We define, for each $N\geq 2$, $i\in \{1,\ldots, N\}$ and $s\in [0,t)$, 
 \beq\label{alex0}
 \lambda_0^{N,i}=X_0^N(x), \ \ f_s^{N,i}=f_s^{N,i}(x) \ \  \text{and}  \ \ g_s^{N,i}=g_s^{N,i}(x) \ \ \text{for some} \ \ x\in I^N_i,
 \eeq
 and consider the process $\big((\lambda^{N,j}_s)_{s\in [0,t)}\big)_{1\leq i\leq n}$ evolving by 
 \beq\label{alex1}
 \ds \frac{d \lambda^{N,i}_s}{ds}= f^{N,i}_s + \frac{1}{N}\sum_{j=1, j\neq i}^N \frac{1}{\lambda^{N,j}_s-\lambda^{N,i}_s} -\frac{g_s^{N,i}}{t-s}. 
 \eeq
 
We now claim %from standard arguments (see, for example, \citep{AGZ}) 
 that, for  there exists a unique process solving \eqref{alex0} and  \eqref{alex1} such that 
\beq\label{alex2}
\lambda^{N,i}_s<\lambda^{N,i+1}_s \ \ \text{for all} \ \ s\in [0,t),
\eeq
and
\beq\label{alex3}
\ds W_2\big(\frac{1}{N}\sum_{j=1}^N \delta_{\lambda^{N,j}_s}, \frac{1}{N}\sum_{j=1}^N \delta_{\lambda^{N,j}_{\hat s}}\big)\leq \omega(|s-\hat s|) \ \ \text{for all} \ \ s, \hat s \in [0,t), 
\eeq
for some modulus $\omega$ which depends on $\|f\|_{L^2([0,1]\times [0,1])}$ and $\|g\|_{L^2([0,1]\times [0,1])}$ and is  independent of $N$.
\vs

The existence and \eqref{alex2} imply the uniqueness. The existence follows by an  immediate adaptation to the deterministic setting of the  standard techniques used to handle the finite dimensional Dyson Brownian motion (see \citep{AGZ}). 
For  \eqref{alex2}, we  use
$$\ds I(s)=\frac{1}{N^2} \sum_{i, j=1, i\neq j}^N \log |\lambda^{N,i}_s - \lambda^{N,j}_s|$$
as a Lyapunov functional. Finally,  \eqref{alex3} is proved using arguments similar to the ones used 
in the proof of Proposition~\ref{takis1250}. 
 \vs
 
Next we reconstruct from $\big((\lambda^{N,j}_s)_{s\in [0,t)}\big)_{1\leq i\leq n}$ a process 
$(X^n_s)_{s \in [0,t)}$ such that 
\beq\label{zoe100}
((X^n_s)_{s \in [0,t-\delta]})_{n \geq 2} \ \ \text{ is uniformly continuous for any $\delta > 0$}, 
% This is a translation of the usual continuity estimate on the empirical measure of the Dyson Brownian motion.}
\eeq 
and
\beq\label{zoe101}
X^n_s \in \mcl K \ \ \text{for all $n \geq 2$ \ and \ $s \in [0,t)$,}
\eeq
which is a solution to
\[\frac{d}{ds} X^n_s = f^n(s) + H[\mcl L(X_s)](X_s) + \frac{g^n-X^n_s}{t-s} \text{ in } (0,t). \]

%and such that
%\beq\label{zoe100}
%((X^n_s)_{s \in [0,t-\delta]})_{n \geq 2} \ \ \text{ is uniformly continuous for any $\delta > 0$}, 
%% This is a translation of the usual continuity estimate on the empirical measure of the Dyson Brownian motion.}
%\eeq 
%and
%\beq\label{zoe101}
%X^n_s \in \mcl K \ \ \text{for all $n \geq 2$ \ and \ $s \in [0,t)$.}
%\eeq
%We note that \eqref{zoe100} is  a translation of the usual continuity estimate on the empirical measure of the Dyson Brownian motion, while \eqref{zoe101} results from the the fact that the particles do not cross.
%\begin{itemize}
%\item $((X^n_s)_{s \in [0,t-\delta]})_{n \geq 2}$ is uniformly continuous for any $\delta > 0$. This is a translation of the usual continuity estimate on the empirical measure of the Dyson Brownian motion.
%\item For all $n \geq 2$, $s \in [0,t)$ $X^n_s \in \mcl K$. This is the translation of the fact that the particles do not cross.
%\end{itemize}
%% (it is the usual fact that the particles do not cross for the Dyson Brownian motion)

% \eqref{odeH} which verifies: it is uniformly continuous in $\mcl H$ (it is the usual continuity argument for $W_2$ like distance) provided that the $L^2(ds\,d\omega)$ norm of $f^n$ is bounded by the one of $f$, it is non-decreasing in $\omega$ for any time (it is the usual fact that the particles do not cross for the Dyson Brownian motion). 

We can then pass to the $N \to \infty$ limit  to obtain a process $(X_s)_{s\in [0,t)}$ which satisfies all the requirements except for the fact that is solves \eqref{odeH}.  To complete the proof, we remark that we can indeed obtain \eqref{odeH} by passing to the limit in the  weak formulation of \eqref{odeH} involving smooth test functions, like is usually done using Ito's Lemma for stochastic processes.

\end{proof}

\begin{proof}[Second proof]
%This time also, the last term in \eqref{odeH} does not raise any particular issue. 
We fix $\delta > 0$ and we work on $[0,t-\delta]$. We omit the term $g_s/(t-s)$ as it can be added to the source term $f_s$. Then  \eqref{odeH} is a standard evolution equation associated to a maximal monotone operator, except for the fact that the term $-\frac{X}{t-s}$, which is of course associated to a monotone operator which is the gradient of the convex function $(t-s)^{-1}|x|^2$, depends explicitly on the time variable $s$. This does not raise any particular issue and we do not comment further on it. The rest of the equation also comes from the gradient of a convex function. Indeed, because of the displacement convexity of the entropy, we know that the function $$\phi:X \mapsto -\frac12 \int_0^1\int_0^1\log(|X(\omega) - X (\omega')|) d\omega d \omega' + \mathbb 1_{\mcl K}(X)$$ is convex on $\mcl H$. 
\vs
From standard arguments (see, for example,  Theorem 3.4 and Theorem 3.6 in \citep{brezis}) we obtain, for any $X_0$ in the domain of $\phi$, a strong solution $X: [0,t-\delta]\mapsto \mathcal H$ to  
$$
\frac{d}{ds} X_s + \partial \phi(X_s) \ni f(s),
$$
with the property that $X_s$ belongs to the domain of $\partial \phi $ for almost every $s \in [0,t-\delta]$. 
% we only consider strong solution of the previous evolution equation.

%Because, once the existence of such a weak solution is established, the estimates of Proposition \ref{prop:?} implies absolute continuity of $X$, and thus the fact that it is also a strong solution, we only consider strong solution of the previous evolution equation. 
%
\vs
It now remains to check that  process $X\in \mathcal C([0,t-\delta],\mcl H)$ is indeed a solution to \eqref{odeH}, that is, to establish that, 
if $X \in \mcl K$ is such that there exists $ \xi \in \partial \phi(X)$, then necessary $\xi = H[\mcl L(X)](X)$.
% Observe that, if such a requirement holds, then the result is proved. To establish this, we consider such a couple $(X,\xi)$. 
\vs
Observe first that $X_s$ is increasing. Indeed, it is non decreasing because it belongs to $\mcl K$, and is  actually increasing,  because otherwise its entropy would not be finite. 
\vs

Let $h: \R \mapsto \R$ be a Lipschitz function with constant $L$, and, for  $Y= h(X)$,  define $X_\eps = X + \eps Y$. For $|\eps|$ small enough, $X_\eps \in \mcl K$,  since $X_\eps(x) - X_\eps(x') \geq (1-\eps L)(X(x) - X(x'))$. 
\vs
Next we  compute
$$
\frac{\phi(X_\eps) - \phi(X)}{\eps} = -\frac{1}{2\eps} \int_{[0,1]^2}\log\left(1 + \eps\frac{h(X(x))-h(X(y))}{X(x) -X(y)}\right)dxdy.
$$
Since the quotient inside the $\log$ is uniformly bounded, we can apply the dominated convergence theorem to pass to the limit $\eps \to 0$ to obtain
$$
\lim_{\eps \to 0}\frac{\phi(X_\eps) - \phi(X)}{\eps}= -\frac{1}{2} \int_{[0,1]^2}\frac{h(X(x))-h(X(y))}{X(x) -X(y)}dxdy.
$$
But, because we can choose $\eps$ with both signs, we deduce from $\xi \in \partial \phi(X)$ that, for all $h$ as above,
$$
\langle \xi,h(X)\rangle = -\frac{1}{2} \int_{[0,1]^2}\frac{h(X(x))-h(X(y))}{X(x) -X(y)}dxdy.
$$
Since the right hand side of this last equality is the Hilbert transform of $X$, the result follows once we show that the set of such $h(X)$ is dense in $L^2([0,1];\R)$. Indeed, Lipschitz continuous functions are dense in $L^2(\R;\mathcal L(X))$, and using that $X$ is strictly increasing, we have that $X^{-1}$ is a well defined measurable map. It then follows  %from the density of the set of Lipschitz functions on $L^2(\R,\mathcal L(X))$ 
that the set  $\{f(X)|f \text{ Lipschitz}\}$ is indeed dense in $L^2([0,1],\R)$.
%from the one of the set of Lipschitz functions on $L^2(\R,\mathcal L(X))$. 
\vs
Observe in particular that there is no component in the gradient arising from $\mathbb 1_{\mcl K}$, and that it is an important consequence of the fact that $X$ is increasing.

\end{proof}

We believe that this approach on the limit Dyson equation, or more generally on interacting particle systems could have an interest outside of the present paper. Note that techniques of evolution equation associated to monotone operators were also used, at the $N$ particle level, by C\'epa and L\'epingle in \citep{CL}.

\section{The value function of the controlled Dyson equation}\label{sec:bound}

We introduce  here  the value function of the controlled Dyson equation \eqref{takis3}, develop several important facts and make the connection with  a backward in time Hamilton-Jacobi-Bellman equation. %in $\Pt\times [0,T)$. 
\vs

%We consider two types of terminal  data $U_T$.  
%The first is ``regular'', that is $U_T \in C(\mcl P_2)$. The second corresponds to the planning problem, that is, to driving $\mu \in \Pt$ to $\nu \in \mcl P_2$, in which case the terminal condition  is  $U_T=\mathbb{1}_{\nu}$. Finally, we consider a running cost $F\in BC(\Pt)$.
%\beq\label{takis29}
%U_T=\mathbb{1}_{\nu}.
%\eeq
%where $\mathbb{1}_{\nu}(m)=0$, if $m=\nu$, and   $\infty$ otherwise.
%\vs

%When the terminal data is in $\Pt$, the value function $U : [0,T)\times \Pt\times \times [0,T)  \R$ is 

%The functional for  controlling \eqref{takis32.1} is 
%\beq\label{takis31}
%J[m,\al]=\dint_0^T  \dint [\frac{1}{2} |\al|^2 m dx]  + F[m]  dt  +U_T(m(T)).
%\eeq

\subsection{The regular terminal condition} 
%For each $(\mu, t)\in \Pt \times (0,T]$, we  consider the set of admissible contros  $\mcl A(\mu, t)$  introduced in section~3,  
%the set of pairs $(\al,m)$ such that $\al: \R\times [t,T]\to \R$ is measurable, $m\in C([t,T]; \Pt)$, % m_t=\mu$
%%and $(\al,m)$  is a weak solution to % is the solution to \eqref{takis32.1}.
%%\beq\label{takis30.1}
%%\pt_s m + \pt_x \big(m_s\big(H[m_s] +\al_s\big)\big)=0  \ \ \text{in} \ \ \R\times (t,T) \ \ \text{and} \ \ m_t=\mu.
%%\eeq

For each admissible pair $(\al,m)\in \mcl A(\mu,t)$--recall that  $\mcl A(\mu,t)$ was introduced in Section~3--the payoff $P(\al,m,t)$ of the ``regular'' controlled Dyson problem is 
\beq\label{payoff1}
P(\al,m,t)=\dint_t^T \big [\dint_\R  \frac{1}{2} \al(x,s)^2 dm_s(x)  + F[m_s]\big]  ds  +U_T(m_T),
\eeq
and the value function $U :  \Pt \times [0,T] \to  \R$ is 
\beq\label{takis30}
U(\mu, t)=\underset{(\al,m)\in A(\mu,t)} \inf P(\al,m,t) =\underset{(\al,m)\in \mcl A(\mu,t)} \inf \Big\{\dint_t^T \big [\dint_\R  \frac{1}{2} \al(x,s)^2 dm_s(x)  + F[m_s]\big]  ds  +U_T(m_T) \Big\}. 
\eeq
\vs
Let $\ol {\mcl A}(\mu,t)$ be the set  of pairs $(\beta,m)$ such that  $m\in \mathcal{C}([t,T],\Pt)$ and $\beta$ 
is a measure on $\R\times [t,T]$ which, for each $s\in[t,T]$, is absolutely continuous with respect to $m_s$ and, in the weak sense, 
% m(t) = \mu$,
%$\beta <<m$, and $(\beta,m)$ satisfies in the weak sense 
\be\label{dysoncontr2}
\partial_s m_s + \partial_x (\beta + m_sH[m_s]) = 0 \  \text{ in } \  \R \times (t,T)  \ \ \text{and} \ \ m_t=\mu.
\ee
Writing  $\beta=\al m$ in ${\mcl A}(\mu,t)$ we see that  $U(\mu,t)$ can be equivalently defined by
%the following equivalent rep
%Both value functions  can be written in an equivalent form with the minimization taking place over a convex set. Indeed,
\beq\label{takis31.1}
U(\mu,t)=\underset{(\beta,m) \in \ol {\mcl A}(\mu,t)}\inf \Big\{\dint_t^T\left[\dint_\R \dfrac{1}{2} \dfrac{\beta(x,s)^2}{m_s(x)} dx + F[m_s]\right] ds + U_T(m_T) \Big\}.
\eeq
%where $\ol {\mcl A}(\mu,t)$ is the set  of pairs $(\beta,m)$ such that $\beta:\R\times [t,T]$ is measurable,  $m\in \mathcal{C}([t,T],\Pt)$,
%% m(t) = \mu$,
%$\beta <<m$, and $(\beta,m)$ satisfies in the weak sense
%\be\label{dysoncontr2}
%\partial_s m + \partial_x (\beta + mH[m]) = 0 \  \text{ in } \  \R \times (t,T) \ \  \text{and} \ \ m_t=\mu,
%\ee

\subsection{The target problem}
%The ``target'' Dyson control problem is to ``connect''  two measures  $\mu \in \Pt$ and  $\nu \in \widetilde \Pt$ over the time interval $[t,T]$ using admissible controls 
For $t \in [0,T)$, $\mu,\nu \in \Pt$, we introduce $\mcl A(\mu,\nu,t) = \{ (\alpha,m) \in \mcl A(\mu,t) | m_T = \nu\}$. For each pair 
$(\alpha,m) \in \mcl A(\mu,\nu,t)$
%of admissible controls  consists   of pairs $(\al,m)$ such that $\al: \R\times [t,T]\to \R$ is measurable, $m\in C([t,T]; \Pt)$, %m(t)=\mu, m(T)=\nu$, 
%%and $(\al,m)$ is  a weak  solution to %qref{takis30.1}.
%%\beq\label{takis30.5}
%%\pt_s m_s + \pt_x \big(m_s\big(H[m_s] +\al\big)\big)=0  \ \ \text{in} \ \ \R\times (t,T) \ \ \text{and} \ \ m_t=\mu \ \text{and} \  m_T=\nu.
%%\eeq
%%
%%We introduce next the pay-off functional of the controlled Dyson problem which has two forms depending on whether the terminal datum  $U_T\in \Pt$ is regular, that is, continuous, or singular, that is, $U_T=\mathbb{1}_{\nu }.$
%
%%for $\nu\in \Pt$, 
%%\beq\label{takis29}
%%U_T=\mathbb{1}_{\nu },
%%\eeq
%%where $\mathbb{1}_{\nu}(m)=0$, if $m=\nu$, and   $\infty$ otherwise.
%\vs
%
%%In the former case, for each $(\al,m)\in \mcl A(\mu,t)$, the payoff is 
%%\beq\label{payoff1}
%%J(\al,m,t)=\dint_t^T \big [\dint_\R  \frac{1}{2} |\al |^2 m dx  + F(m(s))\big]  ds  +U_T(m_T),
%%\eeq
%%while for 
%For each $(\al,m) \in \mcl A(\mu, \nu, t)$, 
the payoff of the target problem is
\beq\label{payoff2}
\widehat P(\al,m, t)=\dint_t^T \left [\dint_\R  \frac{1}{2} \al(x,s)^2 \;dm_s(x)  + F[m_s]\right]  ds.
\eeq
%The value function 
%When the terminal condition is regular,  the value function $U :  \Pt \times \times [0,T) \to  \R$ is 
%%\beq\label{takis30}
%%U(\mu, t)=\underset{(\al,m)\in \mcl A(\mu,t)} \inf \Big\{\dint_t^T \big [\dint_\R  \frac{1}{2} |\al|^2 m dx  + F(m(s))\big]  ds  +U_T(m(T)) \Big\},
%%\eeq
%\beq\label{takis30}
%U(\mu, t)=\underset{(\al,m)} \inf J(\al,m) =\underset{(\al,m)\in \mcl A(\mu,t)} \inf \Big\{\dint_t^T \big [\dint_\R  \frac{1}{2} |\al|^2 m dx  + F(m(s))\big]  ds  +U_T(m(T)) \Big\},
%\eeq
%for each $m_0\in \mcl P_2$, $F \in C(\mcl P_2)$ and either $U_0 \in C(\mcl P_2)$ or, for  $m_1\in\mcl P_2$,
%\beq\label{takis29}
%U_0=\mathbb{1}_{m_1},
%\eeq
%where $\mathbb{1}_{m_1}(m)=0$ if $m=m_1$ or $\infty$ otherwise, 
%the variational problem 
%\beq\label{takis30}
%\underset{m,\al} \inf \big\{J[m,\al] \; : (m,\al) \in \mcl A(\mu, t) \big\},
%\eeq
%where $\mcl A(\mu, t)$ is the set of pairs $(\al,m)$ such that $\al: \R\times [0,T]\to \R$ is measurable, $m\in C([0,T]; \Pt), m(t)=\mu$
%and $(\al,m)$  is a weak solution to % is the solution to \eqref{takis32.1}.
%\beq\label{takis30.1}
%\pt_s m + \pt_x \big(m\big(Hm +\al\big)\big)=0  \ \ \text{in} \ \ \R\times (t,T).
%\eeq
%When the terminal condition is \eqref{takis29},
Since, in view of Proposition \ref{prop:comp}, $\mcl A(\mu,\nu,t) = \emptyset$ if $E(\nu) = \oo$, 
%for $\mu,\nu \in \mcl P_2$, $t < T$ and $\mcl A(\mu,\nu,t) = \emptyset$ if $E(\nu) = \oo$. 
we restrict our attention to target $\nu \in \tilde \Pt$. 
\vs
The value function $\widehat U : \Pt \times \Pt \times [0,T] \to  \R$ is 
\beq\label{takis30.2}
\widehat U(\mu, \nu, t)=\underset{(\al,m) \in \mcl A(\mu,\nu,t)} \inf \widehat P(\al, m,t)=\underset{(\al,m) \in \mcl A(\mu,\nu,t)} \inf \Big\{\dint_t^T \left [\dint_\R  \frac{1}{2} \al(x,s)^2 dm_s(x)  + F[m_s]\right]  ds  \Big\},
\eeq
and, equivalently, 
\beq\label{takis31.2}
\widehat U(\mu,\nu, t)=\underset{(\beta, m) \in \ol {\mcl A}(\mu,\nu, t)}\inf \Big\{\dint_t^T\left[\dint_\R \dfrac{1}{2} \dfrac{\beta(x,s)^2}{m_s(x)} dx + F[m_s]\right] ds\Big\}, 
\eeq

where $\ol {\mcl A}(\mu,\nu,t)=\{(\beta,m) \in \ol {\mcl A}(\mu,t) \; : m_T=\nu\}$.
% is  the set  of pairs $(\beta,m)$ such that  $m\in \mathcal{C}([t,T],\Pt)$,   $\beta$ is a measure on $\R\times [t,T]$ which, for each $s\in[t,T]$, is absolutely continuous with respect to $m_s$ and, in the weak sense, 
%% m(t) = \mu$,
%%$\beta <<m$, and $(\beta,m)$ satisfies in the weak sense 
%\be\label{dysoncontr2}
%\partial_s m_s + \partial_x (\beta + m_sH[m_s]) = 0 \  \text{ in } \  \R \times (t,T) \ \ \text{and} \ \ m_t=\mu \ \text{and} \  m_T=\nu.
%\ee

%
%$\beta:\R\times [0,t]\to \R$ is measurable, 
%$\beta <<m$, 
%and  $(\beta,m)$ satisfies %\eqref{dysoncontr2}  
%in the  weak sense 
%\be\label{dysoncontr3}
%\partial_s m_s + \partial_x (\beta + m_sH[m_s]) = 0 \  \text{ in } \  \R \times (0,t) \ \ \text{and} \ \ m_0=\mu \ \  m_t=\nu,
%\ee
%then                        
%\beq\label{takis31.2}
%\widehat U(\mu,\nu, t)=\underset{(\beta, m) \in \ol {\mcl A}(\mu,\nu, t)}\inf \Big\{\dint_0^t\big[\dint_\R \dfrac{1}{2} \dfrac{\beta^2}{m} dx + F(m(s))\big] ds\Big\}.
%\eeq
%where $\ol {\mcl A}(\mu,\nu, t)$ is the set  of pairs $(\beta,m)$ such that $\beta:\R\times [t,T]$ is measurable, $m\in \mathcal{C}([t,T],\Pt)$,
%%m(t) = \mu, m(T)=\nu$, 
%$\beta <<m$, 
%and  $(\beta,m)$ satisfies %\eqref{dysoncontr2}  
%in the  weak sense 
%\be\label{dysoncontr3}
%\partial_s m + \partial_x (\beta + mH[m]) = 0 \  \text{ in } \  \R \times (t,T) \ \ \text{and} \ \ m_t=\mu \ \  m_T=\nu.
%\ee

\vs

\subsection{The Hamilon-Jacobi-Bellman equation and dynamic programming principle}
The value functions $U$ and $\widehat U$ are  expected to be respectively the unique viscosity solutions to 
\beq\label{HJB1}
\begin{cases}
-\partial_t U + \dfrac 12 \dint_\R(D_\mu U)^2d \mu - \dint_\R D_\mu U H[\mu] d\mu  = F[m] \  \text{ in } \  \Pt \times [0,T),\\[1.5mm]
U(\cdot,T)=U_T, 
\end{cases}
\eeq
and 
\beq\label{HJB2}
\begin{cases}
-\partial_t \widehat U + \dfrac 12 \dint_\R(D_\mu \widehat U)^2d \mu - \dint_\R D_\mu \widehat U H[\mu] d\mu  = F[m] \  \text{ in } \  \Pt \times [0,T),\\[1.5mm]
\widehat U(\cdot,T)=\mathbb{1}_\nu \ \text{on} \ \Pt.
\end{cases}
\eeq
These two equations are the HJB equations of the optimal control problems at hand. They are an infinitesimal version of the dynamic programming principle (DPP), which trivially holds in this deterministic setting. Stating it, for instance,  for $U$,  the DPP says that, for every $(\mu,s)\in \Pt\times [t,T]$, 
$$
U(\mu,t) = \inf_{(\alpha,m) \in \mcl A(\mu,t)} \left\{ \int_t^s\left[\int_\R \frac12  \alpha(x,u)^2m_u(dx)+ F[m_u] \right] du + U(m_s,s) \right\}.
$$
%Using the formula for $ \| \pt_{\mu} U \|^2$ in the introduction, we remark that we can rewrite the equations in \eqref{HJB1} and \eqref{HJB2} as 
%%\beq\label{HJB1}
%%\begin{cases}
%\[-\partial_t U + \dfrac 12 \dint_\R(D_\mu U)^2d \mu - \dint_\R D_\mu U H[\mu] d\mu  = F[m] \  \text{ in } \  \Pt \times [0,T),\]
%%\\[1.2mm]
%%U(\cdot,T)=U_T \ \text{on} \ \Pt,
%%\end{cases}
%%\eeq
%%\beq\label{HJB2}
%%\begin{cases}
%and 
%\[
%-\partial_t \widehat U + \dfrac 12 \dint_\R(D_\mu \widehat U)^2d \mu - \dint_\R D_\mu \widehat U H[\mu] d \mu  = F[m] \  \text{ in } \  \Pt \times [0,T).
%\]
%\\[1.2mm]
%\widehat U(\cdot,T)=\mathbb{1}_\nu \ \text{on} \ \Pt.
%\end{cases}
%\ee
%where $\kappa>0$ is a constant determined by the data.
\vs
The definition of the viscosity solution as well as the precise way in which $\widehat U$ satisfies the terminal condition are   explained later in the paper. 

\vs

In order to establish the well-posedness of solutions to \eqref{HJB1} and \eqref{HJB2} we need (i)~to investigate whether and under what conditions on $\mu$ and $\nu$  the value  functions are finite, (ii)~to give a definition for viscosity solutions to \eqref{HJB1} and \eqref{HJB2} (see Section \ref{sec:visc}), and, finally, (iii)~to show  that $U$ and $\widehat U$ are  viscosity solutions  to \eqref{HJB1} and  \eqref{HJB2} (see Section \ref{sec:valuevisc}).
% this is the subject of Section \ref{sec:valuevisc}. 
%and to establish their continuity. 
%The second is about the conditions needed for the controls $\al$ and/or $\beta$  for   the forced Dyson equation to have a weak solution. 
%The second is to define viscosity solutions to \eqref{HJB1} and \eqref{HJB2}, which is done in Section \ref{sec:visc}. %which is done in Section \ref{sec:visc}.The third  is to justify that the value functions 
%Finally, it is shown in Section \ref{sec:valuevisc} that $U$ and $\widehat U$ are  viscosity solutions  to \eqref{HJB1} and  \eqref{HJB2}.
% this is the subject of Section \ref{sec:valuevisc}. 
%An important step in this direction is to establish  $U$  and ref{HJB2}. 
%An important step in this direction is to establish  that  $U$  and  $\widehat U$ are   continuous  in $\Pt\times [0,T]$  and  $\Pt\times\Pt\times [0,T)$ respectively as well as  the behavior of $\widehat U$ as $t\to T$. 
% as well as to understand the precise way  that $\widehat U$ attains the initial  condition. 

\vs
\subsection{An alternative formulation}
Before establishing bounds on the value function, we present another formulation for the value function which is equivalent to \eqref{takis30} and \eqref{takis30.2} when the free entropy of the initial condition of the controlled Dyson equation is finite. 
% and allows to minimize a convex functional over a convex set of constraints. 
%To simplify the notation, we write in what follows 
%\beq\label{665}
%c_0=\dfrac{\pi^2}{3}
%\eeq
%for the constant in Cotlar's identity.

%The advantage of this equivalent formu
%
%Using the free entropy \eqref{entropy} it is possible to provide another formulation for the value functions 
%We introduce next the entropy of the physical system associated to the Dyson system, which is the functional $E:\mcl P_2\to \R$ given %by 
%\beq\label{takis33}
%E[m]=-\frac{1}{2}\dint \dint log (|x-y|) dm(x) dm(y). 
%\eeq
%It turns out that, although $E$ may not be nonnegative, it is bounded from below. Moreover,  in \citep{BLL} it is proven that 
%\beq\label{takis34}
%E[m]<\infty \ \text{ if and only if} \ m \in H^{-1/2}.
%\eeq

\vs

We begin with \eqref{takis30},  assume that 
%with the observation that   the infimum in \eqref{takis30} % and \eqref{takis30.2}  
%is finite if
%It is also shown in  \citep{BLL}  that $E$ is decreasing along the Darcy flow, that is, for \eqref{takis32} with $\al\equiv 0$. This suggests that for the finiteness of the infimum in \eqref{takis30} it is necessary to assume that 
\beq\label{takis35}
E(\mu)<\infty,
\eeq
write  $\beta=Hm +\al$ and note that $(\beta,m)$ satisfies 
\beq\label{takis36}
%\begin{cases}
\pt_s m_s + \pt_x (m\beta)=0  \ \text{in} \ \R\times (t,T) \ \ \text{and} \ \  m_t=\mu.
%\end{cases}
\eeq
Then, 
%\beq\label{takis37}
\[
P(\al, m, t)= \dint_t^T \Big[ \dint_\R \left[ \frac{1}{2}  (\beta^2  + H[m_s]^2 ) -\beta m H[m] \right] dm_s(x)  + F[m_s] \Big] ds +U_T(m_T), % - \beta m Hm  dxdt,
\]
%\eeq
and, after using Cotlar's identity, % and \eqref{takis36}  and assuming that $m\in L^3$, we find 
\beq\label{takis38}
P(\al,m,t )= \dint_t^T \Big[ \dint_\R \left[\frac{1}{2} m \beta^2  +\dfrac{c_0}{6}  m^3  -\beta m H[m]\right] dx   + F[m_s[ \Big ] ds +U_T(m_T). 
\eeq
% - \beta m Hm  dxdt,
%\eeq
%where
%\beq\label{takis38}
%\tilde J[m,\beta] = \dint_0^T [ \dint (\frac{1}{2} m |\beta|^2  + c_0 m^3 )dx  + F[m(t)] ] dt +U_0(m(T).
%\eeq 
Multiplying \eqref{takis36} by $-\log  \star m$ and integrating over $\R$ we find 
\beq\label{takis39}
\dfrac{d}{ds} E(m_s)= \dint_\R \beta H[m_s] m_s dx.
\eeq
Inserting \eqref{takis39} in \eqref{takis38} gives %and using Cotlar's identity gives 
\beq\label{takis40}
P(\al,m,t)=\dint_t^T \Big[ \dint_\R \big[\frac{1}{2} m \beta^2  + \dfrac{c_0}{6}  m^3 \big ]dx - \frac{d}{ds} E(m_s)  + F[m_s] \Big] ds +U_T(m_T),
\eeq
and, finally, 
%\beq\label{takis41}
\[
P(\al,m,t)=\widetilde P(\beta,m,t)+ E[\mu],
\]
%\eeq
with
\beq\label{takis42}
\widetilde P(\beta,m,t)=\dint_t^T \Big [ \dint_\R\left[\frac{1}{2} m \beta^2  + \dfrac{c_0}{6} m^3\right]dx  + F[m_s] \Big ] ds +U_T(m_T)-E(m_T).
\eeq

%\beq\label{takis42}
%\widetilde J{m,\beta)=\dint_0^T [ \dint (\frac{1}{2} m |\beta|^2  + c_0 m^3)dx  + F[m(t)] ] dt +U_0(m(T))-E[m(T)].
%\eeq
\vs
Since
%\beq\label{takis43}
\[ \left|\dint_\R \beta H[m] m dx\right| \leq \frac{1}{2} \dint_\R \beta^2 m dx+ \frac{1}{2} \dint_\R H[m]^2 m dx, \]
%\eeq
it follows %from \eqref{takis39} 
that, if  
\[\dint_t^T\dint_\R \big[\beta^2 m + H[m]^2 m \big] dx <\infty,\] 
something which is true if, for example, $\beta\equiv 0$ and $m\equiv \mu \in L^3$, then 
\[\inf_{(\beta, m) \in \mcl B(\mu,t)} \widetilde P(\beta,m,t) <\infty.\]
\vs
Hence,
\[U(\mu,t)=\wtilde U(\mu, t) +E(\mu)\] 
where
\beq\label{takis679}
\wtilde U(\mu, t)=\underset{(\beta, m) \in \mcl B(\mu,t)} \inf \wtilde P(\beta, m,t),
\eeq
with $\mcl B(\mu,t)$ the collection of pairs $(\beta, m)$  such that $\beta:\R\times (t,T)\to \R$ is measurable,  $(m_s)_{t\leq s\leq T} \in C([t,T]; \Pt)$,  $m_t=\mu$ and $(\beta,m)$ solves 
\eqref{takis36}.
\vs

Similarly, 
if $E(\mu)<\oo$ and $E(\nu)<\oo$, then
\[\what U(\mu, \nu,t)=\wtilde{\widehat U}(\mu,\nu, t) +E(\mu)\]
where
\[\wtilde{\widehat U}(\mu,\nu, t)=\underset{(\beta, m) \in \mcl B(\mu,\nu, t)} \inf \wtilde{\widehat P}(\beta, m,t),\]
with 
\[\wtilde{\widehat P}(\beta, m,t)=\dint_t^T \Big [ \dint_\R\big[\frac{1}{2} m \beta^2  + \dfrac{c_0}{6} m^3\big]dx  + F[m_s] \Big ] ds -E(\nu),\]
and $\mcl B(\mu,\nu, t)$ the collection of pairs $(\beta,m)$ of measurable maps $\beta:\R\times (t,T)\to \R$ and paths  $(m_s)_{t\leq s\leq T} \in C([t,T]; \Pt)$ such that $m_t=\mu$ and $m_T=\nu$ solving \eqref{takis36}.
\vs

The HJB equation associated to this optimal control problems is naturally
$$
-\partial_t \wtilde U + \dfrac 12 \dint_\R(D_\mu \wtilde U)^2d \mu  = F[m] + \frac{c_0}{6}\|m\|_{L^3} \  \text{ in } \  \Pt \times [0,T),
$$

\subsection{Bounds on the value function}

We investigate here bounds for $U(\cdot,t)$ and $\widehat U(\cdot, \cdot, t)$. To simplify statements, we assume that $$ U_T \leq  \mathbb 1_{\nu},$$ which is equivalent to $U_T(\nu) \leq 0$.

\begin{proposition}\label{prop:bounds}
For any $h >0,  t \in [0,T-h), \mu \in \Pt$,  $ \nu \in \widetilde \Pt$ and $U_T: \Pt \mapsto \R$, bounded from below by $-K$, there exists a constant $C>0$ depending only on $T-h, \|F\|_\oo$ and $M_2(\nu)$ such that 
\beq\label{takis410}
-K - (T-t) \|F\|_\oo \leq U(\mu, t)\leq \widehat U(\mu,\nu,t) \leq C\Big(U_T(\nu)+(T-t) \| F \|_\oo|+ W^2_2(\mu, \nu) + |E(\nu)|\Big).
\eeq
% wit the constant  $C$ depends only on $T-\kappa$ and $M_2(\nu)$.
Moreover, if  $\nu=\mu$ and $E(\nu)<\oo$, then
\beq\label{takis299.10}
\widehat U(\nu, \nu,t)\leq   2\big[E(S_{\frac{T-t}{2}}\nu)-E(\nu)\big] +  (T-t) \|F\|_\oo.
\eeq
%Moreover, if $\nu \notin \wtilde \Pt$, then, for all $t<T$ and $\mu \in \Pt$, $\widehat U(\mu,\nu,t) =  \infty$.
Finally, 
\beq\label{takis246}
\widehat U(\mu,\nu,t)<\infty \ \ \text{if and only if} \ \ E(\nu)<\oo
\eeq
\end{proposition}

\begin{proof}
In the proof we assume, without any loss of generality, that $F=0$. Moreover, to simplify the argument,

\vs

The inequality between $U(\mu,t)$ and $\widehat U(\mu,\nu,t)$ as well as the lower bound in \eqref{takis410} are immediate from the definitions.
\vs
For the upper bound, we fix  $\ep \in (0, \min((T-t)/2, h)]$, write  $m_{t+\epsilon} = S_\epsilon\mu$ and $m_{T-\epsilon} = S_\epsilon\nu$, and recall that 
\[
\|m_\epsilon\|_\infty +  \|m_{T-\epsilon}\|_\infty \leq 2\ep^{-1/2}, 
\ \  \dint_\R x^2dm_\eps(x) \lesssim \dint_\R x^2 d\mu(x) \ \ \text{and} \ \  \dint_\R x^2m_{T-\eps}(dx) \lesssim  \dint_\R x^2 d\nu(x).
\]

Let  $(\wtilde \alpha,\rho)$  be an optimal pair for the optimal transport of $m_\eps$ to $m_{T - \eps}$ in time $T-t - 2 \eps$, that is, 
\[
\partial_s \rho + \partial_x(\wtilde \alpha \rho) = 0 \ \ \text{ in } \R\times (t+\ep, T-\ep), \quad \rho_{t+\epsilon} = m_\epsilon \ \ \text{and} \ \ \rho_{T-\epsilon} = m_{T-\epsilon},\]
and 
\[
\dint_{t+\epsilon}^{T-\epsilon} \dint_\R \wtilde \al(x,s)^2 d\rho_s(x)ds = \frac{W_2^2(m_\epsilon,m_{T-\epsilon})}{T-t -2 \epsilon}.
\]

\vs

Define next the maps $\alpha : \R \times  [t,T] \to \R$ and $m:\R \times [t,T]\to \R$ by
\[
\alpha(x,s) = \begin{cases} 0 \ \  \text{ if } \ \ s <t+ \epsilon,\\[1.2mm]
\tilde \alpha(x,s) - H[\rho_s](x) \  \  \text{ if } \ \  s \in [t+\ep, T-\epsilon], \\[1.2mm]
-2H[S_{T-s}\nu](x) \ \  \text{ if } \ \  T-\ep<s<T, \end{cases}
\]
and 
\[
m_s = \begin{cases} S_s\mu \ \   \text{ if } \ \  s <t+ \epsilon,\\[1.2mm]
\rho_s \  \ \text{ if } \  \ s \in [t+\epsilon, T-\epsilon], \\[1.2mm]
S_{ T -s}\nu \  \ \text{ if } \ \ T-\ep<s<T. \end{cases}
\]
Then, $(\alpha,m) \in \mathcal{A}(\mu,\nu,t)$ and 

\[
\bal
\dint_t^T \dint_\R \alpha(x,s)^2 dm_s(x)ds &= \dint_{t+\epsilon}^{T-\epsilon}(\wtilde \al (x,s) - H[\rho_s](x))^2d\rho_s(x) + 4\dint_0^\epsilon H[S_s\nu]^2(x)S_s d\nu(x)ds\\[1.5mm]
& \leq 2\dint_{t+\ep}^{T-\ep}\dint_\R \wtilde \al^2(x,s) d\rho_s(x) ds + 2\dint_{t+\ep}^{T-\ep}\dint_\R H[\rho_s]^2(x) d\rho_s(x) +4(E(S_\epsilon\nu) - E(\nu))\\[1.5mm]
& \leq 2  \dfrac{W_2^2(m_\epsilon,m_{T-\epsilon})}{T-t -2 \epsilon} + \dfrac{2c_0}{3} \dint_{t+\epsilon}^{T-\epsilon}\rho_s^3(x)dxds + 4(E(S_\epsilon\nu) - E(\nu)).
\eal
\]
\vs
The first term of the last inequality above comes from the definition of $(\wtilde \al,\rho)$, the second from Cotlar's identity and the third from the properties of the Dyson flow in $(T-\ep, T)$.
\vs

The upper bound in \eqref{takis410} follows from the inequalities
\[\sup_{s \in [ t+\eps,T-\eps]}\|\rho_s\|_3^3 \leq \eps^{-1},\]
which is a consequence of Proposition~\ref{prop:displace} and the regularizing properties of the Dyson flow, and 
\[W_2^2(m_\epsilon,m_{T-\epsilon}) \leq W_2^2(\mu,\nu),\]
that follows from the fact that the Dyson flow is a contraction for the Wasserstein distance. 
%
%is due to  the fact that, since  the Dyson flow is a contraction for the Wasserstein distance, we have 
%\[W_2^2(m_\epsilon,m_{T-\epsilon}) \leq W_2^2(\mu,\nu),\]
%and the inequality 
%\[\sup_{s \in [ t+\eps,T-\eps]}\|\rho_s\|_3^3 \leq \eps^{-1},\]
%which is a consequence of Proposition~\ref{prop:displace} and the regularizing properties of the Dyson flow. 
\vs

The argument leading to  \eqref{takis410} can be simplified when $\nu=\mu$ and $E[\mu]<\oo$.
Indeed to get \eqref{takis299.10}, we do not need to divide $(t,T)$ in three subintervals, but rather choose $\ep=(T-t)/2$, so that $t+\ep=T-\ep=\frac{t+T}{2}$ and skip the step about the optimal  transportation. 
\vs

Let 

\qquad \qquad $
\alpha(x,s) = \begin{cases} 0 \ \text{ if } \ t<s < \frac{t+T}{2},\\[1.2mm]
%\tilde \alpha(x,s) - H[\rho_s](x) \  \text{ if } \  s \in [t+\ep, T-\epsilon], \\[1.2mm]
-2H[S_{T-s}\nu](x) \ \text{ if } \ \frac{t+T}{2}<s<T, \end{cases}$ \ \ and \ \ $m_s = \begin{cases} S_s\mu \  \text{ if } \  t\leq s <\frac{t+T}{2},\\[1.2mm]
%\rho_s  \ \text{ if } \  s \in [t+\epsilon, T-\epsilon], \\[1.2mm]
S_{ T -s}\nu  \ \text{ if } \ \frac{t+T}{2}\leq s\leq T. \end{cases}$

\vs

%and 
%\[
%m_s = \begin{cases} S_s\mu \  \text{ if } \  t\leq s <\frac{t+T}{2},\\[1.2mm]
%%\rho_s  \ \text{ if } \  s \in [t+\epsilon, T-\epsilon], \\[1.2mm]
%S_{ T -s}\nu  \ \text{ if } \ \frac{t+T}{2}\leq s\leq T. \end{cases}
%\]
Since $\mu=\nu$, $(\alpha,m) \in \mathcal{A}(\nu,\nu,t)$, and, in view of the choice of $(\alpha,m)$ and Proposition~\ref{takis502}, we have

\[
%\ba
\dint_t^T \dint_\R\dfrac{1}{2}\al^2_s(x)m_s(x) dx = 2 \dint_{\frac{t+T}{2}}^T \dint_\R \big(H[S_{T-s} \mu]\big)^2 S_{T-s} \mu dx ds \leq 2\big[E(S_{\frac{T-t}{2}}\nu)-E(\nu)\big] + \|F\|_\oo(T-t).%\\[1.2mm]
%&\leq 2\big[E(S_{\frac{T-t}{2}}\nu)-E(\nu)\big] + \|F\|_\oo(T-t).
%\ea
\]

\vs

Finally, if $\widehat U(\mu,\nu,t)<\infty$, it was already shown earlier that, for each $h\in (t,T)$, $\underset{s\in [h,T]}\sup E(m_s)\lesssim_h 1,$ and, hence, $E(\nu)<\oo$. The other direction of \eqref{takis246} follows from \eqref{takis410} and \eqref{takis299.10}.

%there exists $C_0(h)$ such that 
%$$\underset{s\in [h,T]}\sup E(m_s)\leq C_0(h),$$ and, hence, $E(\nu)<\oo$.
%\vs
%
%The other direction of \eqref{takis246} follows from \eqref{takis410} and \eqref{takis299.10}.

\end{proof}

Finally, we  provide a stronger bound from below on the value function when the terminal cost is bounded from below by a quadratic function. 
\begin{proposition}
Assume that there exists $K > 0$ and $\nu \in \tilde{\mcl P_2}$ such that,  for all $\mu \in \mcl P_2$, $U_T(\mu) \geq KW_2^2(\mu,\nu)$. Then, there exists $C > 0$ that depends only on $T$ and $K$ such that, for all $t \in [0,T]$ and $\mu \in \mcl P_2$,
$$
U(\mu,t) \geq C^{-1}W_2^2(\mu,\nu) - C.
$$
\end{proposition}
\begin{proof}
Let $t \in [0,T), \mu \in \mcl P_2$ and consider an optimal control $(\alpha,m)$ for $U(\mu,t)$. Then
$$
\frac{d}{dh}M_2(m_{t+h})= 1 + 2 \int_{\R}\alpha(t+h,x)x dm_{t+h}(x),
$$
and, after employing Cauchy-Schwarz inequality, % and integrating we deduce 
$$
\sqrt{M_2(m_T)} - \sqrt{M_2(\mu)} \geq -\int_t^T\|\alpha\|_{L^2(m_s)}ds \geq -\sqrt{T-t}\sqrt{\int_t^T\|\alpha\|^2_{L^2(m_s)}ds}.
$$
Moreover, using twice the triangular inequality, we find 
$$
W_2(\mu,\nu) \leq W_2(m_T,\nu) + \Big((T-t)\int_t^T\|\alpha\|^2_{L^2(m_s)}ds \Big)^{1/2} + 2 \Big(M_2(\nu)\Big)^{1/2},
%\sqrt{(T-t)\int_t^T\|\alpha\|^2_{L^2(m_s)}ds} + 2 \sqrt{M_2(\nu)}.
$$
and, thus %Taking the square leads to
$$
W_2^2(\mu,\nu) \leq 3 W_2^2(m_T,\nu) + 3(T-t)\int_t^T\|\alpha\|^2_{L^2(m_s)}ds + 12 M_2(\nu),
$$
from which the result easily follows.

\end{proof}

\subsection{The behavior of $\what U$ near the final time} 
The next proposition describes the behavior of $\what U(\mu,\nu,t)$ as $t\to T$. 

\begin{proposition}\label{takis252.1} Assume that $\mu, \nu \in  \wtilde \Pt$. Then: % $E(\mu)<\oo$ and $\nu\in \wtilde \Pt $. Then:
\vskip.05in
(i)~$\what U(\mu,\nu,t)\geq -(T-t)  \|F\|_\oo.$
\vskip.05in
(ii)~$\underset{t\to T}\lim \; \what U(\nu,\nu,t)=0.$

\vskip.05in
(iii)~For every $\rho>0$,  $\underset{t\to T} \lim \; \what U(\mu,\nu,t)=\oo$ uniformly in $\mu \in \{m \in \mcl P_2 \; : W_2(m,\nu)\geq \rho \}$.  
\end{proposition}
\begin{proof}
Since, for  each $(m,\al)\in \mcl A(\mu,\nu,t)$,  
$m\geq 0$ and 
\beq\label{takis253}
\dint_t^T \Big[\dint \dfrac{1}{2} \al_s^2 m_sdx +F[m_s]) \Big]ds \geq -(T-t) \|F\|_\oo,
\eeq
the first claim is immediate.
\vs
The proof of  (ii) follows from (i) and \eqref{takis299.10}, since $\underset{t\to T}\lim E(S_{\frac{T-t}{2}}\nu)=E(\nu)$.

\vs

The third claim follows  from the compactness of minimizing paths arguing by contradiction. Indeed, if the claim were false,
there would exist some $M\in (0,\oo)$ and sequences $(t_n)_{n\geq 1}$, $(\mu_n)_{n\geq 1}$, and $(\al_n)_{n\geq 1}$ such that  
\[
\ba
& (\al_n, m_n) \in  A(\mu_n,\nu,t_n) \ \text{ is an approximate minimizer of} \  \what U(\mu_n,\nu, t_n), \text{and, as $n\to \oo$, }\\[1.2mm]
& t_n\to t \ \text{and} \ \mu_n \rightharpoonup \mu  \ \text{ with} \  W_2(\mu,\nu)\geq \rho \ \text{and} \  \dint_{t_n}^T\dint_\R m_n \al_n^2 dx ds \leq M.
%& t_n \underset{n\to \oo} \to T,  \ \mu_n \underset{n\to \oo}\rightharpoonup \mu \ \text{ with} \  W_2(\mu,\nu)\geq \rho, \  (\al_n,m_n) \in \mcl A (\mu_n,\nu,t_n) \ \text{and} \  \dint_{t_n}^T\dint_\R m_n \al_n^2 dx ds \leq M.
\ea
\]

It would then follow from Proposition~\ref{takis2004} that there exists $(\ol \al,\ol m)\in \mcl A(\mu,\nu,T)$ which, of course, is impossible since  $\mu\neq \nu$.

\end{proof}

\section{The Hamilton-Jacobi-Bellman equation and the comparison principle}\label{sec:visc}

In this section we introduce the notion of (viscosity) solution to the HJB equation of the controlled Dyson problem 
\beq\label{takis100}
%\begin{cases} 
\ds -\pt_t u + \dfrac{1}{2}\int_\R(D_\mu U)^2 d\mu - \int_\R D_\mu U H[\mu] d\mu =F[m]  \ \ \text{in} \ \ \mcl P_2 \times (0,T),
%\\[2mm]
%u(\cdot, T)=U_T,
%\end{cases}
\eeq
and state and prove the comparison between continuous sub- and super-solutions with continuous terminal data. The comparison for the target problem, which is based on the control interpretation of the problem, is stated and proved in the next section. 

\subsection{The notion of viscosity solutions }%Motivation of the choice of test functions}
Since, in general, \eqref{takis100} does not have smooth solutions, we will use the theory of viscosity solutions in  $\mcl P_2 \times (0,T)$. 
However, there are various difficulties hidden in \eqref{takis100}, which complicate the notion of viscosity solution.
%particularly heavy in terms of notation. Those difficulties are:
They are related to the facts that (i)~the term $\ds \int_\R D_\mu U H[\mu] d\mu$  is singular and to make sense of it we need more regularity for $\mu$ than merely $\mu\in \mcl P_2$, (ii)~the space $\mcl P_2$ is unbounded and not locally compact, and (iii)~the solutions to \eqref{takis100} that we are interested in are unbounded.
\vs
Of course, in the absence of the singular term due to the Hilbert transform, HJB equations set in $\Pt$
%\eqref{takis100} is 
are well understood, see, for example, \citep{Lionscours,B}. 
\vs
Before we state  the  notion of viscosity solution for \eqref{takis100}, we present some computations and results which explain how we address the three difficulties mentioned above.

\vs

\textbf{The singular term involving the Hilbert transform.}
To introduce the notion of a viscosity  solution $u$ to \eqref{takis100}, we look at the equation satisfied by $u - \delta E$ where $E$ is the free entropy defined by \eqref{entropy}--recall that  $D_\mu E=-H$. 
\vs
We assume next that $u$ is a smooth sub-solution to \eqref{takis100}, and, for $\del$ small enough, we consider the function
\beq\label{takis104}
u_\del(\mu,t)=u(\mu,t)-\delta E(\mu).
\eeq

Inserting \eqref{takis104}  in  \eqref{takis100} we find, formally, that %after using \eqref{gradE}, that
\beq\label{takis105}
-\pt_t u_\del +\dfrac{1}{2} \dint_\R \Big(D_\mu u_\del- \del H[\mu]\Big)^2d\mu \;  - \dint_\R \Big(D_\mu u  - \del H[\mu] \Big) H[\mu] d\mu \leq F,
\eeq

and, after rearranging terms 
\beq\label{takis105}
-\pt_t u_\del +\dfrac{1}{2} \dint_\R\Big(D_\mu u_\del\Big)^2d\mu+ (1 +\delta)\dint_\R \Big(D_\mu u_\del \Big) H[\mu] d\mu + ( \del +\dfrac{\del^{2}}{2})\dint (H[\mu])^2 d\mu \leq F.
\eeq

We assume next that, for  a smooth function $\phi:\mcl P_2 \times (0,\infty)\mapsto \R$, $u_\delta-\phi$ attains a global maximum at some $(\mu_0,t_0)$. Then, $(\partial_t phi, D_\mu \phi) = (\partial_t u_\delta,D_\mu u_\delta)$ at that point which implies
\beq\label{takis105}
\bal
- \pt_t \phi(\mu_0,t_0)  +\dfrac{1}{2} \dint_\R\Big(D_\mu \phi(\mu_0,t_0)\Big)^2d\mu_0 &+ (1+\delta)\dint \Big(D_\mu \phi(\mu_0,t_0) \Big) H[\mu_0] d\mu_0\\
&+ ( \del +\dfrac{\del^{2}}{2})\dint (H[\mu_0])^2 d\mu_0 \leq F(\mu_0).
\eal
\eeq
It follows from Cotlar's identity and \eqref{takis105} that there exists a constant $C(\phi)$, such that  
%for $\ol \nu_\del=\dfrac{c_0}{3}\del $,
\beq\label{takis106}
\dint \mu_0^3 dx= \dfrac{3}{c_0}\dint (H[\mu_0])^2 d\mu_0 \leq  \dfrac{C(\phi)}{\delta}. %  \dfrac{3C(\phi)}{c_0\del}.
\eeq

In particular, the maximum point $\mu_0$ is regular, that is,  in $L^3$,  which is exactly what is needed to make sense of the singular term, with an estimate that depends on the test function $\phi$. 
%the previous yields a regularity estimate on the point $\mu_0$, which depends on the regularity of the test function $\phi$, and which is exactly the regularity we need to make sense of the singular term. 
\vs
Of course, this comes at the price of adding error terms in $\del$ in the relation \eqref{takis105}.\vs

\textbf{The lack of local compactness of $\mcl P_2$.} To define and establish the uniqueness of viscosity solutions to \eqref{takis100}, we  need to consider points of maximum of $u - \phi$, when  $u$ a sub-solution to \eqref{takis100}, and points of minimum of $u - \phi$, when  $u$ a sub-solution to \eqref{takis100},  and $\phi$ a smooth test function.
\vs
% $\phi$ a smooth test function, points of maximum of $u - \phi$ when  $u$ a sub-solution to \eqref{takis100} and $\phi$ a smooth test function, 
 
The existence of such points is not obvious because of the lack of local compactness of $\mcl P_2$. For example, even if we know that $u(\mu) - \phi(\mu) \to -\infty$ as $\ds \int |x|^2d\mu \to +\infty$ and $u$ and $\phi$ are continuous, there is no reason for $u-\phi$ to have a maximum precisely. 
\vs
To deal with this difficulty and to create points of maximum, we use 
% because of the lack of local compactness of $\mcl P_2$. In order to be able to create such points of maximum, we shall need 
the following non-linear variant of Stegall's Lemma \citep{stegall} which was established in \citep{BL2}. For the statement, we recall 
that a point of maximum is strongly exposed if all maximizing sequences converge toward it and that $\mathcal I_2$ is defined in \eqref{defI2}.
\begin{theorem}\label{thm:stegall}
Let $f: \mcl P_2\mapsto \mathbb R$ be a bounded from above upper semicontinuous  function such that $f(\mu) \to -\infty$ as $\ds \int |x|^2\mu(dx) \to +\infty$. Then, for any $\eta > 0$, there exists $\xi \in \mcl P_2$ such that $M_2(\xi) \leq \eta$ and 
$$
\mu \mapsto f(\mu) - \mathcal I_2(\xi,\mu)
$$
has a strongly exposed point of maximum on $\mcl P_2$.
\end{theorem}

%A point of maximum is strongly exposed if all maximizing sequences converge toward it. 
This perturbation introduces a number of  terms in the definitions of sub-and super-solution to \eqref{takis100} which are of order $\eta$ and  need to be compensated, as it is shown below,   using different terms in the equation. 
\vs

Assume that $u$ is a smooth solution to \eqref{takis100} and, for some smooth function $\phi$, $\eta>0$ and $\nu \in \Pt$ such that $M_2(\nu) \leq \eta$, let  $(t_0,\mu_0) \in [0,T)\times L^3$  be a maximum point of the map 
$$(t,\mu) \mapsto u(t,\mu) - \mcl I_2(\mu,\nu) - \phi(t,\mu).$$
\vs

%we  consider a maximum point  $\nu \in \Pt, M_2(\nu) \leq \eta$ and let $(t_0,\mu_0)$ be a point of maximum of $(t,\mu) \mapsto u(t,\mu) - \mcl I_2(\mu,\nu) - \phi(t,\mu)$ for some smooth function $\phi$, such that $\mu_0 \in L^3$. 

%Let $T: \R\mapsto \R$ the optimal transport of $\mu_0$ toward $\nu$. 
Using the differentiability of $\mcl I_2(\cdot,\nu)$ at $\mu_0$,  we deduce  that 
\beq\label{alex5}
D_\mu u(t_0,\mu_0) = D_\mu \phi(t_0,\mu_0) - \mcl T, 
\eeq
where $\mcl T: \R\mapsto \R$ is  the optimal transport of $\mu_0$ toward $\nu$.
\vs

Inserting  \eqref{alex5}  in \eqref{takis100} we find 
$$
-\partial_t \phi(t_0,\mu_0) + \frac12\int_\R \left(D_\mu \phi(t_0,\mu_0) - \mcl T\right)^2d\mu_0 - \int_\R(D_\mu \phi(t_0,\mu_0) - \mcl T)H[\mu_0]d\mu_0 = F[\mu_0].
$$
Expanding the square in the equality above and making  use of the fact that $\int_\R \mcl T^2 d\mu_0 = M_2(\nu) \leq \eta$ leads to
\[
\begin{split}
-\partial_t \phi(t_0,\mu_0) + &\frac12\int_\R \left(D_\mu \phi(t_0,\mu_0)\right)^2d\mu_0 - \int_\R D_\mu \phi(t_0,\mu_0)H[\mu_0]d\mu_0\\
& \geq F[\mu_0] - \frac\eta2 -\sqrt\eta\big(\|D_\mu\phi(t_0,\mu_0)\|_{L^2(\mu_0)} +\|H[\mu_0]\|_{L^2(\mu_0)} \big),
\end{split}
\]
which is a good justification of the formula for sub-solutions below. 
\vs
When we  introduce  the viscosity super-solution property, we naturally arrive at a different formulation because of the effect of the  perturbation on the level of the underlying controlled trajectories. 
\vs
Indeed, we recall that %in \eqref{takis100} the term $\frac12\int_\R (D_\mu u)^2d\mu$ arises from
$$
\frac12\int_\R (D_\mu u)^2d\mu = - \inf_{\alpha} \int_\R (\frac12 \alpha^2 - \alpha \cdot D_\mu u)d\mu.
$$
\vs
It turns out that some of the $\frac 12 \alpha^2$ has to be used to compensate for the fact that $\mcl I_2$ is simply $\sqrt{\eta}$ Lipschitz continuous, and this leaves us with  
$$
- \inf_{\alpha} \int_\R ((\frac12- \eta) \alpha^2 - \alpha \cdot D_\mu u)d\mu,
$$
and, after optimizing over $\alpha$,  \eqref{takis124}. 
\vs
%We note that the above  should not be a concern to  the careful reader, as, following the proof of the fact that the value function is a viscosity sub-solution below. We can recover an estimate similar to the one of viscosity solutions from the one of super-solutions, after some additional approximation steps...

%Those terms will be denoted by $\eta$, and the form of the terms in $\eta$ will be justified in the proofs where we establish that the value function is indeed a viscosity solution to \eqref{takis100}.

\color{black}
\vs

%Those terms will be denoted by $\eta$, and the form of the terms in $\eta$ will be justified in the proofs where we establish that the value function is indeed a viscosity solution to \eqref{takis100}.

\vs

\textbf{The unboundedness of the solution.} Dealing with the growth of viscosity solutions to HJB equations is a classical question which dates back to the work of Ishii \citep{Ishii}. In view of the bounds obtained in  the previous section, we expect that the value functions will  grow linearly in $M_2(\mu)$, which is a critical growth for the "quadratic" nonlinearity in \eqref{takis100}. As a result,   it is not possible to employ techniques similar to the ones of Ishii. Instead, we use the fact that \eqref{takis100} is convex in $D_\mu u$ and the observation of Alvarez \citep{alvarez} that, for such equations, the minimum of two sub-solutions is also a sub-solution, and, thus,  the minimum of two solutions is a solution since it is always true for super-solutions. This idea was made rigorous for equations on Wasserstein spaces without singular terms by the first two authors in \citep{BL2}. Since constants are clearly solutions to \eqref{takis100}, instead of sub-solutions $u$ to \eqref{takis100}, we will  thus work with $u\wedge A$ for $A \in \mathbb R\cup\{+\infty\}$.
\vs
Note that when working with super-solutions, we shall not face this problem, since, we expect the solutions to \eqref{takis100} to be bounded from below.
 
 \vs
\textbf{The definition of viscosity solutions.} 
We are now in position to present the rigorous definitions of viscosity solutions to \eqref{takis100}. We remark that, instead of assuming that the test functions  are smooth, here we restrict our attention to tests functions in $C^{1,1}(\Pt\times [0,T])$.  The reason for this  is that in the proof of the comparison principle, due to available regularity of the Wasserstein distance, we can construct approximations which are at most in $C^{1,1}$.

\vs
\begin{definition}\label{takis122}
(i)~ An upper semicontinuous map $u:\mcl P_2\times (0,T)\to \R$ is a viscosity sub-solution to \eqref{takis100}, if for every $\phi\in C^{1,1}( \mcl P_2\times (0,T); \R)$, $A \in \mathbb R \cup \{+ \infty\}$, $\eta > 0$, $\xi \in \mcl P_2$ such that $M_2(\xi) \leq \eta$, $\del >0$ and a global maximum $(\ol \mu,\ol t) \in \mcl P_2\times (0,T)$   of $(\mu,t) \mapsto u\wedge A-\phi-\del E(\mu) - \mathcal I_2(\xi,\mu)$, there exists $C=C(\|\phi\|_{C^{1,1}}, \delta)>0$ such that 
\beq\label{takis123.0}
\ol \mu\in L^3 \ \ \text{and} \ \  \dint \ol \mu^3 dx \leq C
\eeq
and 
\beq\label{takis123}
\bal
\bal
\color{black}\left(\delta+ \frac12\delta^2\right)\dint_\R & (H[\ol \mu])^2\ol \mu dx -(1+\delta)\sqrt\eta\|H[\ol \mu]\|_{L^2(\ol \mu)}- \partial_t \phi(\ol \mu, \ol t) - ( 1+\delta)\dint_\R D_\mu \phi(\ol \mu,\ol t) H[\ol \mu]\ol \mu dx+\\
&+\frac12\dint_\R(D_\mu \phi(\ol \mu,\ol t))^2 \ol \mu dx \leq F[\ol \mu]+ \frac12\eta + \sqrt{\eta}\|D_\mu\phi(\ol \mu,\ol t)\|_{L^2(\ol \mu)}.
\eal
\eal
\eeq

(ii)~A lower semicontinuous map $u:\mcl P_2\times (0,T)\to \R$ is a viscosity super-solution to \eqref{takis100} if for every smooth $\phi \in C^{1,1}( \mcl P_2\times (0,T); \R)$, $\delta \in (0,1)$, $\eta \in (0,\min(\frac14;\delta \frac{1-\delta}{2}))$, $\xi \in \mcl P_2$ such that $M_2(\xi) \leq \eta$,
and $(\ol \mu,\ol t) \in \mcl P_2\times (0,T)$ a global minimum of $u-\phi+\del E(\mu) + \mathcal I_2(\xi,\mu)$,  there exists $C=C(\|\phi\|_{C^{1,1}})>0$ such that 
\beq\label{takis124.0}
\ol \mu\in L^3 \ \ \text{and} \ \  \dint \ol \mu^3 dx \leq C
\eeq
and 
\beq\label{takis124}
\bal
-\partial_t \phi(\ol \mu, \ol t) &+ \dfrac{1}{2(1-2\eta)} \dint_\R(D_\mu \phi(\ol \mu, \ol t))^2d\ol \mu - \left(1 - \frac{\delta}{1-2\eta}\right)\dint_\R D_\mu \phi H(\ol \mu) d\ol \mu- \\
&\qquad \qquad -\left( \delta-\eta - \dfrac{\delta^2}{2(1-2\eta)}\right)\dint_\R H[\ol \mu]^2d\ol \mu \geq F[\ol \mu] -\frac \eta 2.
\eal
\eeq

(iii)~A continuous function $u:\mcl P_2\times (0,T)\to \R$ is a viscosity solution to \eqref{takis100} if it is both a sub-and super-solution. 
\end{definition}

A very important property of viscosity solutions is their stability under limits. In view of fact that the entropy is only weakly lower semicontinuous with respect to the weak convergence of measures, stability appears to be problematic for the solutions. However, the definition of the sub-solution (resp.  super-solution) yields a uniform estimate in $L^3$ for the maximizing (resp. minimizing) measures. With this uniform estimate it is then possible to show the continuity of the entropy provided we know that the approximations converge locally uniformly.
\vs

We formulate the observation as a separate lemma which we state without a proof.

\begin{lemma}\label{TT1}
(i)~Let $(u_n)_{n\in \N}$ be a sequence of sub-solutions to \eqref{takis100} and assume that, as $n\to \infty$, the $u_n$'s converge locally uniformly in $\Pt\times [0,T)$ to some $u$, then $u$ is a sub-solution to \eqref{takis100}.

(ii)~Let $(u_n)_{n\in \N}$ be a sequence of super-solutions to \eqref{takis100} and assume that, as $n\to \infty$, the $u_n$'s converge locally uniformly in $\Pt\times [0,T)$ to some $u$, then $u$ is a super-solution to \eqref{takis100}. 
\end{lemma}

\begin{remark}
The local uniform convergence assumption of the previous result is a strong requirement which is not always verified. We shall comment on it once again at the end of the paper.
\end{remark}

\subsection{The comparison principle for continuous terminal data} We state the comparison principle for viscosity sub- and super-solutions to \eqref{takis100}. Then we discuss a first attempt of a proof in order to point out some of the difficulties. The technical steps needed to overcome the main obstacles are formulated  as  separate lemmata. The proof of the comparison is then given at the end of the subsection. Note that because we are dealing with a time dependent problem, the continuity of the sub- and super-solutions at time $t = T$ plays an important role. 
\vs

The comparison result deals with two separate cases. In the first one, both sub-and super-solutions are assumed to be uniformly continuous. The second deals with  sub-and super-solutions that are only locally uniformly continuous  but the  super-solution is bounded from below by a "quadratic" function.

\begin{theorem}\label{thm:comp}
Let $U:\Pt\times [0,T]\to \R$  and $V:\Pt\times [0,T]\to \R$ be respectively  viscosity sub- and super-solutions to  \eqref{takis100}, both bounded from below and such that $U(\cdot, T) \leq V(\cdot, T)$ on  $\Pt$. Assume that, for some $h > 0$, either 
%\vskip.0125in
%(i)~$U$ and $V$ are uniformly continuous on $\mcl P_2\times[T-h,T] $ for some $h > 0$, or \\
%(ii)~$U$ and $V$ are  locally uniformly continuous on $\mcl P_2\times[T-h,T] $ for some $h > 0$ and there exists $C > 0, \nu \in \mcl P_2$ such that, for all $(t,\mu) \in [T-h,T]\times \mcl P_2$, $V(\mu,t) \geq C^{-1} W_2^2(\mu,\nu) - C.$
%%
(i)~$U$ and $V$ are uniformly continuous on $\mcl P_2\times[T-h,T] $, or 
(ii)~there exists $C > 0$ and $ \nu \in \mcl P_2$ such that for all $(t,\mu) \in [T-h,T]\times \mcl P_2$, $V(\mu,t) \geq C^{-1} W_2^2(\mu,\nu) - C.$
%\[
%\begin{split}
%&\text{either $U$ and $V$ are uniformly continuous on $\mcl P_2\times[T-h,T] $ for some $h > 0$,}\\
%&\text{
%or, $U$ and $V$ are locally uniformly continuous on $\mcl P_2\times[T-h,T] $ for some $h > 0$ and}\\
%&\text{  there exists $C > 0, \nu \in \mcl P_2$ such that for all $(t,\mu) \in [T-h,T]\times \mcl P_2$, $V(\mu,t) \geq C^{-1} W_2^2(\mu,\nu) - C.$}
%\end{split}
%\]
Then, $U \leq V$ \ in \  $\Pt\times [0,T]$. 
\end{theorem}

Before we go into the details, we 
continue with a brief sketch of the general strategy of the proof of the comparison. In what follows we ignore the issues arising because we work in $\mcl P_2$ and we assume that  upper semicontinuous functions on $\Pt \times [0,T]$ attain their maximum.
\vs

Let $u$ and $v$ be respectively an upper semicontinuous sub-solution and lower semicontinuous super-solution and assume that the map
\beq\label{takis111}
(\mu,\nu) \to u(\mu,t) -v(\nu,t) - \dfrac{1}{2\ep} W^2_2(\mu,\nu) -\delta E[\mu]-\delta E[\nu]
\eeq
attains a global maximum at  $(\ol \mu, \ol \nu,\ol t) \in \mcl P_2\times \mcl P_2 \times [0,T)$. 
\vs
It then follows from the definitions of the sub- and super-solutions   that 
\beq\label{takis111}
\ol \mu, \, \ol\nu\in L^3.
\eeq

Since \eqref{takis111} implies that both $\ol\mu$ and $\ol\nu$ have no atoms,  the map  $(\mu,\nu)\to W^2_2(\mu,\nu)$ is differentiable at $(\ol \mu,\ol \nu)$ and 
\beq\label{takis117}
\int_\R |D_\mu W^2_2(\ol \mu,\ol\nu)|^2\ol \mu(dx)=\int_\R |D_\nu W^2_2(\ol \mu,\ol\nu)|^2\ol \nu(dx).
\eeq
\vs
Indeed, if  $\mcl T$ (resp. $S$) is  the optimal transport map for the $W_2-$distance of $\ol \mu$ to $\ol \nu$ (resp. $\ol \nu$ to $\ol \mu$), we have that $\mcl T_\#\ol \mu = \nu$ and $D_\mu W^2_2(\ol \mu,\ol\nu) (x)= x-\mcl T(x)$ (resp. $S_\#\ol \nu = \mu$ and $D_\nu W^2_2(\ol \mu,\ol\nu) (x)= x-S(x)$).
\vs

If we had ``$1 +\del=1-\del$'', 
then we could conclude the uniqueness proof using that 
the $W_2^2-$distance between two solutions of the Dyson equation is decreasing. 
\vs

To guarantee  the existence of a maximum point and deal with the fact that $1 +\del\neq 1-\del$, it is necessary to introduce a number of approximations. Then we need to know that the map $(\mu,\nu) \to W^2_2(\mu,\nu)$ is differentiable in a neighborhood of  $(\ol \mu,\ol \nu)$, a fact which, in general, is not true. As a result, $W^2_2$ is not  a good test function to apply the definitions of sub- and super-solutions, since it is necessary to work with measures near the global maximum. 
\vs
Hence, before proceeding with a proof of the comparison, we need a stability result which justifies the use of squared Wasserstein distances as test functions.

\vs 

The first step is a precise estimate for  the $L^3-$norm of the point of maximum and minimum 
in terms of $\phi$ and its derivatives. 
%Since the arguments leading to the estimate are  similar for both %Below we derive such an estimate for the 
%the possible maximum  and minimum point, we only show the details for the former, and simply state the claim for the latter.

%possible minima.
%
%Since the argument is similar, we only state the bound for the possible minimum points. The proof follows the formal derivation of such bound shown at the beginning of the section.

\vs

\begin{lemma}\label{lemma:estL3}
Let $U$ be a viscosity sub-solution (resp. super-solution) to \eqref{takis100}, $\phi \in C^{1,1}\big(\Pt \times [0,T];\R\big)$, $\delta \in (0,1/2)$ and $(\ol \mu,\ol t)$  a point of maximum (resp. minimum) of the map $(\mu,t) \to U(\mu, t) -\phi(\mu,t) - \delta E(\mu)$ (resp. $(\mu,t) \to U(\mu,t) - \phi(\mu, t) + \delta E(\mu)$). Then, there exists $C >0$ such that
\[
\|\ol \mu \|^3_3 \leq \dfrac{C}{\delta } \left(\|\partial_t \phi\|_\infty + \|F\|_\oo + \delta^{-1} \|D_\mu \phi\|_{L^2(\mu)}^2\right).
\]
%\[\qquad \Big( \text{resp.} \hskip.5in 
%\|\ol \mu \|^3_3 \leq \dfrac{C}{\delta} \left(\|\partial_t \phi\|_\infty + \|F\|_\oo + \delta^{-1} \|D_\mu \phi\|_{L^2(\mu)}^2\right).  \Big)\]

\end{lemma}

\begin{proof}
The sub-solution property implies that 
\[
\dfrac{1}{2} \dint_\R |D_\mu \phi(\ol \mu,y,\ol t)|^2 \ol \mu(y) dy - (1 +\del) \dint_\R D_\mu \phi (\ol \mu, y,\ol t) H[\ol \mu] \ol \mu dy
+\del \dint_\R H[\ol \mu]^2\ol \mu dy \leq \|\pt_t \phi\|_\infty + \|F\|_\oo, 
\]
while an application of the Cauchy-Schwarz inequality yields that 
\[
(1+\del) \dint_\R D_\mu \phi (\ol \mu, y,\ol t) H[\ol \mu] \ol \mu dy\leq \dfrac{ \del}{2} \dint_\R |D_\mu \phi(\ol \mu,y,\ol t)|^2 \ol \mu dy + \dfrac{(1+\del)^2}{2 \del} \dint_\R H[\ol \mu]^2\ol \mu dy.
\]
Combining the last two inequalities we find 
\[\dfrac{\delta}{2} \dint_\R H[\ol \mu]^2\ol \mu dy\leq \dfrac 12 \left( 1 + \dfrac{(1 +\del)^2}{\del} \right) \dint_\R |D_\mu \phi(\ol \mu,y,\ol t)|^2 \ol \mu(y) dy + \|\partial_t \phi\|_\infty + \|F\|_\oo,\]
\vs
and the claim follows.
\vs
The super-solution property gives that, for all $\del\in (0,\frac12)$,
\[\del \dint_\R H[\ol \mu]^2\ol \mu dy \leq \dfrac{1}{2} \dint_\R |D_\mu \phi(\ol \mu,y,\ol t)|^2 \ol \mu(y) dy -  (1 -\del) \dint_\R D_\mu \phi (\ol \mu, y,\ol t) H[\ol \mu] \ol \mu dy
-\|\pt_t \phi\|_\infty + \|F\|_\oo. \]
Observing that 
\[-(1-\del) \dint_\R D_\mu \phi (\ol \mu, y,\ol t) H[\ol \mu] \ol \mu dy\leq \dfrac{ \del}{2} \dint_\R H[\ol \mu]^2\ol \mu dy  + \dfrac{(1-\del)^2}{2  \del} \dint_\R |D_\mu \phi(\ol \mu,y,\ol t)|^2 \ol \mu(y) dy,\]
we may conclude following the previous argument.

\end{proof}

The choice of the regularity of the test functions  in Definition~\ref{takis122}  is made in order to deal with the existence part, and we refer the reader to the next section for more details on this.
However, when we double variables to prove the comparison principle, we need to use  the Wasserstein distance which is not smooth. 
\vs
The next lemma provides the technical justification that allows  for the  use of  the squared Wasserstein metric. 

\begin{lemma}\label{lemma:Uwas}
Let $U$ be a viscosity sub-solution to \eqref{takis100}, fix $A \in \R$, $\eps, \theta, \alpha> 0$, $s\in (0,T)$, $\nu \in \Pt \cap L^3$, $\delta \in (0,1/2)$, $\eta \in (0,\delta)$, $\xi \in \mcl P_2$ such that $M_2(\xi) \leq \eta$,  and assume that $(\ol \mu, \ol t)\in \Pt\times (0,T)$ is a strongly exposed maximum on $\Pt\times [0,T]$ of the map 
\[
(\mu, t)  \to (U\wedge A)(\mu,t) - \dfrac{1}{2 \eps}W_2^2(\mu,\nu) - \delta E(\mu) - \mcl I_2(\xi,\mu) - \dfrac{1}{2\eps}(t-s)^2 + \theta t - \alpha M_2(\mu).
\]
Then $\ol \mu \in L^3$ and, if $S_1$ and  $S_2$ are respectively   the optimal transport maps for the quadratic cost to move $\ol \mu$ toward $\nu$ and  $\xi$, then

\[
\ba
-\dfrac{\bar t -s}{\eps} &+ \theta + \dfrac{1}{2} \dint_\R \Big(\dfrac{x - S_1(x)}{\ep}  + S_2(x) + 2 \alpha x\Big )^2d\ol \mu+  \delta \int_\R H[\ol \mu ]^2(x)d\ol \mu \\
&  -(1 + \delta)\int_\R \Big(\dfrac{x - S_1(x)}{\ep} + S_2(x) + 2\alpha x\Big) H[\ol \mu](x) d\ol \mu \leq F[\ol \mu].
\ea
\]
\end{lemma}
\vs
\begin{proof}
Fix $\gamma > 0$ and let  $\Phi_\gamma$ and $\Psi_\gamma$ be the $\text{sup}-$convolution regularizations of the maps $\mu  \to W_2^2(\mu,\nu)$  and $\mu  \to \mcl I_2(\xi,\mu)$ respectively introduced in Section \ref{sec:regW2}.
\vs

%$\Psi_\gamma$ the one associated to $\mcl I_2(\xi,\cdot)$.  

Let  $\eta_\gamma > 0$ be such that $\lim_{\gamma\to 0} \eta_\gamma =0$ and   use Theorem \ref{thm:stegall} to find $\xi_\gamma$ such that $M_2(\xi_\gamma) \leq \eta_\gamma$  and  a maximum point $(\mu_\gamma, t_\gamma)$ in $\Pt\times[0,T]$ of the map 
\[
(\mu,t) \to (U\wedge A)(\mu,t) - \frac{1}{2\eps} \Phi_\gamma(\mu)  - \delta E(\mu) - \Psi_\gamma(\mu) - \frac{1}{2\eps}(t-s)^2 + \theta t  - \alpha M_2(\mu). %- \mcl I_2(\xi_\gamma,\mu). 
\]

A simple computation yields
\[M_2(\mu_\gamma) \ \text{ is bounded independently of $\gamma$},\]
\vs
and, since  $(\ol \mu, \ol t)$ is assumed to be exposed, we know that, as $\gamma \to 0$, %provided that $\eta_\gamma \to %0$ which we assume,
\[|t_\gamma -\bar t| + W_2(\mu_\gamma,\ol \mu) \to 0.\]

Before we continue we note that, 
as shown in \citep{BL}, $\Phi_\gamma$ is, uniformly in $\gamma$, Lipschitz continuous in $(\Pt,W_2)$ around $\ol  \mu$, thus  the bound on $\|\mu_\gamma\|_3$  in \eqref{takis10001.0}   is uniform in $\gamma$. 
\vs
The function
\[
\tilde\Phi(\mu,t) = \frac{1}{2\eps}\Phi_\gamma(\mu)  + \Psi_\gamma(\mu) +\frac{1}{2\eta}(t-s)^2 - \theta t  + \alpha M_2(\mu)
\]
is sufficiently smooth, that is, in $C^{1,1}$,  and can  be used as a test function. Moreover, for all $ (\mu ,t)\in \Pt \times [0,T]$ and $x \in \R$,
\[
D_\mu\tilde\Phi(\mu,x, t)) = \frac{1}{2\eps}D_\mu\Phi_\gamma(\mu)(x) + D_\mu\Psi_\gamma(\mu)(x) + 2\alpha x.
\]
\vs
Since $U$  is a viscosity sub-solution to \eqref{takis100}, it follows  that $\mu_\gamma \in L^3$ and, in view of Lemma~\ref{lemma:estL3}, we have 
\beq\label{takis10001.0}
\| \mu_\gamma\|^3_3 \leq \dfrac{C}{ \delta}\left(\dfrac{1}{\eps} +\|F\|_\oo +  \dfrac{1}{\delta}\|D_\mu\Psi(\mu_\gamma)\|^2_{L^2(\mu_\gamma)}\right).
\eeq
and
\be\label{eq:671}
\ba
-&\frac{ t_\gamma -s}{\eps}+ \theta + \frac 12 \int_\R\left(\frac{1}{2\ep}D_\mu\Phi_\gamma (\mu_\gamma)( x) + D_\mu\Psi_\gamma(\mu_\gamma)(x) + 2\alpha x\right)^2 \mu_\gamma(dx) \\
&-(1 + \delta)\int_\R \left(\frac{1}{2\ep}D_\mu\Phi_\gamma (\mu_\gamma)( x) + D_\mu\Psi_\gamma(\mu_\gamma)(x) + 2\alpha x\right)H[\mu_\gamma] d\mu_\gamma\\
&+ (\delta+ \frac12\delta^2)\dint_\R  (H[\mu_\gamma])^2\mu_\gamma dx   \leq F[\mu_\gamma] +\\
&\quad \quad \quad+ \frac12\eta_\gamma + \sqrt{\eta_\gamma}\|D_\mu\tilde\Phi( \mu_\gamma)\|_{L^2( \mu_\gamma)} + (1+\delta)\sqrt{\eta_\gamma}\|H[\mu_\gamma]\|_{L^2( \mu_\gamma)}.
\ea
\ee
\vs
%
%Before we continue we note that, 
%as shown in \citep{BL}, $\Phi_\gamma$ is, uniformly in $\gamma$, Lipschitz continuous in $(\Pt,W_2)$ around $\ol  \mu$, thus  the bound on $\|\mu_\gamma\|_3$  in \eqref{takis10001.0}   is uniform in $\gamma$. 
%\vs

%Since $U$ is a viscosity sub-solution to \eqref{takis100}, we deduce that
%
%
%\be\label{eq:671}
%\ba
%-&\frac{ t_\gamma -s}{\eps}+ \theta + \frac 12 \int_\R\left(\frac{1}{2\ep}D_\mu\Phi_\gamma (\mu_\gamma)( x) + D_\mu\Psi_\gamma(\mu_\gamma)(x) + 2\alpha x\right)^2 \mu_\gamma(dx) \\
%&-(1 + \delta)\int_\R \left(\frac{1}{2\ep}D_\mu\Phi_\gamma (\mu_\gamma)( x) + D_\mu\Psi_\gamma(\mu_\gamma)(x) + 2\alpha x\right)H[\mu_\gamma] d\mu_\gamma\\
%&+ \lambda(\delta+ \frac12\delta^2)\dint_\R  (H[\mu_\gamma])^2\mu_\gamma dx   \leq F[\mu_\gamma] \\
%&+ \frac12\eta_\gamma + \sqrt{\eta_\gamma}\|D_\mu\tilde\Phi( \mu_\gamma)\|_{L^2( \mu_\gamma)} + (1+\delta)\sqrt{\eta_\gamma}\|H[\mu_\gamma]\|_{L^2( \mu_\gamma)}.
%\ea
%\ee
%Our objective is now to pass to the limit $\gamma \to 0$ in \eqref{eq:671} to obtain the result. 
Since $\ol \mu\in L^3$, there exist optimal transport maps $S_1,S_2$ as in the statement of the lemma, and, formally, as $\gamma \to 0$, we have that 
\[D_\mu \Phi_\gamma(\mu_\gamma)(x) \rightarrow x-S_1(x) \ \ \text{ and} \ \  D_\mu \Psi_\gamma(\mu_\gamma)(x) \rightarrow -S_2(x).\] 

%Also, we take a sequence $\eta_\gamma$ converging toward $0$ as $\gamma \to 0$.

Observe now that,  as $\gamma \to 0$ and along, if necessary, subsequences, 
\[
M_2(\mu_\gamma) \to M_2(\ol \mu) \ \ \text{and} \ \ 
\limsup \int_\R H[\mu_\gamma]^2d\mu_\gamma \geq \int_\R H[\ol \mu]^2d\ol \mu.
\]
Moreover, an immediate extension of Corollary~1 in \citep{BL} yields that
\[
%\ba
\lim_{\gamma \to 0}\int_\R\left(\frac{1}{2\ep}D_\mu \Phi_\gamma( \mu_\gamma)(x) + D_\mu\Psi_\gamma(\mu_\gamma)(x) + 2\alpha x\right)^2 \mu_\gamma(dx)= \int_\R\left(\frac{1}{\eps}(x -S_1(x)) -S_2(x) + 2\alpha x\right)^2 \ol \mu(dx).
% \\
%&= \int_\R\left(\frac{1}{\eps}(x -S_1(x)) -S_2(x) + 2\alpha x\right)^2 \ol \mu(dx).
%\ea
\]

\vs

The claim follows,  if we establish the convergence of the second line in \eqref{eq:671}. Since this passage to the limit is rather technical, we present it as a separate lemma next.
 \vs
Using the above and Lemma~\ref{lem333} below completes the proof. 

\end{proof}

\begin{lemma}\label{lem333}
Under the notation of the proof of Lemma \ref{lemma:Uwas}, we have 

\[
%\ba
\lim_{\gamma \to 0} \int_\R \left(\frac{1}{2\ep} D_\mu\Phi_\gamma(\mu_\gamma)(x) + D_\mu\Psi_\gamma(\mu_\gamma)(x) + 2 \alpha x\right)H[\mu_\gamma] d\mu_\gamma= \int_\R \left(\frac{1}{\ep}(x- S_1(x)) -S_2(x)+ 2\alpha x\right)H[\ol \mu] d\ol \mu. \\
%= \int_\R \left(\frac{1}{\ep}(x- S_1(x)) -S_2(x)+ 2\alpha x\right)H[\ol \mu] d\ol \mu. 
%\ea
\]
\end{lemma}
\vs
\begin{proof}
We first note that,  since $\mu_\gamma$ converges weakly to  $\ol \mu$, the convergence of the term involving $\alpha$ is immediate. Moreover, since the approximation $\Phi_\gamma$ and $\Psi_\gamma$ are built similarly, we only prove here the convergence of the term involving $D_\mu \Phi_\gamma$, and leave the proof of the convergence of $D_\mu \Psi_\gamma$ to the reader. % the last one follows from the same argument.
\vs

We use  the formalism of random variables.  We assume that 
$(\Omega, \mathcal{A}, \mathbb P)$ is  a Polish atomless probability space and denote by $(X_\gamma)_{\gamma>0}$ a family of random variables with law  $(\mu_\gamma)_{\gamma >0}$ which converges in $L^2(\Omega,\R)$ toward $\ol X$ with law $\ol \mu$. 
\vs
The claim we want to prove is that 
\[\ds \underset{\gamma \to 0} \lim \mathbb{E}[D V_\gamma(X_\gamma) H[\mu_\gamma](X_\gamma)] = 2\mathbb{E}[(\ol X- S(\ol X))H[\ol \mu](\bar X)],\]
where $V_\gamma$ is the lift of $\Phi_\gamma$ in $L^2(\Omega,\R)$ defined by $V_\gamma(X) = \Phi_\gamma(\mathcal L(X))$.
\vs
Next recall that, in view of  Theorem \ref{thm:approxW2},  we have 
\[
\underset{\gamma \to 0} \lim \mathbb{E}[|D \wtilde \Phi_\gamma(X_\gamma) - 2(\ol X- S_1(\ol X))|^2]  = 0.
\]
Then the uniform bound on $\|\mu_\gamma\|_3$ yields   that
\beq\label{takis4444}
\text{ the family} \  (H[\mu_\gamma](X_\gamma))_{\gamma > 0} \ \text{ is bounded in} \  L^2(\Omega, \R).
\eeq 
\vs
Moreover,  the properties of the Hilbert transform also give  
\be\label{eq:726}
\lim_{\gamma \to 0} \mathbb{E}[X_\gamma H[\mu_\gamma](X_\gamma)] = \mathbb{E}[\ol X H[\ol \mu](\ol X)].
\ee
\vs
Let  $Z_\gamma = 2(X_\gamma - \mcl T_\gamma(X_\gamma))$ where $ \mcl T_\gamma$ is the optimal transport map of $\mu_\gamma \in L^3$ toward $\nu$. It is immediate that  $(\mcl T_\gamma(X_\gamma))_{\gamma>0}$ converges toward $S_1(\ol X)$ strongly in $L^2(\Omega, \R)$, and, thus,
\beq\label{takis4445}
\mathbb{E}[|Z_\gamma -2(\ol X - S_1(\ol X))|^2]=0.
\eeq
\vs
In view of \eqref{takis4444}, \eqref{eq:726} and \eqref{takis4445}, to conclude it suffices to check that 
\beq\label{takis4446}
\underset{\gamma \to 0} \lim \mathbb{E}[Z_\gamma H[\mu_\gamma](X_\gamma)]=2\mathbb{E}[(\bar X- S_1(X))H[\ol \mu](\bar X)].
\eeq

For any nondecreasing map $f:\R \to \R$ and $M> 0$ we write  $f_{\land M}(x) = \max(-M,\min(f(x),M))$ and observe that, for all $\gamma>0$, 
\[
\underset{M\to \oo} \lim \mathbb{E}[|(\mcl T_\gamma)_{\land M}(X_\gamma) - \mcl T_\gamma(X_\gamma)|^2] = 0;
\]
note that the same holds with $S_1$ and $\ol X$ in place of  $\mcl T_\gamma$ and $X_\gamma$. 
\vs

In view of the uniform estimate on $\|\mu_\gamma\|_3$ and  \eqref{eq:726}, it follows that \eqref{takis4446}  will hold,  if it is shown that, for any $M > 0$, 

\be\label{eq:734}
\lim_{\gamma \to 0} \mathbb{E}[(\mcl T_\gamma)_{\land M}(X_\gamma) H[\mu_\gamma](X_\gamma)] = \mathbb{E}[(S_1)_{\land M}(\ol X) H[\ol \mu](\ol X)].
\ee
\vs
Given a family  $(\phi_{\eps'})_{\eps'>0}$ of approximations of the unity, let $\mcl T_{\gamma,M,\eps'} = ((\mcl T_\gamma)_{\land M} \ast\phi_{\eps'}$. Then, for any $M > 0$, there exists a modulus of continuity $\omega$ such that 
\be\label{eq:738}
\mathbb{E}[|\mcl T_{\gamma,M,\eps'} (X_\gamma)- (\mcl T_\gamma)_{\land M}(X_\gamma)|^2] \leq \omega(\eps') \ \ \text{and} \ \ \mathbb{E}[|(S_1)_{M,\eps'} (\ol X)- (S_1)_{\land M}(\ol X)|^2] \leq \omega(\eps').
\ee
\vs
It is important to note that this modulus does not depend on $\gamma$. Indeed, it follows from the bound on $(M_2(\mu_\gamma))_{\gamma>0}$ that, for any $\wtilde \eps$, there exists $K > 0$ such that 
\[
\ba
\mathbb{E}[|\mcl T_{\gamma,M,\eps'} (X_\gamma)- (\mcl T_\gamma)_{\land M}(X_\gamma)|^2] &\leq \dint_{-K}^K |\mcl T_{\gamma,M,\eps'} (x)- (\mcl T_\gamma)_{\land M}(x)|^2\mu_\gamma(dx) + 4M^2\tilde \eps\\[1mm]
& \leq \|\mcl T_{\gamma,M,\eps'} - (\mcl T_\gamma)_{\land M}\|_3^{\frac 23} \|\mu_\gamma\|_3 + 4 M^2\tilde \eps\\[1mm]
& \leq CK^{\frac 23}M^2 \omega'(\eps') + 4 M^2\tilde \eps.
\ea
\]
\vs
The goal is  to pass to the  $\gamma \to 0$ limit in $\mathbb{E}[\mcl T_{\gamma,M,\eps'}(X_\gamma)H[\mu_\gamma](X_\gamma)]$, which can be written as 

\[
\mathbb{E}[\mcl T_{\gamma,M,\eps'}(X_\gamma)H[\mu_\gamma](X_\gamma)] = \frac 12\dint_\R\dint_\R \frac{\mcl T_{\gamma,M,\eps'}(x) - \mcl T_{\gamma,M,\eps'}(y)}{x-y}\mu_\gamma(x)\mu_\gamma(y)dxdy. 
\]
\vs
Using the regularity of the convolution and the fact that, since $\mu_\gamma$ converges to $\ol \mu$, $\mcl T_{\gamma}$ converges almost everywhere to $S_1$, we find that $(x-y)^{-1}(\mcl T_{\gamma,M,\eps'}(x) - \mcl T_{\gamma,M,\eps'}(y))$ converges locally uniformly toward $(x-y)^{-1}((S_1)_{M,\eps'}(x) -(S_1)_{M,\eps'}(y)$.
\vs
It then follows that 
\[
\lim_{\gamma \to 0} \mathbb{E}[\mcl T_{\gamma,M,\eps'}(X_\gamma)H[\mu_\gamma](X_\gamma)] = \mathbb{E}[(S_1)_{M,\eps'}(\ol X)H[\ol \mu](\ol  X)].
\]
Hence, using \eqref{eq:738}, we obtain  \eqref{eq:734} and, hence, the claim.

\end{proof}

A result similar to Lemma~\ref{lemma:Uwas} holds for viscosity super-solutions and is stated in the next lemma. Since its proof is similar to the one of Lemma~\ref{lemma:Uwas}  we omit it.
\begin{lemma}\label{lemma:Vwas}
Let $V$ be a viscosity super-solution to \eqref{takis100} which is bounded from below. Fix $\mu,\xi \in \Pt$,  $\eps, \theta,t > 0$ and $\delta \in (0, \frac12)$. Assume that  $\ol s,\ol \nu$ is a strongly exposed minimum on $ \Pt \times [0,T]$ of the map 
\[
(s,\nu) \to V(s,\nu) + \frac{1}{2 \eps}W_2^2(\mu,\nu) + \delta E(\nu) -\mcl I_2(\xi,\nu) + \frac{1}{2\eta}(t-s)^2 + \theta s.
\]
Then $\ol \nu \in L^3$ and, if $\wtilde S_1$  and $\wtilde S_2$  are respectively  the optimal transport for the quadratic cost of $\ol  \nu$ toward $\mu$ and towards  $\xi$, we have
\[
\ba
-&\frac{ t -\bar s}{\eta} + \theta + \frac{1}{2} \int_\R(\eps^{-1}(\wtilde S_1(x) -x) + \wtilde S_2(x))^2d\bar \nu\\
& -(1 - \delta)\int_\R (\eps^{-1}( \wtilde S_1(x)-x) +\wtilde S_2(x)) H[\bar\nu](x) d\bar\nu
- \delta \int_\R H[\ol \nu]^2d\ol\nu \geq 0.
\ea
\]
%where $\wtilde T$ is the optimal transport for the quadratic cost of $\ol  \nu$ toward $\mu$.
\end{lemma}
\vs 

We are now ready to prove the comparison principle.

\begin{proof}[Proof of Theorem \ref{thm:comp}]
Throughout the proof we assume for simplicity that $F\equiv 0$, and we start by establishing that, for any $A \in \R$, $(U\wedge A) \leq V$.
\vs
We argue by contradiction assuming that the result is false, that is, that there exist  $K, \rho > 0$, $\delta^* \in (0, \frac12)$ and $ \alpha^* > 0$ such that,  for any $0< \delta \leq \delta^\ast$ and $0 < \alpha \leq \alpha^\ast$,
\be\label{takis3000}
\ba
\sup_{t,s \in [0,T], \mu,\nu \in \Pt} &\Big((U\wedge A)(\mu, t) - V(\nu,s) -\frac{1}{2\eps}\left(W_2^2(\mu,\nu) +(t-s)^2\right) \\
&- \delta E(\mu) - \delta E(\nu) - \alpha M_2(\mu) + \rho(2T -t-s)\Big)\geq K.
\ea
\ee

This inequality, which contradicts the claim, follows  from the  upper semicontinuity of $U$ and the lower semicontinuity of $V$ and the fact that $\widetilde {\mcl P}_2$ 
 %$\{\mu : \; E(\mu) < \infty\}$ 
is a dense subset of $\Pt$. 
%and  the upper semicontinuity of $U$ and the lower semicontinuity of $V$. 
\vs

Using Theorem \ref{thm:stegall}\footnote{More precisely, we use its corollary, Corollary 3.4 in \citep{BL2}.}, we  find a point of strongly exposed maximum of the expression inside \eqref{takis3000}, that is, for any $\eta > 0$, there exist $\theta_1,\theta_2 \in (0,\eta)$ and $\xi_1,\xi_2 \in \mcl P_2$ such that  $M_2(\xi_1) \leq \eta$ and $M_2(\xi_2) \leq \eta$ and 
\[
\ba
(\mu,\nu,t,s) \to (U\wedge A)(\mu,t)& - V(\nu,s) -\frac{1}{2\eps}\left(W_2^2(\mu,\nu) +(t-s)^2\right) - \delta E(\mu) - \delta E(\nu) - \alpha M_2(\mu)\\
&+ \rho(2T -t-s)- \mcl I_2(\xi_1,\mu) - \mcl I_2(\xi_2,\nu) - \theta_1 t + \theta_2 s
\ea
\]
has a strongly exposed maximum at some $(\ol \mu, \ol \nu,\ol t,\ol s)\in \Pt\times\Pt\times[0,T]\times[0,T]$. 
\vs

If either  of $\ol s=0$ or $\ol t=0$, we  perturb the test functions by  $\beta t^{-1}$ and reduce to the case $\ol s, \ol t>0$. 
\vs

We first assume that $\ol t, \ol s \in (0,T)$ and note that,  % and treat the case of either of  $\ol s$ or $\ol t$  is equal to $T$ later.
% at the end, and we focus for the moment on the case of $\ol t, \ol s \in (0,T)$.
since $U$ and $V$ are respectively viscosity sub-solution  and super-solution, we have 
\[\|\ol \mu \|_3 + \|\ol \nu \|_3 < + \infty,\] 

and, thus there exist optimal transport maps $S$, $\wtilde S$, $S_1,S_2$ respectively associated to the pairs $(\ol \mu, \ol \nu)$, $(\ol \nu,\ol \mu)$, $(\ol \mu,\xi_1)$ and $(\ol \nu,\xi_2)$.
\vs
Using Lemma~\ref{lemma:Uwas} and Lemma~\ref{lemma:Vwas}, we obtain that
\[
\ba
\rho -\frac{\bar t -\bar s}{\eps} &+ \theta_1 + \frac{1}{2} \dint_\R \Big(\frac{1}{\eps}(x - S(x)) -S_1(x) + 2 \alpha x \Big)^2\ol \mu(dx) + \delta  \int_\R H[\ol \mu]^2d\ol \mu\\[1.2mm]
& -(1+ \delta)\dint_\R \Big(\frac{1}{\eps}(x - S(x)) -S_1(x) + 2 \alpha x\Big) H[\ol \mu](x) \ol \mu(dx) \leq 0,
\ea
\]
\vs

and
\[
\ba
-\rho -\frac{ \bar t -\bar s}{\eps} + &\theta_2 + \frac{1}{2} \dint_\R\Big(\frac{1}{\eps}( \wtilde S(y)-y) +S_2(y) \Big)^2\ol \nu(dy) -(1 - \delta)\dint_\R \Big(\frac{1}{\eps}(\wtilde S(y)-y) S_2(y)\Big) H[\ol\nu](y) \ol \nu(dy)\\[1.2mm]
& -\delta \dint_\R H[\ol\nu]^2d\ol\nu \geq 0.
\ea
\]

Taking the difference of the two inequalities yields

\[
\ba
\dfrac{1}{2} &\dint_\R \Big(\frac{1}{\eps}(x - S(x)) -S_1(x) + 2 \alpha x\Big)^2\ol \mu(dx) 
- \frac{1}{2} \dint_\R\Big(\frac{1}{\eps}(  \wtilde S(y)-y) +S_2(y)\Big)^2\ol \nu(dy)\\
&+  \delta  \dint_\R H[\ol \mu]^2d\ol \mu  +  \delta \dint_\R H[\ol\nu]^2d\ol\nu\\[1.2mm]
&-(1 + \delta)\dint_\R \Big(\frac{1}{\eps}(x - S(x)) -S_1(x) + 2 \alpha x\Big) H[\ol \mu](x) \ol \mu(dx)\\[1.2mm]
& + (1 - \delta)\dint_\R \Big(\frac{1}{\eps}(\wtilde S(y)-y) + S_2(y)\Big) H[\ol\nu](y) \ol\mu(dy) +2 \rho + \theta_1 - \theta_2 \leq 0.
\ea
\]
\vs

The contraction properties of the Dyson flow also yields that 
\[
- \frac{1}{\eps}\left (\dint_\R \Big(x - S(x)\Big) H[\ol \mu](x) \ol \mu(dx) - \dint_\R \Big( \widetilde S(y)-y\Big) H[\ol\nu](y) \ol\nu(dy)\right ) \geq 0.
\]
\vs
To estimate the difference of the two quadratic terms, we note that, by the definition of $S$ and $\wtilde S$, 
\[
\ba
W_2^2(\ol \mu ,\ol \nu)=\dint_\R \Big((x - S(x)\Big)^2\ol \mu(dx) =  \dint_\R\Big((y - \wtilde S(y) \Big )^2\ol \nu(dy).
\ea
\]
\vs
Moreover,
\[
\delta  \dint_\R H[\ol \mu]^2d\ol \mu  + \delta \dint_\R H[\ol\nu]^2d\ol\nu \geq 0.
\]

It follows that 
\[
\ba
2 \rho \leq& \theta_2 - \theta_1 +\frac 12 \dint_\R \Big(2
\dfrac{\wtilde S(y)-y}{\ep}S_2(y) + (S_2(y))^2 \Big)\ol \nu(dy)\\[1.2mm]
& - \frac{1}{2} \dint_\R \Big(2\dfrac{x - S(x)}{\ep}(-S_1(x)  + 2\alpha x) + (-S_1(x) + 2 \alpha x)^2\Big)\ol \mu(dx) 
-\delta  \dint_\R H[\ol \mu]^2d\ol \mu  -\delta  \dint_\R H[\ol \nu]^2d\ol \nu\\[1.2mm]
& +\delta \dint_\R \Big(\dfrac {x - S(x)}{\ep}  -S_1(x) + 2 \alpha x \Big) H[\ol \mu](x)\ol \mu(dx)
+\delta \dint_\R \Big(\dfrac{\wtilde S(y)-y}{\ep} + S_2(y)\Big) H[\ol\nu](y) \ol \nu(dy).
\ea
\]
\vs
We also find that 
\[
\ba
\delta \dint_\R \dfrac {\wtilde S(y)-y}{\ep}  H[\ol \nu](y)\ol\nu(dy) \leq \delta \frac{W_2(\ol\mu,\ol \nu)}{\eps}\|\ol\nu\|_3^{\frac32} 
\ea
\]
and
\[
\ba
\delta \dint_\R \dfrac{x - S(x)}{\ep} H[\ol \mu](x)\ol\mu(dx) \leq \delta \frac{W_2(\ol\mu,\ol\nu)}{\eps}\|\ol\mu\|_3^{\frac32},
\ea
\]
\vs

while  \[\|S_1\|^2_{L^2(\ol\mu)}, \|S_2\|^2_{L^2(\ol\nu)}, |\theta_1|,|\theta_2| \leq \eta.\] %the assumption on $\psi_1$ and $
Combining all the above leads to

\be\label{est:last?}
\ba
 \rho & \leq \eta + \eps^{-1}W_2(\ol\mu,\ol\nu)\sqrt\eta + \frac\eta4 +  \alpha\eps^{-1} \sqrt{M_2(\ol\mu)}W_2(\ol\mu,\ol\nu) + \frac\delta2 \sqrt{\eta}(\|\ol\mu\|_3^{\frac32} + \|\ol\nu\|_3^{\frac32})\\
& \qquad +  \alpha \sqrt{M_2(\ol \mu)}\|\ol\mu\|_3^{\frac32} +  \delta\frac{W_2(\ol\mu,\ol\nu)}{2\eps}(\|\ol\mu\|_3^{\frac32} + \|\ol\nu\|_3^{\frac32}) - \frac\delta2\Big( \lambda \|\ol\mu\|_3^3 + \kappa\|\ol\nu\|_3^3\Big)\\
& \leq  2\eta + \eta^{1/2}\Big(\frac{W_2(\ol\mu,\ol\nu)}{\ep} + \frac\delta2 (\|\ol\mu\|_3^{\frac32} + \|\ol\nu\|_3^{\frac32}) \Big) + 
\alpha \big(M_2(\ol\mu)\big)^{1/2}\left(\frac{W_2(\ol\mu,\ol\nu)}{\ep}+ \|\ol \mu\|^\frac32_3\right) \\
& \qquad + \delta \|\ol\mu\|_3^{\frac32}\left(\frac{W_2(\ol\mu,\ol\nu)}{\eps} - \lambda \|\ol\mu\|_3^{\frac32}\right) + \delta \|\ol\nu\|_3^{\frac32}\left(\frac{W_2(\ol\mu,\ol\nu)}{\eps} - \kappa \|\ol\nu\|_3^{\frac32}\right).
\ea
\ee
\vs
Next we  pass to the $\eta \to 0$, $\alpha \to 0$, $\delta \to 0$ and, finally, $\eps \to 0$ limits. For this, we need a number of estimates which are independent of $\eta,\alpha, \delta$ or $\eps$.

%To explain such a passage to the limit, we gather first some estimates. In all what follows, $C$ denotes a constant which does not depend on $\eta,\alpha, \delta$ or $\eps$. 
\vs

As it is usual in the viscosity solution theory, we know that 
 $$W_2(\ol\mu,\ol\nu) = o(\sqrt{\eps}),$$
while Lemma \ref{lemma:estL3} yileds that
$$
\|\ol \mu\|^3_3, \|\ol\nu\|_3^3 \leq \frac{C}{\eps \delta^2}.
$$
Above, and in all what follows, $C$ denotes a constant which does not depend on $\eta,\alpha, \delta$ or $\eps$.
\vs
Furthermore, since $U\wedge A$ is bounded from above by $A$, we know that $\alpha M_2(\ol\mu)$ is also bounded and, thus, for sufficiently small $\alpha$,  we have 
$$\alpha^2M_2(\ol \mu) {\approx}C \alpha .$$
Next we argue that the last two terms in \eqref{est:last?} converge to $0$  provided that $\delta/\ep\to 0$. Indeed, either 
$$\left(\frac{W_2(\ol\mu,\ol\nu)}{\eps} - \lambda \|\ol\mu\|_3^{\frac32}\right) \leq 0$$
and, thus, the terms can be ignored or 
$$\ds  \lambda \|\ol\mu\|_3^{\frac32} \leq \frac{W_2(\ol\mu,\ol\nu)}{\eps},$$
which then implies that  $$\delta \|\ol\mu\|_3^{\frac32}\left(\frac{W_2(\ol\mu,\ol\nu)}{\eps} - \lambda \|\ol\mu\|_3^{\frac32}\right)\leq \delta \frac{W_2^2(\ol\mu,\ol\nu)}{\lambda\eps^2}.$$
%Hence, those terms will vanish in the limit provided that $\delta/\ep\to 0$. 
\vs
%$$
%\lim \frac{\delta}{\eps} = 0.
%$$
For the other terms, the estimate is straightforward provided 
$$
\lim\sqrt{\alpha} \left(\frac{1}{\sqrt \eps} + \frac{1}{\eps\delta^2}\right) = 0 \ \text{and} \  \lim \sqrt\eta\left(\frac{1}{\sqrt \eps} + \frac{1}{\delta\eps}\right) = 0.
$$
\vs
Hence, if we set, for instance,  $\delta = \eps^2$, $\alpha = \eps^{11}$ and $\eta = \eps^{12}$ ((the reason for this choice for  the parameter $\eta$  will become clear when we consider next that either  $\ol t = T$ or $\ol s = T$), we can indeed pass to the limit $\eps \to 0$ in \eqref{est:last?} and obtain the contradiction $\rho \leq 0$.
\vs
%(the reason for this choice for  the parameter $\eta$  will become clear when we consider next that either  $\ol t = T$ or $\ol s = T$)
%\vs
%
%

%
%$$
%\delta\frac{W_2(\ol\mu,\ol\nu)}{\eps}\|\ol\mu\|_3^{\frac32}  - \delta \lambda \|\ol\mu\|_3^3= \delta \|\ol\mu\|_3^{\frac32}\left(\frac{W_2(\ol\mu,\ol\nu)}{\eps} - \|\ol\mu\|_3^{\frac32}\right).
%$$
%\vs
%Observe, that either $\|\ol\mu\|_3^{\frac32}$ is bounded by some constant depending on $\eps$, and thus 
%\[\lim_{\delta \to 0} \Big(\delta\frac{W_2(\ol\mu,\ol\nu)}{\eps}\|\ol\mu\|_3^{\frac32}  - \delta \color{red}\lambda\color{black}\|\ol\mu\|_3^3\Big)=0,
%\]
%or $\|\ol\mu\|_3^{\frac32}$ will be larger $\frac{W_2(\ol\mu,\ol\nu)}{\eps}\leq C\frac{1}{\sqrt{\eps}}$ at some point, and then we can also discard this term.\\
%
%The terms in $\eta$ do not raise any particular issue. Finally, since $U\wedge A$ is bounded from above by $A$, we know that $\alpha M_2(\ol\mu)$ is also bounded and, thus,   $\lim_{\alpha\to 0}\alpha^2M_2(\ol \mu) = 0$. 
%\vs
%We finally pass to the limit as announced and obtain a contradiction since $\rho > 0$.
%\vs
It remains to treat the case of either $\ol t = T$ or $\ol s = T$ and note that it is here that the distinction between the two cases in the statement of the result plays a role. 
\vs

Assume that, for some $h > 0$, $U$ are $V$ are uniformly continuous in $\Pt\times [T-h,T]$  with a common modulus of continuity $\omega$. 
\vs

Then, since  $U(T,\cdot) - V(T,\cdot) \leq 0$ yields  $(U\wedge A)(T,\cdot) - V(T,\cdot) \leq 0$, we find that, at the point of maximum,  
%, since $(U\wedge A)(T,\cdot) - V(T,\cdot) \leq 0$, 
\beq\label{estT}
\begin{aligned}
&\omega(W_2(\ol \mu,\ol \nu) + |\ol t - \ol s|) -\frac{1}{2\eps}(W_2^2(\ol \mu,\ol \nu) + (\ol t - \ol s)^2)\\
&- \delta(E(\ol \mu) + E(\ol \nu)) - \alpha M_2(\ol \mu) + \rho|\ol t - \ol s| + \eta T + \sqrt{\eta}(\sqrt{M_2(\ol \mu)} + \sqrt{M_2(\ol \nu)})
\geq K.
 \end{aligned}
\eeq
\vs
The goal is to pass  in  the $\eta \to 0$, $\alpha \to 0$, $\delta \to 0$ and, finally, $\eps \to 0$ limits in \eqref{estT} to arrive in  a contradiction. 
\vs
This time, we do not have an estimate involving $\delta$ similar to the one of Lemma \ref{lemma:estL3}. However,  we still have that 
$$\alpha M_2(\ol \mu) \text{ is bounded}$$
and 
$$
\lim_{\eps \to 0} \frac{1}{2\eps}(W_2^2(\ol \mu,\ol \nu) + (\ol t - \ol s)^2) = 0. $$
% \ \text{and} \  (\alpha M_2(\ol \mu)) \text{ is bounded}.

Moreover, using that $$M_2(\ol \mu) \leq \frac C\alpha \ \ \text{and} \ \ M_2(\ol \nu) \leq 2W_2^2(\ol \mu,\ol\nu) + 2 M_2(\ol \mu),$$
we note that, thanks to \eqref{entropy10}, we have 
\be\label{eq:2375}
\ba
- \delta(E(\ol \mu) + E(\ol \nu)) &\leq C\delta ( \log(M_2(\ol \mu)) + \log(M_2(\ol \nu)) +1)\\
& \leq C \delta\left( -\log(\alpha) + \log\left(\frac1\eps + \frac 1 \alpha \right) +1 \right).
\ea
\ee
%where we used that $M_2(\ol \mu) \leq \frac C\alpha$ and $M_2(\ol \nu) \leq 2W_2^2(\ol \mu,\ol\nu) + 2 M_2(\ol \mu)$. 
\vs
Thus we can pass to the limit $\eps \to 0$ (with the other parameters fixed as above) in \eqref{estT} to obtain that $K \leq 0$ which is a contradiction as well.
\vs
Next assume that  $U$ and $V$ are only locally uniformly continuous and $V$ is bounded, uniformly on $\Pt\times [T-h,T]$,   from below by a  "quadratic" function. 
\vs
In this setting, we can let $\alpha= 0$ directly without  changing  the arguments above. In particular, we know that $M_2(\ol \mu)$ and $M_2(\ol \nu)$ are bounded uniformly in all parameters which we send to $0$, and thus the equivalent of \eqref{eq:2375} raises no issue in the limit $\delta \to 0$.
\vs
Hence, the result follows from the previous argument since we can restrict our attention to a bounded set where  $U$ and $V$ are uniformly continuous.
\vs

In conclusion, we have proved that $U \wedge A \leq V$, and since $A \in \R$ is arbitrary, the original claim follows.

\end{proof}

\section{Continuity and viscosity solutions properties of the value function}\label{sec:valuevisc}
%As we saw in the proof of the comparison principle, 
The  continuity properties of viscosity sub- and super-solutions especially near $t=T$ are  very important for the comparison principle. %, especially near $t=T$. 

\vs

In this section, we prove that the value functions of the optimal control problems we are considering are indeed continuous. In fact, we first prove  a regularizing effect, which implies that,  even if the terminal cost is singular as in the target problem, the value function is locally uniformly continuous  away from the terminal time. Then, we show that,  for a terminal cost which is locally uniformly continuous, the value function is locally uniformly continuous near $T$.
\vs

Once these continuity estimates are proven, we establish that the value function is a viscosity solution to \eqref{takis100}. In turn, this concludes the proof of the fact that the value function is characterized as  the unique viscosity solution to the associated HJB equation.

\subsection{A regularizing result}
This subsection is devoted to the regularizing effect property of the Dyson control problem.

\begin{theorem}\label{thm:estlip}
Assume that $F$ is locally uniformly continuous and $\nu \in \widetilde{\mathcal{P}_2}$. Then, for each $\ep>0$, the value functions $U(\cdot,t)$ and $\widehat U(\cdot,\nu,t)$ are locally uniformly continuous on $\mcl P_2$ for all $t\in [0,T-\ep]$, with moduli of continuity  depending  on $\eps$. Furthermore, if $F$ is Lipschitz continuous, then $U(\cdot,t)$ and $\widehat U(\cdot,\nu,t)$ are locally uniformly Lipschitz in $\mcl P_2$ for all $t\in [0,T-\ep]$ with constants which depend only on $\eps$.
\end{theorem}
\begin{proof}
Since the proofs for $\widehat U$ and $U$ are similar, here we concentrate on $U$. We also assume $F \equiv 0$ for the moment and consider the general case at the end. 
\vs
In what follows we fix  $t \in [0,T-\eps]$, $M > 0$ and $\mu^1,\mu^2 \in \mcl P_2$ such that $M_2(\mu^1) \leq M$ and $M_2(\mu^2) \leq M$.

% and fix an optimal control $(\ol \al^1, \ol m^1)$ for $U(\mu^1,t)$. To compare $U(\mu^2,t)$ to $U(\mu^1,t)$, we need  to find an appropriate  admissible control for $U(\mu^2,t)$. 
\vs

As in the classical optimal control theory, we would like to use for $U(\mu^2,t)$ a control similar to the one for $U(\mu^1,t)$ since this would imply that the two associated trajectories remain close to one another. Moreover, since no regularity is assumed on the terminal data, it is necessary to perturb the new admissible control so that the two trajectories in fact coincide at time $T$. Once such a control is constructed, it remains to estimate the difference of  between the two payoffs. % In order to do so, we shall use Lemma \ref{lemma:brezis}.
\vs
Given an optimal pair $(\ol \alpha^1, \ol m^1)$ for $U(\mu^1,t)$, 
we follow  an idea from \citep{tangent} and adopt the Lagrangian point of view.  
\vs
Using  Ambrosio's superposition principle (Theorem 8.2.1 in \citep{ags}), we  consider a process $(X^1_s)_{s \geq t}$ solving %solution to
$$
\frac{dX^1_s}{ds} = \ol \al ^1(s,X^1_s) + H[\ol m^1_s](X^1_s) \ \  \text{in} \ \ (t,T) \ \ \text{and} \ \ \mathcal L(X^1_s) = \ol m^1_s \ \ \text{for all \ \ $s \in [t,T]$},
$$
%such that, for all $s \in [t,T]$, $\mathcal L(X^1_s) = \ol m^1_s$. 
%We argue thanks to Ambrosio superposition principle (Theorem ? in \citep{ags}). 
and note that we can indeed use this result of \citep{ags} since, thanks to the bound obtained in Proposition \ref{prop:comp}, we know that $$\|H[\ol m^1_s]\|^2_{L^2(\ol m^1_s)} \lesssim t^{-\frac 32}.$$
It follows that,  for all $p \in (1,4/3)$,  
$$\|H[\ol m^1_s]\|^p_{L^p(\ol m^1_s)} \lesssim t^{-\frac {3p}{4}},$$
and, thus,  for some $p \in (1,4/3)$,

%In particular, $\|H[\ol m^1_s]\|^p_{L^p(\ol m^1_s)} \leq \frac{C}{t^{\frac {3p}{4}}}$. Hence, we know that 
$$
\int_t^T(\ol \al^1 + H[\ol m^1])^pm^1_sdxds < \infty,
$$
which is the required assumption for Theorem 8.2.1 in \citep{ags}. 
\vs
It follows that  the  process $(X^1_s)_{s \in [t,T]}$ is thus absolutely continuous in time in $L^p([0,1],Leb)$.
\vs
A classical argument in measure theory, that we do not detail here, implies that we can change the probability space on which $X^1_s$ is defined to be $([0,1],Leb)$.
\vs
We now argue that  we can assume  that $(X^1_s(\omega))_{s \in [t,T]}$ is non-decreasing in $\omega$ for all $s$. 
\vs
%The  change of the probability space is a classical argument in measure theory that we do not detail.
%\vs
Indeed, in view  Proposition \ref{prop:comp}, for any $\eps > 0$ we have 
$$
\sup_{s \geq t+ \eps} E[\mathcal L(X^1_s)] < \infty.
$$
Then, in view  of  the continuity of $(X^1_s)_{s \in [t,T]}$ in $L^p$, we know that the map $s \mapsto E[\mathcal L(X^1_s)]$ is continuous as well. In particular, this implies that, if $X^1_s$ is non-decreasing in $\omega$, then it remains non-decreasing for all future times. Considering a rearrangement of the probability space, we can always assume that  $X^1_{t + \frac1n}$ is non-decreasing in $\omega$. Passing to the limit $n \to \infty$ yields a process which satisfies our  requirements.
\vs
Next,   we take  $X^2_t$ to be the inverse cumulative distribution function (CDF) of $\mu^2$, and seek  a process $(X^2_s)_{s \in [t,T]}$ with initial condition $X^2_t$ such that  % solution to
\beq\label{defX2}
\frac{dX^2_s}{ds} = \ol \al ^1(s,X^1_s) + H[\ol m^2_s](X^2_s) + \frac{X^1_s-X^2_s}{T-s} \  \ \text{ in } \ \ (t,T) \ \ \text{and} \ \  X^2_t = X^1_t. 
\eeq
%and 
%\beq\label{defX2.1}
%X^2_s = X^1_s \ \ \text{for} \ \ s \geq t_c.
%\eeq
The existence of such a process, which is non-decreasing in $\omega$ on any time interval $(t, T-\delta)$ for $\delta > 0$,  was proven in Lemma \ref{lemma:brezis}.  
\vs
We now show how to extend such a process by continuity to $s= T$.
\vs

Set $m^2_s = \mathcal L(X^2_s)$. Then, since both processes are non-decreasing in $\omega$ for all $s$, we have that, for all $s \in [t,T]$, 
\[W_2^2(\ol m^1_s,\mathcal L(X^2_s)) = \int_0^1|X^1_s(\omega) - X^2(\omega)| d\omega. \]
Then Proposition~\ref{prop:decreas} implies that  
$$
\frac{d}{ds}W^2_2(\ol m^1_s,m^2_s) \leq -2\frac{W_2^2(\ol m^1_s,m^2_s)}{T-s},
$$
and, thus, 
\[W_2(\ol m^1_s,m^2_s) \leq W_2(\mu^1,\mu^2)\frac{T -s}{T-t} \ \  \text{ on}  \ \ [t,T).\] 

In particular, $m^2$ is continuous in $\mcl P_2$ and $m^2_T = \ol m^1_T$, and, thus,  $\Big(\ol \alpha^1(s, X^1_s) + \frac{X^1_s-x}{t+h-s}\Big)_{s\in [t,T]}$ is an admissible control for $U(\mu^2,t)$, from which we deduce that 
%We now estimate the control associated to $(X^2_s)_{s \in [t,T]}$ to 
$$
U(\mu^2,t) \leq \frac12\int_t^{T} \mathbb E\Big[\Big(\ol \al^1(s,X^1_s) + \frac{X^1_s-X^2_s}{t+h-s}\Big)^2\Big]ds,
$$
and, after expanding the square, % leads to
$$
U(\mu^2,t) - U(\mu^1,t) \leq \frac12\int_t^{T} \mathbb E\Big[2\ol \al^1(s,X^1_s)\frac{X^1_s-X^2_s}{t_c-s} + \Big(\frac{X^1_s-X^2_s}{t_c-s}\Big)^2\Big]ds.
$$
Using the Cauchy-Schwarz inequality and the estimate on $\mathbb E[|X^1_s-X^2_s|^2]= W_2^2(\ol m^1_s,m^2_s)$ finally imply that 

$$
\begin{aligned} 
\ds U(\mu^2,t) - U(\mu^1,t) &\leq \left(\int_t^{T}\mathbb E[(\ol\al(s,X^1_s))^2]ds\right)^{1/2}\; \left(\int_t^{T}\frac{W_2^2(\ol m^1_s,m^2_s)}{(T-s)^2}ds\right)^{1/2}\; + \frac12 \int_t^{T}\frac{W_2^2(\ol m^1_s,m^2_s)}{(T-s)^2}ds\\[1.5mm]
&\leq \Big(2 U(\mu^1,t)\Big)^{1/2}\;\frac{W_2(\mu^1,\mu^2)}{\sqrt{T -t}} + \frac12\frac{W^2_2(\mu^1,\mu^2)}{T -t}.
\end{aligned}
$$
\vs
Recalling that $W_2^2$ is locally Lipschitz continuous and the bound on $U$ given in Proposition \ref{prop:bounds}, we deduce the Lipschitz regularity of $U(\cdot,t)$ from the previous estimate.
\vs
%When $F \ne 0$, the argument is similar. 
The Lipschitz case follows directly from the estimate on the trajectories $W_2(\ol m^1_s,m^2_s) \leq W_2(\mu^1,\mu^2)\frac{T -s}{T-t}$. 
\vs
When $F$ is locally uniformly continuous, we recall that, in view of 
Proposition \ref{prop:comp}, we can restrict our attention to a bounded set on  which $F$ is uniformly continuous. And thus a similar estimate once again yields the result.

\end{proof}

From the regularity in space of the value function, we can now deduce the regularity in time.
%because of the regularity of the optimal trajectories, as the next result shows.
\begin{proposition}\label{prop:regularitytime}
Assume that $F$ is locally uniformly continuous and $\nu \in \widetilde{\mathcal{P}_2}$. Then, the value functions $U(\cdot,\cdot)$ and $\widehat U(\cdot,\nu,\cdot)$ are locally uniformly continuous on $\mcl P_2\times [0,T-\eps]$ for any $\eps > 0$ with moduli of continuity which only depend on $\eps$.
\end{proposition}
\begin{proof}
We assume $F \equiv 0$, and leave the general case to the reader. 
\vs
Let $0 \leq t < s < T-\eps$ for some $\eps > 0$ and $\mu \in \mcl P_2$, and note that we always have
$$
U(\mu,t) \leq U(S_{s-t}\mu,s) \leq U(\mu,s) + C_{T-\eps,M_2{\mu}}W_2(S_{s-t}\mu,\mu) \leq U(\mu,s) + C_{T-\eps,M_2{\mu}}\sqrt{s-t}.
$$

It thus remains to obtain a locally modulus of continuity in the other direction. 
\vs
Consider an optimal trajectory $(\ol \alpha, \ol m)$ in $U(\mu,t)$, and denote by $\ol X_u$ the inverse cdf of $\ol m_u$. Then, it follows from Proposition \ref{takis1250} that there exists a modulus of continuity $\omega$ depending only $T-t$ and $U(\mu,t)$ such that \[W_2(\ol m_s,\mu) \leq \omega (|s-t|).\] 
\vs
It now remains to choose an admissible pair for $U(\mu,s)$ which yields a cost close to $U(\mu,t)$ while the two trajectories remain close to each other. Since the dynamics are homogeneous, we simply use a time rescaling of the control.
\vs
In order to do so, we use once again Lemma \ref{lemma:brezis} to find a process on $[0,1]$ solving 
$$
\frac{d}{du} X_u = \frac{T-t}{T-s}\ol \alpha(u,\ol X_u) + H[\mcl L(X_u)](X_u) \ \text{ in } \ (s,T-\eps),
$$
which is non-decreasing for all $u \in (s,T)$, and such that $\mathcal L(X_s) = \mu$. 
\vs
The fact that the Dyson flow contracts distance yields 
$$
\frac{d}{du}W_2^2(\mathcal L(X_u),\ol m_u) \leq 2 W_2(\mathcal L(X_u),\ol m_u)\Big(\mathbb E \Big[\Big(\Big(\frac{T-t}{T-s}-1\Big)\ol \alpha(u,\ol X_u)\Big)^2\Big]\Big)^{1/2}.
$$
It then follows that, for all $u \in [s,T-\eps]$, 

$$
W_2(\mathcal L(X_u),\ol m_u) \leq W_2(\mu,\ol m_s) + \left(\frac{T-t}{T-s}-1\right)\int_s^u\mathbb E[(\ol \alpha(u',\ol X_{u'})^2]du' \leq W_2(\mu,\ol m_s) + \left(\frac{T-t}{T-s}-1\right)U(t,\mu).
$$

Using dynamic programming, the sub-optimality of the trajectory and the optimality of $(\ol \alpha,\ol m)$,  we get, after  changing time, that  
$$
\begin{aligned}
\ds U(\mu,s) &\leq \int_s^{T-\eps} \frac12\mathbb E[(\frac{T-t}{T-s}\ol \alpha(u,\ol X_u))^2]du + U(\mathcal L(X_{T-\eps}),T-\eps)\\[1.2mm]
&\leq U(\mu,t) - U( \ol m_{T-\eps},T-\eps) + U(\mathcal L(X_{T-\eps}),T-\eps)\\[1.2mm]
& \leq U(\mu,t) + C_{T-\eps,M_2(\mu)}W_2(\ol m_{T-\eps},\mathcal L(X_{T-\eps}))\\[1.2mm]
& \leq U(\mu,t) + C_{T-\eps,M_2(\mu)}\Big(W_2(\mu,\ol m_s) + \Big(\frac{T-t}{T-s}-1\Big)U(\mu,t)\Big).
\end{aligned}
$$
\vs
We are then able to conclude using Proposition \ref{prop:comp} to bound $W_2(\mu,\ol m_s)$ by a modulus of $(s-t)$ depending only on $U(\mu,t)$, and by remarking that the last term in the estimate above goes to $0$ as $|s-t| \to 0$. 

\vs
The global continuity then follows from combining this time estimate with the previous spatial one.

\end{proof}

\subsection{The continuity near the terminal time for a regular terminal cost}
The estimate obtained in the previous section is a  regularizing effect that  is supposed to deteriorate as $t \to T$. 
However,  when the terminal data is uniformly continuous, we can show quite simply an estimate all the way up to the terminal time.
\begin{proposition}
Assume that the terminal data $U_T$ is locally uniformly continuous on $\mcl P_2$. Then, $U$ is locally uniformly continuous on $ \mcl P_2\times [0,T]$.
\end{proposition}
\begin{proof}
Fix $t \in [0,T)$, $\mu, \mu' \in \mcl P_2$ and let   $(\ol \alpha, \ol m)$ be an optimal control for $U(t,\mu)$. Arguing as in the proof of Theorem \ref{thm:estlip}, we can  use, at the Lagrangian level thanks to Lemma \ref{lemma:brezis}, the same control in $U(t,\mu')$, which yields an admissible control $(\alpha',m')$. 
\vs
Since the controls have the same cost, we get 
$$
|U(\mu,t)- U(\mu',t)|\leq |U_T(\ol m_T) - U_T(m'_T)|.
$$
Hence, the local uniform continuity in space follows because  the Dyson flow contracts distances and, thus, $W_2(\ol m_t,m'_T) \leq W_2(\mu,\mu')$,  Proposition \ref{prop:comp} yields $M_2(\ol m_T),M_2(m'_T) \leq K$ for $K$ depending only on $U(\mu,t), U(\mu',t)$, and, finally,  the local uniform continuity of $U_T$ and the local boundedness of $U$.
\vs
The existence of such $K$ follows from the fact that, if $M$ is a bound for  $M_2(\mu),M_2(\mu')$,  then, $U(\mu,t)$ and $U(\mu',t)$ are bounded by a constant depending only on $\sup \{U_T(m)| M_2(m) \leq C\}$ for some $C$ depending only on $M$. %This yields a  $K$ as above.
\vs

Finally, noting that $\omega_K$ a modulus of continuity of $U_T$ on $\{m, M_2(m)\leq K\}$, we obtain
$$
|U(\mu,t)- U(\mu',t)| \leq \omega_K(W_2(\mu,\mu')).
$$
% then yields the result.
\vs

Since the local uniform continuity in time can be established as in Proposition \ref{prop:regularitytime}, we leave the details to the reader. % so we do not detail it here again.

\end{proof}

\subsection{The viscosity super-solution property}

We check that $U$ and $\what U$ are viscosity super-solutions to \eqref{takis100}.

\begin{proposition}\label{super}
For any $\delta> 0$, $\ds \eta \in (0,\min(1/4; \delta(1-\delta)/2))$, %\frac{\delta(1-\delta)}{2}))$, 
any $C^{1,1}-$function $\phi :  \Pt \times (0,T)\to \R$, any $\xi \in \mcl P_2$ such that $M_2(\xi)\leq \eta$, and any  minimum point $(\ol \mu, \ol t) \in  \Pt \times (0,T)$ of the map $(\mu,t)  \to \what U(\mu, t) -\phi(\mu, t) + \delta E(\mu) + \mathcal I_2(\xi,\mu)$ on $\Pt \times (0,T)$, we have that $\ol \mu \in L^3$ and 
\[\ba
-\partial_t \phi(\ol \mu, \ol t) +& \dfrac{1}{2(1-2\eta)} \dint_\R(D_\mu \phi(\ol \mu, \ol t))^2d\ol \mu\\
& - \left(1 - \frac{\delta}{1-2\eta}\right)\dint_\R D_\mu \phi H(\ol \mu) d\ol \mu- \left( \delta-\eta - \dfrac{\delta^2}{2(1-2\eta)}\right)\dint_\R H[\ol \mu]^2d\ol \mu \geq F[\ol \mu] -\frac \eta 2.
\ea
\]
The same claim holds for $\what U(\cdot, \nu, \cdot)$ provided we assume, in addition,  that $\nu \in \wtilde \Pt.$ 
%(ii)~Assume that $\nu \in \wtilde \Pt.$ For any $\delta> 0$, $\eta \in (0,\min(\frac14;\delta\frac{1-\delta}{2}))$, any $C^{1,1}-$function $\phi :  \Pt \times (0,T)\to \R$, $\xi \in \mcl P_2$ such that $M_2(\xi) \leq \eta$,  and any minimum point  $(\ol \mu, \ol t) \in  \Pt \times (0,T)$ of the map $(\mu,t)  \to \what U(\mu, \nu, t) -\phi(\mu, t) + \delta E(\mu) + \mathcal I_2(\xi,\mu)$ on $\Pt \times (0,T)$, we have that $\ol \mu \in L^3$ and 
%\[\ba
%-\partial_t \phi(\ol \mu, \ol t) +& \dfrac{1}{2(1-2\eta)} \dint_\R(D_\mu \phi(\ol \mu, \ol t))^2d\ol \mu \\
%&- \left(1 - \frac{\delta}{1-2\eta}\right)\dint_\R D_\mu \phi H(\ol \mu) d\ol \mu- ( \delta-\eta - \dfrac{\delta^2}{2(1-2\eta)})\dint_\R H[\ol \mu]^2d\ol \mu \geq  F[\ol \mu] -\frac \eta 2.
%\ea\]
\end{proposition}
\begin{proof}
Since the proofs of the two claims are almost identical, here we provide the details for $\what U(\cdot, \nu, \cdot)$. 
%\vs
%Note that the assumption that 
%$\nu\in \wtilde \Pt$ implies that $\what U(\mu,\nu,t)<\oo$ for all $(\mu,t)\in \Pt\times [0,T)$.

\vs
Let $(\ol \mu, \ol t)$ be as in the statement of the proposition and recall that  
$\nu\in \wtilde \Pt$ implies that $\what U(\mu,\nu,t)<\oo$ for all $(\mu,t)\in \Pt\times [0,T)$. Hence, since $\what U$ is bounded near $\ol t$, it follows that $E(\ol \mu) < \infty$.
\vs

Take $0 < \theta^* < T-  \bar t$. Since $\what U$ satisfies the dynamic programming principle, we have
\beq\label{takis550}
\what U(\ol \mu, \nu, \ol t) = \inf_{(\alpha,m) \in \mcl A(\ol \mu,\nu,\ol t)} \left \{\dint_{\ol t}^{\ol t+\theta^\ast}\Big[\dint_\R \dfrac 12\alpha^2_s m_s dx +F[m_s] \Big]ds + \what U(m_{\ol t+\theta^\ast}, \nu, \ol  t + \theta^\ast) \right\}. 
\eeq
It follows from the lower semicontinuity of  $\what U$ in $\mu$  and the convexity in $(\alpha m,m)$ of the cost  that the infimum in \eqref{takis550} is achieved at some $(\alpha^\ast,m^\ast) \in \mcl A(\ol \mu,\nu,\ol t)$. 
\vs
Moreover, again from dynamic programming, we find that, for all $0 < \theta \leq \theta^\ast$, 
\[
\what U(\ol \mu, \nu, \ol t) \geq \dint_{\ol t}^{ \ol t+\theta}\Big[\dint_\R \dfrac 12(\alpha_s^\ast)^2 m_s^\ast dx +F[m_s^\ast] \Big]ds + \what U(m^\ast_{\ol t+\theta}, \nu, \ol t + \theta).
\]
The definition of  $(\ol \mu, \ol t)$ then implies that 
\[\ba
0 \geq  \dint_{\ol t  }^{ \ol t+\theta}&\Big[\dint_\R \dfrac 12(\alpha_s^\ast)^2 m_s^\ast dx +F[m_s^\ast] \Big]ds -\delta\big(E(m^\ast_{\ol t+ \theta}) - E(\ol \mu)\big )\\[1.2mm]
&+ \phi(m^\ast_{\ol t + \theta}, \ol t + \theta) - \phi(\ol \mu, \ol  t) - \mathcal I_2(\xi,m^\ast_{\ol t + \theta}) + \mathcal I_2(\xi,\ol \mu).
\ea
\]

Since $E(\ol \mu) < \infty$,  Proposition~\ref{prop:comp} yields  that, for all $s \geq t, E(m_s^\ast) < \infty$ and 
\[
E(m^\ast_{\ol t + \theta}) - E(\ol \mu) = -\dint_{\ol t}^{ \ol t + \theta}\dint_\R H[m^\ast_s](\alpha^\ast_s + H[m^\ast_s])m^\ast_s dx ds.
\]
Using this relation, the $C^{1,1}-$regularity of $\phi$ and the Lipschitz continuity of $\mathcal I_2$, we deduce that there exists a constant $C_\phi>0$, depending only on $\phi$, such that
\beq\label{alex10}
\bal
0 \geq&  \dint_{\ol t}^{\ol t + \theta }\Big[\dint_\R \left[\dfrac 12(\alpha_s^\ast)^2 + \delta H[m^\ast_s](\alpha^\ast_s + H[m^\ast_s])\right]m_s^\ast dx +F[m_s^\ast] \Big]ds
+ \theta \partial_t \phi( \ol \mu, \ol t)\\[1.2mm]
& +  \dint_{\ol t}^{\ol t+\theta } \dint_\R D_\mu \phi(m^\ast_s,s)\big(\alpha^\ast_s + H[m^\ast_s]\big)m_s^\ast ds -C_\phi\left(\theta^2 +  W^2_2(m_{ \ol t + \theta}^\ast, \ol \mu)\right)
 - \eta W_2(m_{ \ol t + \theta}^\ast, \ol \mu).
\eal
\eeq
Moreover, it follows from 
\[
W_2(m_{\ol t + \theta}^\ast,\ol \mu) \leq \Big(\theta \; \dint_{\ol t }^{\ol t +\theta }\dint_\R\big (\alpha^\ast_s + H[m^\ast_s]\big )^2m^\ast_s dx ds\Big)^{1/2},
\]
that $$W^2_2(m_{ \ol t + \theta}^\ast, \ol \mu)= o(\theta).$$ 

To handle  the term in $\eta$, we need to compensate it with terms in the equation, so we first use the estimate
\beq\label{eq:esttraj2}
\eta W_2(m_{\ol t + \theta}^\ast,\ol \mu) \leq \frac 12 \eta \theta + \eta\frac12\dint_\R\big (\alpha^\ast_s + H[m^\ast_s]\big )^2m^\ast_s dx ds.
\eeq
The first term on the right hand side of \eqref{eq:esttraj2} is treated as an error  in $\eta$, while the second one will be compensated by other terms. 
\vs
Using \eqref{eq:esttraj2} we rewrite \eqref{alex10} as %We now write
\vs
$$
\bal
0 \geq&  \dint_{\ol t}^{\ol t + \theta }\Big[\dint_\R \Big[\left(\dfrac 12-\eta\right)(\alpha_s^\ast)^2 + (\delta H[m^\ast_s] +D_\mu \phi(m^\ast_s,s)) \alpha^\ast_s\\
&\quad \quad\quad+ H[m^\ast_s](D_\mu \phi(m^\ast_s,s) + (\delta - \eta) H[m^\ast_s])\Big]m_s^\ast dx +F[m_s^\ast] \Big]ds
+ \theta \partial_t \phi( \ol \mu, \ol t)
  -C_\phi\, o(\theta) - \frac12 \eta \theta,
\eal
$$
and, thus, 
\beq\label{eq:417}
\bal
0 \geq&  \dint_{\ol t}^{ \ol t + \theta}\Big[\dint_\R \left\{-\dfrac{1}{2(1-2\eta)}(D_\mu \phi(m^\ast_s,s) +\delta H[m^\ast_s])^2 + H[m^\ast_s](D_\mu \phi(m^\ast_s,s) +(\delta-\eta) H[m^\ast_s] )\right\}m_s^\ast dx \\[1.2mm]
&+ F[m_s^\ast] \Big]ds+ \theta \partial_t \phi(\ol \mu ,\ol t) -C_\phi\,o(\theta) - \frac 12 \eta \theta.
\eal
\eeq
It follows from \eqref{eq:417}  that there exists   constant $C>0$ depending on $\delta,\eta$ and $\phi$ such that
$$
\int_{\ol t}^{\ol t + \theta} H[m^\ast_s]^2 m_s^\ast dxds \leq C\theta.
$$
Thus, recalling Cotlar's identity, the lower semicontinuity of the $L^3$ norm for the weak topology and the continuity of $s \mapsto m^\ast_s$ in $\mcl P_2$, we conclude  that $\ol \mu \in L^3$. 
\vs
Moreover, dividing by $\theta$ and letting $\theta \to 0$  in \eqref{eq:417} we obtain the viscosity solution relation.

\end{proof}

\subsection{The viscosity sub-solution property}

We now establish that $\hat U$ and $U$ are viscosity sub-solutions to \eqref{takis100}. As mentioned in Section \ref{sec:visc}, a truncation of unbounded sub-solutions is typically needed in order to obtain a comparison principle. To highlight the techniques we use, we start with proving a viscosity relation in the non-truncated case, and then show how it implies the required estimate in the truncated case. 

\begin{proposition}
For any $\frac12>\delta >\eta > 0$, any $C^{1,1}-$function $\phi : \Pt \times (0,T)\to \R$, any $\xi \in \mcl P_2$ such that $M_2(\xi) \leq \eta$ and any  strongly exposed maximum point of $(\ol \mu, \ol t ) \in (0,T)\times \Pt$ of $(\mu, t) \to U(\mu, t) -\phi(\mu,t) - \delta E(\mu) - \mathcal I_2(\xi,\mu)$ on $\Pt \times (0,T)$, we have that $\ol \mu \in L^3$ and 
\[
-\pt_t \phi(\ol \mu, \ol t ) + (\dfrac 12 -\eta)\dint_\R(D_\mu \phi(\ol \mu, \ol t ))^2d\ol \mu  -(1+ \delta(1 -2\eta))\int_\R D_\mu \phi H(\ol \mu) d\ol \mu + (\delta -\eta)\int_\R H(\ol \mu)^2d\ol \mu \leq F[\ol \mu].
\]
The same claim holds for $\widehat U(\cdot, \nu, \cdot)$ provided we assume, in addition,  that  $\nu\in \wtilde P$.
%(ii)~Fix  $\nu\in \wtilde P$. For any $\frac 12 >\delta> \eta > 0$, any $C^{1,1}$ function $\phi : \Pt \times (0,T)\to \R$, $\xi \in \mcl P_2$ such that $M_2(\xi)\leq \eta$ and strongly exposed maximum point  of $(\ol \mu, \ol t ) \in (0,T)\times \Pt$ of $(\mu, t) \to \widehat U(\mu, \nu, t) -\phi(\mu,t) - \delta E(\mu) - \mcl I_2(\xi,\mu)$ on $\Pt \times (0,T)$, we have that $\ol \mu \in L^3$ and 
%\[
%-\pt_t \phi(\ol \mu, \ol t ) + (\dfrac 12 -\eta)\dint_\R(D_\mu \phi(\ol \mu, \ol t ))^2d\ol \mu  -(1+ \delta(1 -2\eta))\int_\R D_\mu \phi H(\ol \mu) d\ol \mu + (\delta -\eta)\int_\R H(\ol \mu)^2d\ol \mu \leq F[\ol \mu].
%\]
\end{proposition}

\begin{proof}

Since the arguments are identical, we present the details for $ U$. 
\vs
It follows from the definition of  $U$ that, for any $\theta \in (0,T-\ol t)$ and any admissible pair $(\alpha,m)$ for $ U(\ol \mu, \ol t)$,
\beq\label{eq:529}
 U(\ol \mu, \ol t) \leq \dint_{\ol t}^{\ol t + \theta}\Big[\dint_\R \dfrac 12 \alpha^2_s m_s dx  +F[m_s] \Big] ds +  U(m_{\ol t + \theta},  \ol t + \theta)).
\eeq

Since $ U$ is locally bounded, we  have that $E(\ol \mu ) < \infty$. Then it follows from Proposition~\ref{prop:comp}  that, for any admissible pair $(\alpha,m)$, we have  $E(m_{\ol t + \theta}) < \infty$. %, see Proposition \ref{prop:comp}. 
\vs

Moreover, the fact that $(\ol \mu, \ol t)$ is a maximum point  gives 
\[
\delta(E(\ol \mu) - E(m_{\ol  t + \theta}) )\leq \dint_{\ol t}^{\ol t + \theta} \Big[\dint_\R \dfrac 12 \alpha^2_s m_s dx + F[m_s] \Big]ds + \phi( m_{\ol t + \theta}, \ol t + \theta) -\phi(\ol \mu, \ol t) + \mathcal I_2(\xi,m_{\ol t + \theta}) - \mathcal I_2(\xi,\ol \mu).
\]
We now proceed  as in the super-solution case. We use the $C^{1,1}-$regularity of $\phi$, the $\eta$-Lipschitz estimate on $\mathcal I_2(\xi,\cdot)$ as well as the expression of the entropy to deduce that, for some $C>0$ that depends only on $\phi$, 

$$
\bal
\delta \dint_{\ol t }^{\ol t + \theta}\dint_\R & H[m_s](\alpha_s + H[m_s])m_s dx ds \leq \dint_{\ol t}^{\ol t + \theta} \dint_\R \dfrac 12 \alpha^2_s m_s dx ds +  \theta \partial_t \phi(\ol \mu, \ol t) + \\
&+\dint_{\ol t }^{\ol t + \theta}\dint_\R D_\mu \phi(m_s,s)(\alpha_s + H[m_s])m_s dx ds
+ C(\theta^2 + W_2^2(m_{\bar t + \theta},\ol \mu)) + \eta W_2(m_{\ol t + \theta},\ol \mu).
\eal
$$
\vs
Using above that  $\eta W_2(m_{\ol t + \theta},\ol \mu)$ is bounded as in \eqref{eq:esttraj2}, we obtain
$$
\bal
(\delta- \eta) \dint_{\ol t }^{\ol t + \theta}\dint_\R & (H[m_s])^2m_s dx ds \leq\left(\dfrac 12 + \eta\right) \dint_{\ol t}^{\ol t + \theta} \dint_\R  \alpha^2_s m_s dx ds +  \theta \partial_t \phi(\ol \mu, \ol t) \\
&+ \dint_{\ol t }^{\ol t + \theta}\dint_\R D_\mu \phi(m_s,s) H[m_s]m_s dx ds
+\dint_{\ol t }^{\ol t + \theta}\dint_\R(D_\mu \phi(m_s,s)- \delta H[m_s])\alpha_s m_s dx ds\\
&\qquad + C(\theta^2 + W_2^2(m_{\bar t + \theta},\ol \mu)) + \frac12\eta\theta .
\eal
$$

We choose the control 
\[\alpha_s(x) = -(D_\mu \phi(m_s, s)(x) -\delta H[m_s](x))\mathbb{1}_{\{s \geq \ol t + \eps\}},\]
which is admissible because $D_\mu \phi$ is smooth, and observe that with this choice the  state equation reduces to a standard Dyson equation with a  drift, which can be handled using Proposition~\ref{T2}. This leads to the estimate
\beq\label{eq:530}
\bal
(\delta- \eta) \dint_{\ol t }^{\ol t + \theta}\dint_\R & (H[m_s])^2m_s dx ds \leq  \theta \partial_t \phi(\ol \mu, \ol t) + \dint_{\ol t }^{\ol t + \theta}\dint_\R D_\mu \phi(m_s,s) H[m_s]m_s dx ds+\\
&+\left(\eta-\frac12\right)\dint_{\ol t + \eps}^{\ol t + \theta}\dint_\R(D_\mu \phi(m_s,s)- \delta H[m_s])^2 m_s dx ds
+ C(\theta^2 + W_2^2(m_{\bar t + \theta},\ol \mu)) + \frac12\eta\theta .
\eal
\eeq

\vs

Now that we can let $\eps \to 0$ in \eqref{eq:530}, using the same stability argument which was employed  in the  proof
of the uniqueness of the controlled Dyson equation established in Proposition \ref{T2}. 
\vs
%In fact it is not this uniqueness which is used but rather the same stability argument which was used in its proof). 
Once this limit in $\eps$ is taken, we can divide by $\theta$ and let $\theta \to 0$ to obtain the viscosity relation. 
\vs

The fact that $\bar \mu \in L^3$ follows from the same argument as in the super-solution case.

\end{proof}

We now turn to the case in which the sub-solution is truncated, which we recall is a fundamental step in order to establish the comparison principle. The following proof is quite technical, although all  the arguments  used are standard in the viscosity solution theory. The main issue is that those techniques do not combine  particularly well with each other. Thus, we shall establish the next result via a stability argument, which is the main reason why the error terms in $\eta$ appear differently here.

\begin{proposition}
For any $\frac12>\delta >\eta > 0$, any $C^{1,1}$ function $\phi : \Pt \times (0,T)\to \R$, any $\xi \in \mcl P_2$ such that  $M_2(\xi) \leq \eta$, any $A \in \mathbb R\cup\{+\infty\}$ and any exposed maximum point  $(\ol \mu, \ol t ) \in (0,T)\times \Pt$ of $(\mu, t) \to (U\wedge A)(\mu, t) -\phi(\mu,t) - \delta E(\mu) - \mathcal I_2(\xi,\mu)$ on $\Pt \times (0,T)$, we have that $\ol \mu \in L^3$ and 
$$
\bal
\left(\delta+ \frac12\delta^2\right)\dint_\R & (H[\ol \mu])^2\ol \mu dx -(1+\delta)\sqrt\eta\|H[\ol \mu]\|_{L^2(\ol \mu)}- \partial_t \phi(\ol \mu, \ol t) - ( 1+\delta)\dint_\R D_\mu \phi(\ol \mu,\ol t) H[\ol \mu]\ol \mu dx+\\
&+\frac12\dint_\R(D_\mu \phi(\ol \mu,\ol t))^2 \ol \mu dx \leq F[\ol \mu]+ \frac12\eta + \sqrt{\eta}\|D_\mu\phi(\ol \mu,\ol t)\|_{L^2(\ol \mu)}.
\eal
$$
The same claim holds true for $\what U(\cdot,\nu,\cdot)$ provided we assume, in addition, that $\nu\in \wtilde P$.
%(ii)~Fix  $\nu\in \wtilde P$. For any $\frac 12 >\delta> \eta > 0$, any $C^{1,1}$ function $\phi : \Pt \times (0,T)\to \R$, $\xi \in \mcl P_2$ such that $M_2(\xi)\leq \eta$, any $A \in \mathbb R \cup \{+\infty\}$ and any exposed maximum point $(\ol \mu, \ol t ) \in (0,T)\times \Pt$ of $(\mu, t) \to (\wtilde U\wedge A)(\mu, \nu, t) -\phi(\mu,t) - \delta E(\mu) - \mcl I_2(\xi,\mu)$ on $\Pt \times (0,T)$, we have that $\ol \mu \in L^3$ and 
%$$
%\bal
%(\delta+ \frac12\delta^2)\dint_\R & (H[\ol \mu])^2\ol \mu dx -(1+\delta)\sqrt\eta\|H[\ol \mu]\|_{L^2(\ol \mu)}- \partial_t \phi(\ol \mu, \ol t) - ( 1+\delta)\dint_\R D_\mu \phi(\ol \mu,\ol t) H[\ol \mu]\ol \mu dx+\\
%&+\frac12\dint_\R(D_\mu \phi(\ol \mu,\ol t))^2 \ol \mu dx \leq F[\ol \mu]+ \frac12\eta + \sqrt{\eta}\|D_\mu\phi(\ol \mu,\ol t)\|_{L^2(\ol \mu)},
%\eal
%$$
\end{proposition}
\begin{proof}
Once again, since the two proofs are identical here  we show the details for $U$. % on the first point.
\vs 
%Let $A \in \mathbb R$, $\frac12>\delta >\eta > 0$, $\phi$ be a $C^{1,1}$ function, $\xi \in \mcl P_2$ such that $M_2(\xi) \leq \eta$, such that $(\ol \mu, \ol t ) \in \Pt\times (0,T)$ is a point of exposed maximum of $(\mu, t) \to ( U\wedge A)(\mu, t) -\phi(\mu,t) - \delta E(\mu) - \mathcal I_2(\xi,\mu)$ on $\Pt \times (0,T)$. 
%\vs
We  introduce  the approximation $(\psi_\gamma)_{\gamma > 0}$ of $\mathcal I_2(\xi,\cdot)$ given by Theorem~ \ref{thm:approxI2}, and, following  an approach similar to the one of \citep{BL2},  we use a smooth regularization $G_\kappa(x,y)$ of $x\wedge y$ such that $\partial_x G_\kappa,\partial_y G_\kappa$ are valued in $(\kappa,1-\kappa)$ for $\kappa \in (0,\frac12)$.
\vs
%We want to consider a point of maximum of $(\mu, t) \to G_\kappa(\wtilde U(\mu, \nu, t),A) -\phi(\mu,t) - \delta E(\mu) - \psi_\gamma(\mu)$. 

Let $\phi_\gamma = \phi + \psi_\gamma$ ,which is in $C^{1,1}$, and note that,  up to a slight change of the test function $\phi_\gamma$, we can restrict our attention to a bounded set. Then, thanks to Theorem \ref{thm:stegall}, for any $\omega_\kappa$, there exists $\xi_\kappa \in \mcl P_2$ with $M_2(\xi_\kappa) \leq \omega_\kappa$ such that $(\mu, t) \to G_\kappa(\wtilde U(\mu, \nu, t),A) -\phi_\gamma(\mu,t) - \delta E(\mu) - \mathcal I_2(\xi_\kappa,\mu)$ has strongly exposed maximum point at some $(\ol \mu_\kappa,\ol t_\kappa)$. 

\vs

Using \eqref{eq:529} and the fact that $G_\kappa(\cdot,A)$ is invertible, we deduce
$$
\bal
0\leq \dint_{\ol t_\kappa}^{\ol t_\kappa + \theta} \Big[\dint_\R \dfrac 12 \alpha^2_s m_s dx + F[m_s] \Big]ds + &G_\kappa(\cdot,A)^{-1}\Big( \phi_\gamma( m_{\ol t_\kappa + \theta}, \ol t_\kappa + \theta) + \mathcal I_2(\xi_\kappa,m_{\ol t_\kappa + \theta})  +\delta E(m_{\ol  t_\kappa + \theta} )\Big)\\
& - G_\kappa(\cdot,A)^{-1}\Big( \phi_\gamma(\ol \mu_\kappa, \ol t)+ \mathcal I_2(\xi_\kappa,\ol \mu_\kappa) + \delta E(\ol \mu_\kappa) \Big).
\eal
$$
In view of the smoothness of $G_\kappa$, there exists $C_\kappa > 0$ such that
$$
\bal
&0\leq \dint_{\ol t_\kappa}^{\ol t_\kappa + \theta} \Big[\dint_\R \dfrac 12 \alpha^2_s m_s dx + F[m_s] \Big]ds\\[1.2mm]
& + C_\kappa\Big( \phi_\gamma( m_{\ol t_\kappa + \theta}, \ol t_\kappa + \theta) -\phi(\ol \mu_\kappa, \ol t_\kappa) + \mathcal I_2(\xi_\kappa,m_{\ol t_\kappa + \theta}) - \mathcal I_2(\xi_\kappa,\ol \mu_\kappa) -\delta(E(\ol \mu_\kappa) - E(m_{\ol  t_\kappa + \theta}) )\Big)^2 +\\[1.2mm]
& \frac{1}{\partial_xG_\kappa(\wtilde U(\ol \mu_\kappa,\ol t_\kappa),A)}\Big( \phi_\gamma( m_{\ol t_\kappa + \theta}, \ol t_\kappa + \theta) -\phi_\gamma(\ol \mu_\kappa, \ol t_\kappa) + \mathcal I_2(\xi_\kappa,m_{\ol t_\kappa + \theta}) - \mathcal I_2(\xi_\kappa,\ol \mu_\kappa) -\delta(E(\ol \mu_\kappa) - E(m_{\ol  t_\kappa + \theta}) )\Big).
\eal
$$
\vs
We now argue by analogy with the case $A = +\infty$. Using similarly defined  controls $\alpha$ and thanks to Proposition \ref{takis1250}, we know that the associated path $(m_s)_{s \geq \ol t_\kappa}$ can be taken to be  uniformly continuous. 
\vs
Then, in view of  the regularity of $\phi_\gamma$ and $\mathcal I_2(\xi_\kappa,\cdot)$, we have 
$$(\phi_\gamma( m_{\ol t_\kappa + \theta}, \ol t_\kappa + \theta) -\phi_\gamma(\ol \mu_\kappa, \ol t_\kappa))^2 = o(\theta) \ \text{ and} \ (\mathcal I_2(\xi_\kappa,m_{\ol t_\kappa + \theta}) - \mathcal I_2(\xi_\kappa,\ol \mu_\kappa))^2 = o(\theta).$$ 
\vs
Let $\lambda_\kappa= \partial_xG_\kappa(\wtilde U(\ol \mu_\kappa,\ol t_\kappa),A) \in (\kappa,1-\kappa)$. It follows that 

$$
\bal
0\leq &\dint_{\ol t_\kappa}^{\ol t_\kappa + \theta} \Big[\dint_\R \dfrac 12 \alpha^2_s m_s dx + F[m_s] \Big]ds +C_\kappa\delta^2(E(\ol \mu_\kappa) - E(m_{\ol  t_\kappa + \theta}) )^2\\[1.2mm]
&+ \lambda_\kappa^{-1}\Big( \phi_\gamma( m_{\ol t_\kappa + \theta}, \ol t_\kappa + \theta) -\phi_\gamma(\ol \mu_\kappa, \ol t_\kappa) + \mathcal I_2(\xi_\kappa,m_{\ol t_\kappa + \theta}) - \mathcal I_2(\xi_\kappa,\ol \mu_\kappa) -\delta(E(\ol \mu_\kappa) - E(m_{\ol  t_\kappa + \theta}) )\Big) + o(\theta).
\eal
$$

%Ignoring for now the term $C_\kappa\delta^2(E(\ol \mu_\kappa) - E(m_{\ol  t_\kappa + \theta}) )^2$, 
We argue as in the case $A = + \infty$, using  $\lambda_\kappa^{-1}\delta, \lambda_\kappa^{-1}\phi_\gamma$ and  $\lambda_\kappa^{-1}\eta_\kappa$ in place of  $\delta, \phi$ and $\eta$, and deduce that 
% \leftarrow \lambda_\kappa^{-1}\delta$, $\phi \leftarrow \lambda_\kappa^{-1}\phi_\gamma$ and $\eta \leftarrow \lambda_\kappa^{-1}\eta_\kappa$, 

$$
\bal
&\frac{\delta- \eta_\kappa}{\lambda_\kappa} \dint_{\ol t_\kappa }^{\ol t_\kappa + \theta}\dint_\R  (H[m_s])^2m_s dx ds -C_\kappa\delta^2(E(\ol \mu_\kappa) - E(m_{\ol  t_\kappa + \theta}) )^2\\
&\quad\leq  \theta \lambda_\kappa^{-1}\partial_t \phi_\gamma(\ol \mu_\kappa, \ol t_\kappa) + \lambda_\kappa^{-1}\dint_{\ol t_\kappa }^{\ol t_\kappa + \theta}\dint_\R D_\mu \phi_\gamma(m_s,s) H[m_s]m_s dx ds+\\
&\quad\quad+(\lambda_\kappa^{-1}\eta_\kappa-\frac12)\dint_{\ol t_\kappa}^{\ol t_\kappa + \theta}\dint_\R(\lambda_\kappa^{-1}(D_\mu \phi_\gamma(m_s,s)- \delta H[m_s]))^2 m_s dx ds
+ C(\theta^2 + W_2^2(m_{\bar t_\kappa + \theta},\ol \mu_\kappa)) + \frac{1}{2\lambda_\kappa}\eta_\kappa\theta .
\eal
$$

We show next that $$C_\kappa\delta^2(E(\ol \mu_\kappa) - E(m_{\ol  t_\kappa + \theta}) )^2 = o(\theta).$$
Proceeding as in the viscosity sub-solution case when we established the $L^3-$estimate, 
we begin with the observation   that $L_\theta^2$ is negligible in front of $L_\theta$ because we know that 
$$L_\theta :=  \dint_{\ol t_\kappa }^{\ol t_\kappa + \theta}\dint_\R  (H[m_s])^2m_s dx ds \to 0 \ \ \text{as} \ \ \theta \to 0.$$ 
%we know that $L_\theta^2$ is negligible in front of $L_\theta$.
\vs
Then,  recalling that 
$$
\ba
|E(\ol \mu_\kappa) - E(m_{\ol  t_\kappa + \theta})| &=  \left|\dint_{\ol t_\kappa }^{\ol t_\kappa + \theta}\dint_\R  H[m_s](\alpha_s + H[m_s])m_s dx ds\right|\\
&\leq C\left(L_\theta + \dint_{\ol t_\kappa }^{\ol t_\kappa + \theta}\dint_\R\alpha^2m_sdxds\right)\\
& \leq C \left(L_\theta + \dint_{\ol t_\kappa }^{\ol t_\kappa + \theta}\dint_\R |D_\mu \phi_\gamma(m_s,s)|^2m_sdxds\right)\\
& \leq C(L_\theta + \theta).
\ea
$$
we find that $$C_\kappa\delta^2(E(\ol \mu_\kappa) - E(m_{\ol  t_\kappa + \theta}) )^2 = o(\theta)$$ is also negligible in front of $L_\theta$, and, hence,  an estimate of the form $$L_\theta \leq C \theta$$  still holds in this case. 
\vs
Thus, the $L^3-$estimate holds and  $$C_\kappa\delta^2(E(\ol \mu_\kappa) - E(m_{\ol  t_\kappa + \theta}) )^2 = o(\theta).$$ 
 \vs
Dividing by $\theta$ and taking the limit $\theta \to 0$ yields
$$
\bal
\frac{\delta- \eta_\kappa}{\lambda_\kappa} \dint_\R & (H[\ol \mu_\kappa])^2\ol \mu_\kappa dx \leq  \lambda_\kappa^{-1}\partial_t \phi_\gamma(\ol \mu_\kappa, \ol t_\kappa) + \lambda_\kappa^{-1}\dint_\R D_\mu \phi_\gamma(\ol \mu_\kappa,\ol t_\kappa) H[\ol \mu_\kappa]\ol \mu_\kappa dx+\\
&+(\lambda_\kappa^{-1}\eta_\kappa-\frac12)\dint_\R(\lambda_\kappa^{-1}(D_\mu \phi_\gamma(\ol \mu_\kappa,\ol t_\kappa)- \delta H[\ol \mu_\kappa]))^2 \ol \mu_\kappa dx 
+ \frac{1}{2\lambda_\kappa}\eta_\kappa.
\eal
$$
Multiplying by $\lambda_\kappa^{-1}$ and recalling that $\lambda_\kappa \in (0,1)$, we arrive at
$$
\bal
(\delta- \eta_\kappa)\dint_\R & (H[\ol \mu_\kappa])^2\ol \mu_\kappa dx \leq \partial_t \phi_\gamma(\ol \mu_\kappa, \ol t_\kappa) +\dint_\R D_\mu \phi_\gamma(\ol \mu_\kappa,\ol t_\kappa) H[\ol \mu_\kappa]\ol \mu_\kappa dx+\\
&+(\lambda_\kappa^{-1}\eta_\kappa-\frac12)\dint_\R(D_\mu \phi_\gamma(\ol \mu_\kappa,\ol t_\kappa)- \delta H[\ol \mu_\kappa])^2 \ol \mu_\kappa dx 
+ \frac{1}{2}\eta_\kappa.
\eal
$$
%It now remains to pass to the limit $\kappa \to 0$. 
%First, we could always have chosen $\eta_\kappa$ sufficiently small so that $\frac{\eta_\kappa}{\lambda_\kappa} \to 0$ as $\kappa \to 0$ since $\lambda_\kappa \geq \kappa$. 
We note that,  since $\lambda_\kappa \geq \kappa$, above we can choose $\eta_\kappa$ sufficiently small so that $\frac{\eta_\kappa}{\lambda_\kappa} \to 0$ as $\kappa \to 0$.

\vs

Next, observe that the sequence $(\|\ol \mu_\kappa\|_{L^3})_{\kappa}$ is bounded, and, since $\ol \mu$ is a strongly exposed point of maximum, we know that $W_2(\ol \mu_\kappa ,\ol \mu) \to 0$ as $\kappa \to 0$, because of the local uniform convergence of $G_\kappa(U,A)$ toward $U\wedge A$. 

\vs
To pass to the limit $\kappa\to 0$, it  suffices  to recall that, in view of  Lemma \ref{lemma:Uwas}, we have  
%and, noting $S: \mathbb R \mapsto \mathbb R$ the optimal transport map from $\ol \mu$ to $\xi$, we finally obtain
$$
\bal
\left(\delta+ \frac12\delta^2\right)\dint_\R & (H[\ol \mu])^2\ol \mu dx \leq \partial_t \phi(\ol \mu, \ol t) +( 1+\delta)\dint_\R (D_\mu \phi(\ol \mu,\ol t) + S) H[\ol \mu]\ol \mu dx+\\
&-\frac12\dint_\R(D_\mu \phi(\ol \mu,\ol t) +S)^2 \ol \mu dx, 
\eal
$$
where $S: \mathbb R \mapsto \mathbb R$ the optimal transport map from $\ol \mu$ to $\xi$.
\vs
Then Cauchy-Schwarz inequality and the estimate $\int_\R S^2\bar \mu dx \leq \eta$ yield
$$
\bal
(\delta+ \frac12\delta^2)\dint_\R & (H[\ol \mu])^2\ol \mu dx -(1+\delta)\sqrt\eta\|H[\ol \mu]\|_{L^2(\ol \mu)}\leq \partial_t \phi(\ol \mu, \ol t) +( 1+\delta)\dint_\R D_\mu \phi(\ol \mu,\ol t) H[\ol \mu]\ol \mu dx+\\
&-\frac12\dint_\R(D_\mu \phi(\ol \mu,\ol t))^2 \ol \mu dx + \frac12\eta - \sqrt{\eta}\|D_\mu\phi(\ol \mu,\ol t)\|_{L^2(\ol \mu)},
\eal
$$
which is the claim.

\end{proof}

\subsection{The characterization of the value function as the unique viscosity solution}
Characterizing the value function as the unique viscosity solution to \eqref{takis100} with appropriate boundary condition at time $t = T$ is a trivial consequence of several results above in the the case in which $U_T$ is locally uniformly continuous. Thus we only state the result without proof.
\begin{theorem}\label{thm:charles}
Assume that $F$ and $U_T$ are locally uniformly continuous and that, either $U_T$ is bounded from below by some $\eps W_2^2(\cdot,\nu) - C$ or that $U_T$ is uniformly continuous. Then $U$ is the unique (continuous) viscosity solution to \eqref{takis100} which satisfies $U(\mu,T) = U_T(\mu)$ for all $\mu \in \mcl P_2$.
\end{theorem}

%
%We establish here a uniqueness result for the viscosity solutions to the target problem 
%\beq\label{takis5000000}
%\begin{cases}
%-\pt_t  V +\dfrac{1}{2}\dint_\R |D_m V|^2 dm(y) - \dint_\R D_m V H[m] dm(y)=F \ \ \text{in} \ \ \Pt\times [0.T)\\
%V(\cdot, T)=\mathbb 1_\nu,
%\end{cases}
%\eeq
%where 
%the terminal condition is $\mathbb 1_\nu$, which we recall that equals to $\oo$ away from $\nu$ and is equal to $0$ at $\nu$. 
%\vs

In the case of the target problem, the proof is more involved, since it uses the comparison principle. %that we  detail next.
\begin{theorem}\label{takis100000}
Assume $F$ is locally uniformly continuous. If $V$ is a locally uniformly continuous in $\Pt \times \wtilde \Pt \times [0,T)$ viscosity solution to \eqref{takis100} which, in addition to being bounded from below by some $C^{-1}W_2^2(\cdot,\nu) - C$,  satisfies
\beq\label{takis5000.1}
\bal
&(i)~V(\mu,\nu,t)\geq -C(T-t) \ \ \text{for some}\ \ C=C(\|F\|_\oo)\geq 0\ \ \text{and \ \ all} \ \ (\mu,\nu, t) \in \Pt\times \wtilde \Pt \times [0,T) ,\\
&(ii)~ \underset{t\to T} \lim  V(\nu, \nu, t)=0 \ \text{and}\\
&(iii)~ \underset{t\to T} \lim  V(\mu, \nu, t)=\oo \ \text{locally uniformly in} \ \{\mu \in \Pt: W_2(\mu, \nu)>\rho \} \ \text{for each} \ \rho>0, 
\eal
\eeq
then $V$ is  the value function $\widehat U$ of the target problem for the controlled Dyson equation.
\end{theorem}

\begin{proof} For simplicity in the proof we take $F\equiv 0$ in which case  the first condition in \eqref{takis5000.1} becomes $V(\mu,\nu,t)\geq 0$ for all $(\mu,\nu, t) \in \Pt\times \wtilde \Pt \times [0,T)$.
\vs

Assume that $V$ is a solution  satisfying the hypotheses. Then, for any $h\in (0,T)$, the uniqueness of viscosity solutions with regular terminal data, which is an immediate consequence of the comparison principle, yields, for any $t\in [0,T-h]$, that 
\beq\label{takis10001}
V(\mu,\nu, t)=\underset{(\alpha, m) \in \mcl A(\mu, t; T-h)}\inf \Big[ \dfrac{1}{2} \dint_t^{T-h} \dint_\R \alpha^2_s(x) dm_s(x) + V(m_{T-h},\nu, T-h)\Big],
\eeq
where $\mcl A(\mu, t; T-h)$, and  $\mcl A(\mu,\nu,t; T-h)$ are similar with the sets $\mcl A (\mu,t)$ and $\mcl A(\mu,\nu,t)$, the only difference being the final time $T-h$ instead of $T$. 
\vs
In other words, we are using Theorem \ref{thm:charles} to say that any viscosity solution to \eqref{takis100} satisfying appropriate continuity and growth condition must satisfy the dynamic programming principle as well, even if it is not defined as the value function of the target problem.
\vs

Since clearly $\mcl A(\mu,\nu,t; T-h)\subset \mcl A(\mu, t; T-h)$, it follows from \eqref{takis10001} that 
\beq\label{takis10002}
V(\mu,\nu,t)\leq \what U(\mu,\nu,t; T-h) + V(\nu,\nu,T-h),
\eeq
where $\what U(\mu,\nu,t; T-h)$ is the value of the target control problem in $[0,T-h]$.
\vs

The compactness of the optimal trajectories clearly imply that
\[\underset{h\to 0} \lim \; \what U(\mu,\nu,t; T-h)=\what U(\mu,\nu,t),\]
while the second  condition in \eqref{takis5000.1} yields that 
\[\underset{h\to 0} \lim \; V(\nu,\nu,t-h)=0.\]

\vs

Combining the last two facts with \eqref{takis10002} yields
\beq\label{takis10002}
V(\mu,\nu,t)\leq \what U(\mu,\nu,t).
\eeq
\vs

For the other inequality needed to prove the claim, we fix $\ep>0$ and choose an $\ep-$optimal pair $(\alpha, m) \in \mcl A(\mu, t; T-h)$ for the equality in \eqref{takis10001}. Note that to keep the notation simple, we do not write the dependence of the $\ep-$optimal pair on $\ep$ and $h$. 
\vs
It follows that 
\beq\label{takis10003}
V(\mu,\nu,t) +\ep \geq  \dfrac{1}{2} \dint_t^{T-h} \dint_\R \alpha^2_s (x) dm_s(x) + V(m_{T-h},\nu, T-h).
\eeq

\vs

Since the left hand side of \eqref{takis10003} is bounded,  the third condition in \eqref{takis5000.1} implies that 
\beq\label{takis10004}
\underset{h\to 0} \lim W_2(m_{T-h},\nu)=0.
\eeq
In view of the compactness of the optimal trajectories, it follows that, as $h\to 0$, the $\ep-$optimal pairs  $(\alpha, m)$ converge to a pair $(\ol \alpha, m)\in \mcl A(\mu,\nu,t)$ which together with the first condition in \eqref{takis10001} give that 
\[V(\mu,\nu,t) +\ep \geq \what U(\mu,\nu,t) + \underset{h\to 0} \liminf \; V(m_{T-h},\nu, T-h) \geq \what U(\mu,\nu,t), \]
\vs
and, after $\ep\to 0$, the reverse inequality of \eqref{takis10002},  and, hence, the claim.

\end{proof}

\subsection{The stability of viscosity solutions and an approximation of the target problem with continuous terminal data} 
Using a penalization argument, we construct here   a family of solutions to \eqref{takis100} which increase to the value function $\what U(\cdot,\cdot;\nu)$ of the target problem.
\vs

For each  $\eps > 0$,  we define $U_\eps:\Pt\times[0,T]\to \R$ by
\beq\label{takis460}
U_\eps(\mu, t; \nu) = \inf_{(\alpha,m) \in \mathcal{A}(\mu,t)}\left\{ \dint_t^T \Big[\dfrac12\dint_\R (\alpha_s(x))^2dm_s(x) + F[m_s] \Big] ds + \dfrac{1}{2\eps}W_2^2(m_T,\nu)\right\},
\eeq
which, in view of discussion in  the previous subsection, is the unique continuous solution to \eqref{takis100} with terminal datum $U_T=\dfrac{1}{2\ep} W_2^2(\cdot,\nu)$ which  satisfies  an  appropriate bound from below.

\begin{proposition}\label{utou}
Assume that $\nu \in \wtilde \Pt$. Then, as $\eps \to 0$, the family $(U_\eps)_{\eps > 0}$, defined by \eqref{takis460}, converges to  the value function of the target problem $\what U(\cdot,\nu,t)$.
\end{proposition}

\begin{proof}
For any $\eps < \eps'$, $U_\eps' (\cdot,\nu,\cdot)\leq U_\eps (\cdot,\nu,\cdot)\leq \what U(\cdot, \nu,\cdot)$. It is then immediate from the compactness of the optimal paths, Proposition~\ref{prop:comp} and Proposition~\ref{takis1250} that, for each  $(\mu,t)\in \Pt\times [0,T)$, 
\[\sup_{\eps > 0} \;U_\eps(\mu, t;\nu) =\underset{\ep\to 0} \lim \; U_\ep(\mu, t; \nu)= \what U(\mu,t ; \nu).\]
\end{proof}

%We note that  the uniform $W_2-$continuity estimates that follow from Theorem \ref{thm:estlip} yield uniform continuity 
%We note that, because of the uniform $W_2-$continuity estimates that follow from Theorem \ref{thm:estlip}, we could have imagined that, using standard techniques from the theory of  viscosity solutions, we should be able to obtain a local convergence uniform through passing to the limit in the PDE itself. This is, however, 
%%We believe that it is important to insist upon the fact that the previous is 
%not clear at all, since uniform modulus of continuity yields uniform convergence on compact sets, which are quite different than bounded sets, and that the latter are the one on which information is required if we want to use classical viscosity solution techniques.

%\section*{Acknowledgments}
%The first author acknowledge a partial support from the chair FDD (Institut Louis Bachelier), the Lagrange Mathematics and Computing Research Center, and the ERC under the project PaDiESeM. \color{red} To do\color{black}
%
 
\bibliographystyle{plainnat}
\bibliography{bibDyson}

\begin{thebibliography}{35}
\providecommand{\natexlab}[1]{#1}
\providecommand{\url}[1]{\texttt{#1}}
\expandafter\ifx\csname urlstyle\endcsname\relax
  \providecommand{\doi}[1]{doi: #1}\else
  \providecommand{\doi}{doi: \begingroup \urlstyle{rm}\Url}\fi

\bibitem[Alfonsi and Jourdain(2020)]{AJ}
Aur{\'e}lien Alfonsi and Benjamin Jourdain.
\newblock Squared quadratic {Wasserstein distance: optimal couplings and L}ions
  differentiability.
\newblock \emph{ESAIM: Probability and Statistics}, 24:\penalty0 703--717,
  2020.

\bibitem[Alvarez(1994)]{alvarez}
Olivier Alvarez.
\newblock \emph{Contributions {\`a} la th{\'e}orie des solutions de
  viscosit{\'e} des {\'e}quations de Hamilton-Jacobi-Bellman et applications
  {\`a} la finance}.
\newblock PhD thesis, Paris 9, 1994.

\bibitem[Ambrosio et~al.(2005)Ambrosio, Gigli, and Savar{\'e}]{ags}
Luigi Ambrosio, Nicola Gigli, and Giuseppe Savar{\'e}.
\newblock \emph{Gradient flows in metric spaces and in the space of probability
  measures}.
\newblock Springer, 2005.

\bibitem[Anderson et~al.(2010)Anderson, Guionnet, and Zeitouni]{AGZ}
Greg~W Anderson, Alice Guionnet, and Ofer Zeitouni.
\newblock \emph{An introduction to random matrices}.
\newblock Number 118. Cambridge university press, 2010.

\bibitem[Arous and Guionnet(1997)]{BG}
G\'{e}rard~Ben Arous and Alice Guionnet.
\newblock Large deviations for {Wigner's law and V}oiculescu's non-commutative
  entropy.
\newblock \emph{Probability theory and related fields}, 108\penalty0
  (4):\penalty0 517--542, 1997.

\bibitem[Bayraktar et~al.(2025)Bayraktar, Ekren, and Zhang]{BEZ}
Erhan Bayraktar, Ibrahim Ekren, and Xin Zhang.
\newblock Comparison of viscosity solutions for a class of second-order {PDE}s
  on the {W}asserstein space.
\newblock \emph{Communications in Partial Differential Equations}, 50\penalty0
  (4):\penalty0 570--613, 2025.

\bibitem[Ben~Arous and Zeitouni(1998)]{BZ}
G{\'e}rard Ben~Arous and Ofer Zeitouni.
\newblock Large deviations from the circular law.
\newblock \emph{ESAIM: Probability and Statistics}, 2:\penalty0 123--134, 1998.

\bibitem[Benamou and Brenier(2000)]{BB}
Jean-David Benamou and Yann Brenier.
\newblock A computational fluid mechanics solution to the {Monge-Kantorovich}
  mass transfer problem.
\newblock \emph{Numerische Mathematik}, 84\penalty0 (3):\penalty0 375--393,
  2000.

\bibitem[Bertucci(2024)]{B}
Charles Bertucci.
\newblock {Stochastic optimal transport and Hamilton--Jacobi--Bellman equations
  on the set of probability measures}.
\newblock \emph{Annales de l'Institut Henri Poincar{\'e} C}, 42\penalty0
  (6):\penalty0 1543--1600, 2024.

\bibitem[Bertucci(2025)]{tangent}
Charles Bertucci.
\newblock The tangent space to the {W}asserstein space: parallel transport and
  other applications.
\newblock \emph{arXiv preprint arXiv:2512.09763}, 2025.

\bibitem[Bertucci(2026)]{Bbook}
Charles Bertucci.
\newblock \emph{Analysis on spaces of measures: {Mean field game master
  equations and first order mean field Hamilton-Jacobi equations}}.
\newblock Forthcoming, 2026.

\bibitem[Bertucci and Lions(2024)]{BL}
Charles Bertucci and Pierre~Louis Lions.
\newblock {An approximation of the squared Wasserstein distance and an
  application to Hamilton-Jacobi equations}.
\newblock \emph{arXiv preprint arXiv:2409.11793}, 2024.

\bibitem[Bertucci and Lions(2026)]{BL2}
Charles Bertucci and Pierre-Louis Lions.
\newblock {Non-linear Stegall's lemma and general Hamilton-Jacobi-Bellman
  equations on Wasserstein spaces}, 2026.

\bibitem[Bertucci et~al.(2022)Bertucci, Debbah, Lasry, and Lions]{BDLL}
Charles Bertucci, M{\'e}rouane Debbah, Jean-Michel Lasry, and Pierre-Louis
  Lions.
\newblock A spectral dominance approach to large random matrices.
\newblock \emph{Journal de Math{\'e}matiques Pures et Appliqu{\'e}es},
  164:\penalty0 27--56, 2022.

\bibitem[Bertucci et~al.(2024)Bertucci, Lasry, and Lions]{BLL}
Charles Bertucci, Jean-Michel Lasry, and Pierre-Louis Lions.
\newblock A spectral dominance approach to large random matrices: part ii.
\newblock \emph{Journal de Math{\'e}matiques Pures et Appliqu{\'e}es},
  192:\penalty0 103630, 2024.

\bibitem[Bertucci et~al.(forthcoming)Bertucci, Lions, and Souganidis]{BLS}
Charles Bertucci, Pierre-Louis Lions, and Panagiotis~E Souganidis.
\newblock {Large deviations of the Dyson Brownian motion}.
\newblock forthcoming.

\bibitem[Biane and Voiculescu(2001)]{BV}
Philippe Biane and Dan Voiculescu.
\newblock {A free probability analogue of the Wasserstein metric on the
  trace-state space}.
\newblock \emph{Geometric \& Functional Analysis GAFA}, 11\penalty0
  (6):\penalty0 1125--1138, 2001.

\bibitem[Brezis(1973)]{brezis}
Haim Brezis.
\newblock \emph{Op\'erateurs maximaux monotones et semi-groupes de contractions
  dans les espaces de {H}ilbert}, volume~5.
\newblock Elsevier, 1973.

\bibitem[Cardaliaguet and Souganidis(2023)]{CS}
Pierre Cardaliaguet and Panagiotis~E Souganidis.
\newblock Regularity of the value function and quantitative propagation of
  chaos for mean field control problems.
\newblock \emph{Nonlinear Differential Equations and Applications NoDEA},
  30\penalty0 (2):\penalty0 25, 2023.

\bibitem[Cecchin and Delarue(2025)]{CD}
Alekos Cecchin and Fran{\c{c}}ois Delarue.
\newblock \emph{Weak solutions to the master equation of potential mean field
  games}, volume 315.
\newblock American Mathematical Society, 2025.

\bibitem[C{\'e}pa and L{\'e}pingle(1997)]{CL}
Emmanuel C{\'e}pa and Dominique L{\'e}pingle.
\newblock Diffusing particles with electrostatic repulsion.
\newblock \emph{Probability theory and related fields}, 107\penalty0
  (4):\penalty0 429--449, 1997.

\bibitem[Chan(1992)]{chan}
Terence Chan.
\newblock The {W}igner semi-circle law and eigenvalues of matrix-valued
  diffusions.
\newblock \emph{Probability theory and related fields}, 93\penalty0
  (2):\penalty0 249--272, 1992.

\bibitem[Cosso et~al.(2024)Cosso, Gozzi, Kharroubi, Pham, and
  Rosestolato]{cosso}
Andrea Cosso, Fausto Gozzi, Idris Kharroubi, Huy{\^e}n Pham, and Mauro
  Rosestolato.
\newblock {Master Bellman equation in the Wasserstein space: Uniqueness of
  viscosity solutions}.
\newblock \emph{Transactions of the American Mathematical Society},
  377\penalty0 (01):\penalty0 31--83, 2024.

\bibitem[Daudin and Seeger(2024)]{DS}
Samuel Daudin and Benjamin Seeger.
\newblock {A comparison principle for semilinear Hamilton--Jacobi--Bellman
  equations in the Wasserstein space}.
\newblock \emph{Calculus of Variations and Partial Differential Equations},
  63\penalty0 (4):\penalty0 106, 2024.

\bibitem[Daudin et~al.(2025)Daudin, Jackson, and Seeger]{DJS}
Samuel Daudin, Joe Jackson, and Benjamin Seeger.
\newblock {Well-posedness of Hamilton-Jacobi equations in the Wasserstein
  space: non-convex Hamiltonians and common noise}.
\newblock \emph{Communications in Partial Differential Equations}, 50\penalty0
  (1-2):\penalty0 1--52, 2025.

\bibitem[Dembo et~al.(2003)Dembo, Guionnet, and Zeitouni]{DGZ}
Amir Dembo, A~Guionnet, and Ofer Zeitouni.
\newblock Moderate deviations for the spectral measure of certain random
  matrices.
\newblock In \emph{Annales de l'Institut Henri Poincare (B) Probability and
  Statistics}, volume~39, pages 1013--1042. Elsevier, 2003.

\bibitem[Gangbo and Tudorascu(2019)]{GT}
Wilfrid Gangbo and Adrian Tudorascu.
\newblock On differentiability in the {Wasserstein space} and well-posedness
  for {Hamilton-Jacobi} equations.
\newblock \emph{Journal de Math{\'e}matiques Pures et Appliqu{\'e}es},
  125:\penalty0 119--174, 2019.

\bibitem[Guionnet and Zeitouni(2000)]{GZ1}
Alice Guionnet and Ofer Zeitouni.
\newblock Concentration of the spectral measure for large matrices.
\newblock \emph{Electronic Communications in Probability [electronic only]},
  5:\penalty0 119--136, 2000.

\bibitem[Guionnet and Zeitouni(2002)]{GZ2}
Alice Guionnet and Ofer Zeitouni.
\newblock Large deviations asymptotics for spherical integrals.
\newblock \emph{Journal of functional analysis}, 188\penalty0 (2):\penalty0
  461--515, 2002.

\bibitem[Ishii(1984)]{Ishii}
Hitoshi Ishii.
\newblock Uniqueness of unbounded viscosity solution of {Hamilton-Jacobi}
  equations.
\newblock \emph{Indiana University Mathematics Journal}, 33\penalty0
  (5):\penalty0 721--748, 1984.

\bibitem[Lions(2009-2010)]{Lionscours}
Pierre-Louis Lions.
\newblock Cours au college de france.
\newblock \emph{Available at www. college-de-france. fr}, 2009-2010.

\bibitem[Rogers and Shi(1993)]{RC}
Leonard~CG Rogers and Zhan Shi.
\newblock Interacting {B}rownian particles and the {W}igner law.
\newblock \emph{Probability theory and related fields}, 95\penalty0
  (4):\penalty0 555--570, 1993.

\bibitem[Santambrogio(2015)]{Santambrogio}
Filippo Santambrogio.
\newblock Optimal transport for applied mathematicians.
\newblock \emph{Birk{\"a}user, NY}, 55\penalty0 (58-63):\penalty0 94, 2015.

\bibitem[Soner and Yan(2024)]{SY}
H~Mete Soner and Qinxin Yan.
\newblock {Viscosity solutions for McKean--Vlasov control on a torus}.
\newblock \emph{SIAM Journal on Control and Optimization}, 62\penalty0
  (2):\penalty0 903--923, 2024.

\bibitem[Stegall(1978)]{stegall}
Charles Stegall.
\newblock Optimization of functions on certain subsets of {B}anach spaces.
\newblock \emph{Mathematische Annalen}, 236\penalty0 (2):\penalty0 171--176,
  1978.

\end{thebibliography}

\end{document}